\documentclass[oneside,english]{amsart}
\usepackage[T1]{fontenc}
\usepackage[latin9]{inputenc}
\usepackage{mathtools}
\usepackage{amstext}
\usepackage{amsthm}
\usepackage{amssymb}
\usepackage{stackrel}
\usepackage[all]{xy}

\makeatletter
\numberwithin{equation}{section}
\numberwithin{figure}{section}

\usepackage{amsmath}

\usepackage{amssymb}
\usepackage{stmaryrd}
\usepackage{amsthm}
\usepackage[mathscr]{eucal}
\usepackage{amscd}
\usepackage[all]{xy}
\usepackage{url}
\usepackage{enumitem}

\usepackage{mathtools}

\newtheorem{Thm}{Theorem}[section]
\usepackage{aliascnt}

\usepackage{comment}

\usepackage{hyperref}
\hypersetup{%
pdftitle={Preprint},
pdfauthor={},
pdfkeywords={},
bookmarksnumbered,
pdfstartview={FitH},
breaklinks=true,
colorlinks=true,
urlcolor=blue,
citecolor=blue,
}%
\usepackage[all]{hypcap} 

\usepackage{breakurl}

\usepackage{color}

\theoremstyle{plain}
\newtheorem{Lem}[Thm]{Lemma}
\theoremstyle{plain}
\newtheorem{Prop}[Thm]{Proposition}
\theoremstyle{plain}
\newtheorem{Cor}[Thm]{Corollary}
\theoremstyle{plain}
\newtheorem{Conj}[Thm]{Conjecture}
\theoremstyle{plain}
\newtheorem*{Thm*}{Theorem}
\theoremstyle{plain}
\newtheorem*{Conj*}{Conjecture}
\theoremstyle{plain}

\theoremstyle{plain}
\newtheorem*{Prob*}{Problem}
\theoremstyle{plain}

\theoremstyle{plain}

\theoremstyle{definition}
\newtheorem{Def}[Thm]{Definition}
\theoremstyle{definition}
\newtheorem*{Def*}{Definition}
\theoremstyle{definition}
\newtheorem{Eg}[Thm]{Example}
\theoremstyle{definition}

\theoremstyle{definition}
\newtheorem{Assumption}{Assumption}
\theoremstyle{definition}

\theoremstyle{definition}

\theoremstyle{definition}

\theoremstyle{remark}
\newtheorem{Rem}[Thm]{Remark}

\newtheoremstyle{custom}
  {}
  {}
  {\itshape}
  {}
  {\bfseries}
  {}
  { }
  {Condition \thmnote{#3}}

\theoremstyle{custom}
\newtheorem*{Cond*}{}

\newcommand{\kk}{\Bbbk}
\newcommand{\Z}{\mathbb{Z}}
\newcommand{\N}{\mathbb{N}}
\newcommand{\Q}{\mathbb{Q}}
\newcommand{\C}{\mathbb{C}}

\newcommand{\F}{\mathbb{F}}

\renewcommand{\hat}[1]{\widehat{#1}}
\renewcommand{\tilde}[1]{\widetilde{#1}}

\newcommand{\opname}[1]{\operatorname{\mathsf{#1}}}

\newcommand{\inj}{\opname{inj}}

\newcommand{\op}{{^{\opname{op}}}}

\newcommand{\sign}{\opname{sign}}

\newcommand{\ext}{\opname{ext}}

\newcommand{\supp}{\opname{supp}}

\renewcommand{\deg}{\opname{deg}}

\newcommand{\K}{\operatorname{\mathbb{K}}\nolimits}

\newcommand{\Hf}{{\frac{1}{2}}}

\newcommand{\Rm}[1]{{\longmapsto}}
\newcommand{\Lm}[1]{{\longmapsfrom}}

\newcommand{\cA}{{\mathcal A}}

\newcommand{\cC}{{\mathcal C}}

\newcommand{\cF}{{\mathcal F}}

\newcommand{\cN}{{\mathcal N}}

\newcommand{\cS}{{\mathcal S}}
\newcommand{\cT}{{\mathcal T}}

\newcommand{\cZ}{{\mathcal Z}}

\newcommand{\bL}{{\mathbf L}}

\newcommand{\bi}{{\mathbf i}}
\newcommand{\bj}{{\mathbf j}}

\newcommand{\uB}{{\underline{B}}}

\newcommand{\uc}{{\underline{c}}}

\newcommand{\ui}{{\underline{i}}}

\newcommand{\uk}{{\underline{k}}}

\newcommand{\tB}{{\widetilde{B}}}
\newcommand{\tC}{{\widetilde{C}}}

\newcommand{\tE}{{\widetilde{E}}}

\newcommand{\tP}{{\widetilde{P}}}

\newcommand{\stdMod}{\mbf{M}}
\newcommand{\can}{L}
\newcommand{\clAlg}{{\cA}}
\newcommand{\qClAlg}{\cA_q}

\newcommand{\diag}{{\delta}}

\newcommand{\rev}{\opname{rev}}

\newcommand{\res}{{\mathrm{Res}}}

\newcommand{\mbf}[1]{{\mathbf{#1}}}

\usepackage{tikz}
\usetikzlibrary{positioning,shapes,shadows,arrows,snakes}
\usetikzlibrary {arrows.meta,positioning}

\pgfdeclarelayer{edgelayer}
\pgfdeclarelayer{nodelayer}
\pgfsetlayers{edgelayer,nodelayer,main}

\tikzstyle{none}=[inner sep=0pt]
\tikzstyle{black box}=[draw=black, fill=black!25]
\tikzstyle{white box}=[draw=black, fill=white]
\tikzstyle{black circle}=[circle,draw=black!50, fill=black!25]
\tikzstyle{red circle}=[circle,draw=red!50, fill=red!25]
\tikzstyle{blue circle}=[circle,draw=blue!50, fill=blue!25]
\tikzstyle{green circle}=[circle,draw=green!50, fill=green!25]
\tikzstyle{yellow circle}=[circle,draw=yellow!50, fill=yellow!25]

\newcommand{\thistheoremname}{}
\newtheorem*{genericthm*}{\thistheoremname}
\newenvironment{namedthm*}[1]
  {\renewcommand{\thistheoremname}{#1}%
   \begin{genericthm*}}
  {\end{genericthm*}}

\renewcommand{\diag}{{d}}

\renewcommand{\inj}{{\bI}}
\renewcommand{\can}{{\bL}}

\newcommand{\fv}{\opname{f}}
\newcommand{\ufv}{\opname{uf}}

\usepackage{geometry}

\ifdefined\showcaptionsetup
 \PassOptionsToPackage{caption=false}{subfig}
\fi
\usepackage{subfig}
\makeatother

\usepackage{babel}

\subjclass[2020]{Primary 13F60; Secondary 17B37, 17B10, 18N25.}
\begin{document}

\newcommand{\hbeta}{\hat{\beta}}
\newcommand{\utB}{\underline{\tilde{B}}}

\newcommand{\tz}{\tilde{z}}
\newcommand{\W}{\mathbb{W}}
\newcommand{\V}{\mathbb{V}}
\newcommand{\txi}{\tilde{\xi}}

\newcommand{\GHL}{\mathrm{GHL}}

\newcommand{\tTheta}{\widetilde{\Theta}}

\newcommand{\Kq}{\mathfrak{K}_{q}}
\newcommand{\hAq}{\hat{\mathcal{A}}_{q}(\mathfrak{n})}

\newcommand{\ee}{\mathbf e}
\newcommand{\ff}{\mathbf f}

\newcommand{\unitObj}{\mathbf 1}
\newcommand{\wtSp}{\opname Q}
\renewcommand{\op}{^{\opname{op}}}

\newcommand{\simpObj}{\opname{Simp}}

\newcommand{\Span}{\opname{Span}}

\renewcommand{\res}{\opname{res}}
\newcommand{\frR}{\mathfrak{R}}
\newcommand{\frI}{\mathfrak{I}}
\newcommand{\frz}{\mathfrak{f}}
\renewcommand{\diag}{\delta}
\newcommand{\symm}{\mathsf{D}}
\newcommand{\ubi}{\underline{\bi}}
\newcommand{\ubj}{\underline{\bj}}

\newcommand{\Br}{\mathsf{Br}}
\newcommand{\cB}{\mathcal{B}}
\newcommand{\xar}[1]{\xymatrix{\ar[r]^{#1}&}}
\newcommand{\xline}[1]{\xymatrix{\ar@{-}[r]^{#1}&}}
\newcommand{\ow}{\overrightarrow{w}}
\newcommand{\Conf}{\mathrm{Conf}}
\newcommand{\leftweave}{\overleftarrow{\mathfrak{m}}}

\newcommand{\ueta}{\underline{\eta}}
\newcommand{\uzeta}{\underline{\zeta}}
\newcommand{\uxi}{\underline{\xi}}

\newcommand{\qO}{\mathcal{O}_q}

\newcommand{\up}{\opname{up}}
\newcommand{\dCan}{\opname{B}^{*}}
\newcommand{\cdCan}{\opname{\mathring{B}}^{*}}
\newcommand{\hdCan}{\hat{\opname{B}^{*}}}

\renewcommand{\qClAlg}{{\clAlg_q}}

\newcommand{\cone}{{M^\circ}}
\newcommand{\uCone}{\underline{\cone}}
\newcommand{\yCone}{N_{\ufv}}
\newcommand{\tropSet}{{\mathcal{M}^\circ}}
\newcommand{\domTropSet}{{\overline{\mathcal{M}}^\circ}}

\newcommand{\sol}{\mathrm{TI}}
\newcommand{\intv}{{\mathrm{BI}}}

\newcommand{\Perm}{\mathrm{P}}

\newcommand{\tui}{\widetilde{\ui}}
\newcommand{\tuk}{\widetilde{\uk}}

\newcommand{\bideg}{\opname{bideg}}

\newcommand{\ubeta}{\underline{\beta}}
\newcommand{\udelta}{\underline{\delta}}
\newcommand{\ugamma}{\underline{\gamma}}
\newcommand{\ualpha}{\underline{\alpha}}

\newcommand{\LP}{{\mathcal{LP}}}
\newcommand{\hLP}{{\widehat{\mathcal{LP}}}}

\newcommand{\bLP}{{\overline{\mathcal{LP}}}}
\newcommand{\bClAlg}{{\overline{\clAlg}}}
\newcommand{\bUpClAlg}{{\overline{\upClAlg}}}
\newcommand{\base}{\mathbb{B}}
\newcommand{\alg}{\mathbf{A}}
\newcommand{\algfr}{\alg_{\mathrm{f}}}

\newcommand{\midClAlg}{{\clAlg^{\mathrm{mid}}}}
\newcommand{\upClAlg}{\mathcal{U}}

\newcommand{\bQClAlg}{{\bClAlg_q}}

\newcommand{\qUpClAlg}{{\upClAlg_q}}

\newcommand{\bQUpClAlg}{{\bUpClAlg_q}}

\newcommand{\canClAlg}{{\clAlg^{\mathrm{can}}}}

\newcommand{\AVar}{\mathscr{V}}
\newcommand{\XVar}{\mathbb{X}}

\newcommand{\Jac}{\hat{\mathop{J}}}

\newcommand{\wt}{\mathrm{wt}}
\newcommand{\cl}{\mathrm{cl}}

\newcommand{\domCone}{{\overline{M}^\circ}}

\newcommand{\tree}{{\mathbb{T}}}

\newcommand{\img}{{\mathrm{Im}}}

\newcommand{\Id}{{\mathrm{Id}}}
\newcommand{\prin}{{\mathrm{prin}}}

\newcommand{\mm}{{\mathbf{m}}}

\newcommand{\cPtSet}{{\mathcal{CPT}}}
\newcommand{\bPtSet}{{\mathcal{BPT}}}
\newcommand{\tCPtSet}{{\widetilde{\mathcal{CPT}}}}

\newcommand{\frn}{{\mathfrak{n}}}
\newcommand{\frsl}{{\mathfrak{sl}}}

\newcommand{\frg}{\mathfrak{g}}
\newcommand{\hfrg}{\hat{\mathfrak{g}}}
\newcommand{\hfrh}{\hat{\mathfrak{h}}}

\newcommand{\frh}{\mathfrak{h}}
\newcommand{\frp}{\mathfrak{p}}
\newcommand{\frd}{\mathfrak{d}}
\newcommand{\frj}{\mathfrak{j}}
\newcommand{\frD}{\mathfrak{D}}
\newcommand{\frS}{\mathfrak{S}}
\newcommand{\frC}{\mathfrak{C}}
\newcommand{\frM}{\mathfrak{M}}

\newcommand{\Int}{\mathrm{Int}}
\newcommand{\ess}{\mathrm{ess}}
\newcommand{\Mono}{\mathrm{Mono}}

\newcommand{\bfm}{{\mathbf{m}}}
\newcommand{\bfI}{{\mathbf{I}}}

\newcommand{\Pot}{\mathrm{Pot}}
\newcommand{\kGp}{\opname{K}}

\newcommand{\s}{\mathrm{s}}
\newcommand{\fd}{\mathrm{fd}}

\renewcommand{\sc}{\mathrm{sc}}
\newcommand{\Hall}{\mathrm{Hall}}
\newcommand{\income}{\mathrm{in}}

\newcommand{\tn}{\tilde{n}}

\newcommand{\sing}{\opname{sing}}

\renewcommand{\inj}{\opname{inj}}

\renewcommand{\ext}{\mathrm{ext}}

\newcommand{\cJac}{\Jac}

\newcommand{\stilt}{\mathrm{s}\tau\mathrm{-tilt}}
\newcommand{\Fac}{\mathrm{Fac}}
\newcommand{\Sub}{\mathrm{Sub}}

\newcommand{\rigid}{\mathrm{rigid}}
\newcommand{\tauRigid}{\tau\mathrm{-rigid}}
\newcommand{\spTilt}{\mathrm{s}\tau\mathrm{-tilt}}
\newcommand{\clTilt}{\mathrm{c-tilt}}
\newcommand{\maxRigid}{\mathrm{m-rigid}}

\newcommand{\Li}{\mathrm{Li}}

\newcommand{\Trop}{\opname{Trop}}

\newcommand{\seq}{\boldsymbol{\mu}}

\newcommand{\bseq}{\mu_{\bullet}}

\newcommand{\val}{\mathbf{v}}
\newcommand{\hookuparrow}{\mathrel{\rotatebox[origin=c]{90}{$\hookrightarrow$}}} 
\newcommand{\hookdownarrow}{\mathrel{\rotatebox[origin=c]{-90}{$\hookrightarrow$}}}
\newcommand{\twoheaddownarrow}{\mathrel{\rotatebox[origin=c]{-90}{$\twoheadrightarrow$}}}

\newcommand{\rd}{{\opname{red}}} 
\newcommand{\bCan}{\overline{\can}}

\newcommand{\im}{\opname{Im}} 

\newcommand{\lex}{\opname{lex}} 

\newcommand{\vu}{\mathbf{u}}

\newcommand{\dsd}{\ddot{\sd}} 
\newcommand{\rsd}{\dot{\sd}}
\newcommand{\sd}{{\bf t}}
\newcommand{\ssd}{{\bf s}}
\newcommand{\ddB}{\ddot{B}}
\newcommand{\ddI}{\ddot{I}}
\newcommand{\dB}{\dot{B}}
\newcommand{\ddH}{\ddot{H}}
\newcommand{\ddLambda}{\ddot{\Lambda}}

\newcommand{\ddbi}{\ddot{\bi}}
\newcommand{\uddbi}{\underline{\ddbi}}

\newcommand{\nsd}{{\bf r}}
\newcommand{\usd}{{\bf u}}

\newcommand{\bti}{{\mathbf{\tilde{i}}}}
\newcommand{\ubti}{{\underline{\bti}}}

\newcommand{\col}{\opname{col}}

\newcommand{\frRing}{\mathcal{R}}
\newcommand{\bFrRing}{\overline{\frRing}}
\newcommand{\frGroup}{\mathcal{P}}
\newcommand{\frMonoid}{\overline{\frGroup}}

\newcommand{\uBase}{\underline{\base}}
\newcommand{\bBase}{\overline{\base}}

\newcommand{\bUBase}{\overline{\uBase}}

\newcommand{\ddBS}{\ddot{X}}
\newcommand{\dBS}{\dot{X}}

\newcommand{\udim}{\mathrm{dim}}

\newcommand{\ev}{\mathrm{ev}}
\newcommand{\CC}{\mathrm{CC}}

\newcommand{\envAlg}{\mathsf{U}_q}
\newcommand{\qAff}{\mathsf{U}_{\varepsilon}(\hat{\mathfrak{g}})}

\newcommand{\BZ}{\mathrm{BZ}}

\newcommand{\HL}{\mathrm{HL}}
\DeclarePairedDelimiter\floor{\lfloor}{\rfloor}
\newcommand{\simeqd}{\mathrel{\rotatebox[origin=c]{-90}{$\xrightarrow{\sim}$}}}
\newcommand{\simequ}{\mathrel{\rotatebox[origin=c]{90}{$\xrightarrow{\sim}$}}}

\newcommand{\ucN}{\underline{\mathcal{N}}}
\newcommand{\sqbinom}[2]{\genfrac{[}{]}{0pt}{}{#1}{#2}}
\title{Based cluster algebras of infinite rank}
\author{Fan Qin}
\address{School of Mathematical Sciences, Beijing Normal University, No. 19, XinJieKouWai St., Beijing 100875, China}
\email{qin.fan.math@gmail.com}
\begin{abstract}
We introduce a systematic extension method for based (upper) cluster
algebras that allows the passage from finite rank to infinite rank,
under which fundamental structures and properties extend naturally.
For example, by extending (quantum) cluster algebras whose initial
seeds are associated with signed words (arising from double Bott--Samelson
cells), we recover infinite rank cluster algebras arising from representations
of (shifted) quantum affine algebras. 

As the first main application, we show that the fundamental variables
of the cluster algebras arising from double Bott--Samelson cells
can be computed via a braid group action when the Cartan matrix is
of finite type. We also obtain the equality $A=U$ for the associated
infinite rank (quantum) cluster algebras. Additionally, several conjectures
regarding quantum virtual Grothendieck rings due to Jang--Lee--Oh
\cite{jang2023quantization} and Oh--Park \cite{oh2024pbw} follow
as consequences. As the second main application, we show that the
cluster algebras arising from representations of shifted quantum affine
algebras, discovered by Geiss--Hernandez--Leclerc \cite{geiss2024representations},
admit natural quantizations.
\end{abstract}

\maketitle
\tableofcontents{}

\section{Introduction}

\subsection{Background}

Cluster algebras were introduced by Fomin and Zelevinsky in order
to study total positivity \cite{Lusztig96} and the dual canonical
bases of quantum groups \cite{Lusztig90,Lusztig91}\cite{Kas:crystal}.
They admit natural quantizations due to Berenstein and Zelevinsky
\cite{BerensteinZelevinsky05}. They include the (quantized) coordinate
rings of many finite-dimensional varieties arising from Lie theory.
Likewise, most previous literature focuses on cluster algebras of
finite rank.

In the previous work \cite{qin2023analogs}, the author introduced
based cluster algebras. These are cluster algebras endowed with well-behaved
bases. The author further showed that many cluster algebras arising
from Lie theory possess the common triangular bases \cite{qin2017triangular},
which are analogues of the dual canonical bases (see \cite[Section 7]{qin2023analogs}).
Moreover, for cluster algebras arising from double Bott--Samelson
cells, we further obtain their \emph{standard bases}, whose elements
are ordered products of distinguished cluster algebra elements called
the\emph{ fundamental variables} \cite[Section 8.3]{qin2023analogs}.
Recall that the (dual) PBW bases of quantum groups consist of ordered
products of the \emph{root vectors} (\cite{Lus:intro}). These standard
bases are analogous to the dual PBW bases, and the fundamental variables
play a role similar to the root vectors in the construction of bases. 

On the other hand, it is well known that the root vectors are obtained
by a braid group action on the quantum group \cite{Lus:intro}. This
naturally leads to the following question:

\emph{Can one compute the fundamental variables in cluster algebras
arising from double Bott--Samelson cells via a suitable braid group
action?}

Fortunately, braid group actions have recently been constructed on
certain cluster algebras of infinite rank by Jang, Lee and Oh \cite{jang2023braid}.
So if we consider the infinite rank setting, we already have nice
braid group actions to work with.

In recent years, there has been growing interest in cluster algebras
of infinite rank. Examples include Grothendieck rings of monoidal
categories by Hernandez and Leclerc \cite{hernandez2013cluster},
categorification of cluster algebras by Kashiwara, Kim, Oh, and Park
\cite{kashiwara2021monoidal}, general infinite rank cluster algebras
by Gratz \cite{gratz2015cluster}, braid group actions on (virtual)
Grothendieck rings by Jang, Lee, and Oh \cite{jang2023quantization}
and, more recently, cluster algebras arising from representations
of shifted quantum affine algebras discovered by Geiss, Hernandez,
and Leclerc \cite{geiss2024representations}. It is very interesting
that the cluster algebras in \cite{geiss2024representations} encode
infinite-dimensional representations, which are usually not encoded
in finite rank cluster algebras \cite{HernandezLeclerc09}. In addition,
the shifted quantum affine algebras are intriguing objects arising
from quantized K-theoretic Coulomb branches of 3d N = 4 SUSY quiver
gauge theories \cite{finkelberg2019multiplicative}, providing further
motivations to understand the corresponding cluster structures.

The aim of this paper is twofold. First, we introduce a systematic
extension method for based (upper) cluster algebras developed in \cite{qin2023analogs}
that allows the passage from finite rank to infinite rank. Second,
we apply this extension method to interesting cluster algebras of
infinite rank in the examples above. In particular, we give an affirmative
answer to the natural question above by computing the fundamental
variables by means of the braid group actions (introduced by Jang--Lee--Oh
in \cite{jang2023braid}) when the Cartan matrix is of finite type.
We also obtain natural quantizations for the Grothendieck rings of
representations of shifted quantum affine algebras in \cite{geiss2024representations}.

\subsection{Main results}

We briefly explain our setup before stating the main results. Throughout
the paper, we choose the base ring $\kk$ to be $\Z$ for the classical
cases and $\Z[q^{\pm\Hf}]$ for the quantum cases. 

Let $I$ denote a countable set with a partition $I=I_{\ufv}\sqcup I_{\fv}$
into the unfrozen and the frozen subsets. Let $\sd$ denote any given
seed. It consists of indeterminates $x_{i}$, $\forall i\in I$, called
cluster variables, and a locally finite $I\times I_{\ufv}$-matrix
$\tB$. We may also quantize $\sd$ by associating with it a quantization
matrix $\Lambda$. 

We denote by $\bClAlg=\bClAlg(\sd)$ the (partially compactified)
ordinary cluster algebra associated with $\sd$, and by $\bUpClAlg=\bUpClAlg(\sd)$
the (partially compactified) upper cluster algebra in the sense of
Berenstein--Fomin--Zelevinsky \cite{BerensteinFominZelevinsky05},
extended to infinite rank; see Theorem \ref{thm:start-fish-inf} for
basic properties (an analogue of the Starfish Lemma \cite{BerensteinFominZelevinsky05}\cite{BerensteinZelevinsky05}).
We write $\clAlg$ (resp. $\upClAlg$) for the localization of $\bClAlg$
(resp. $\bUpClAlg$) at the frozen variables (i.e., $x_{j}$, for
$j\in I_{\fv}$). Let $\alg$ denote the cluster algebra $\bClAlg$
or $\bUpClAlg$. Let $\base$ denote a well-behaved $\kk$-basis of
$\alg$ subject to certain natural conditions, for example, containing
all the cluster variables. The pair $(\alg,\base)$ is called a based
cluster algebra in the sense of \cite{qin2023analogs}. We refer the
reader to Section \ref{subsec:Based-cluster-algebras} for details.

Our first result shows that based cluster algebras behave well under
suitable colimits. More precisely, let there be given a chain of based
cluster algebras $(\alg_{i},\base_{i})_{i\in\N}$ such that $\alg_{i}\subset\alg_{i+1}$
and $\base_{i}=\base_{i+1}\cap\alg_{i}$. We prove that the colimit
of the chain is a based cluster algebra $(\alg,\base)$, called an
extension of $(\alg_{i},\base_{i})$ (under a mild additional hypothesis
in the upper cluster algebra case). This allows us to construct based
cluster algebras of infinite rank as colimits of finite rank ones
(Section \ref{sec:Based-cluster-algebras-infinite-ranks}).

We then apply this general construction to the cluster algebras arising
from double Bott--Samelson cells studied in \cite{qin2023analogs};
see Section \ref{sec:Reviews-on-cluster-words} for a review of their
definitions and properties. Their seeds are denoted by $\rsd=\rsd(\ubi)$
and are associated with finite signed words $\ubi$. They are of finite
rank and satisfy $\bClAlg(\rsd)=\bUpClAlg(\rsd)$. Moreover, they
possess the common triangular bases, denoted $\can(\rsd)$, in the
sense of \cite{qin2017triangular}. They also have distinguished elements
called the fundamental variables. Note that the based cluster algebra
$(\bClAlg(\rsd),\can(\rsd))$ is categorified by a monoidal category
$\cC(\ubi)$ when the corresponding Lie algebra is simply-laced (see
\cite[Theorem 8.19]{qin2023analogs}), which is invariant under braid
moves and shuffles on the signed word $\ubi$.

Let $\rsd_{i}$ denote the finite rank seeds associated with a sequence
of signed words and $(\bClAlg(\rsd_{i}),\can(\rsd_{i}))_{i\in\N}$
be the chain of the corresponding based cluster algebras. Passing
to the colimit, we obtain infinite rank seeds $\sd_{\infty}$, $\ssd_{\infty}$,
$\usd_{\infty}$, and the associated infinite rank cluster algebras
$\bClAlg(\sd_{\infty})$, $\bClAlg(\ssd_{\infty})$, $\bClAlg(\usd_{\infty})$,
where we have $\bClAlg(\ssd_{\infty})=\bClAlg(\usd_{\infty})$; see
Sections \ref{subsec:Limits-of-Coxeter} \ref{subsec:Limits-of-signed}
\ref{subsec:Cluster-algebras-identify-Kq} for detailed definitions
respectively. The bases $\can(\rsd_{i})$ extend to bases of these
infinite rank cluster algebras. 

Note that, when the corresponding Lie algebra is simply-laced, the
cluster algebras $\bClAlg(\sd_{\infty})$ and $\bClAlg(\ssd_{\infty})=\bClAlg(\usd_{\infty})$
are already known to be isomorphic to the Grothendieck rings of the
categories $\cC^{-}$ or $\cC_{\Z}$, which consist of modules of
quantum affine algebras, by Hernandez--Leclerc \cite{HernandezLeclerc09,hernandez2013cluster};
when the Lie algebra is not simply-laced, they are known to be isomorphic
to the virtual quantum Grothendieck rings (also known as Bosonic extensions)
by Jang--Lee--Oh \cite{jang2023quantization}. Using the extension
technique, we prove the following consequences.

\begin{Thm}\label{thm:intro-KR-poly-is-KL-basis}

\cite[Conjecture 2]{jang2023quantization} is true: the KR-polynomials
in the virtual quantum Grothendieck ring $\mathfrak{K}_{q}$ are contained
in the Kazhdan--Lusztig basis with respect to the standard basis. 

\end{Thm}

\begin{Thm}\label{thm:intro-A-equal-U-inf-rank}

We have $\bClAlg(\sd_{\infty})=\bUpClAlg(\sd_{\infty})$ and $\bClAlg(\usd_{\infty})=\bUpClAlg(\usd_{\infty})$.

\end{Thm}

\begin{Rem}

It is a fundamental yet largely open question to determine when we
have $\bClAlg=\bUpClAlg$ and $\clAlg=\upClAlg$, see \cite{ishibashi2023u}
for a list of known cases and \cite{qin2023analogs}\cite{oya2025note}
for more recent results. Theorem \ref{thm:intro-A-equal-U-inf-rank}
appears to make progress for the first time in the context of infinite
rank (quantum) cluster algebras.

\end{Rem}

Note that the virtual quantum Grothendieck ring $\mathfrak{K}_{q}$
admits a braid group action \cite{jang2023braid} \cite{kashiwara2024braid}.
We then show the main result of this paper.

\begin{Thm}[{Theorem \ref{thm:braid-formula-fundamental}}]\label{thm:intro-braid-fundamental}

When the Cartan matrix $C$ is of finite type, the fundamental variables
of $\bUpClAlg(\rsd)$ associated with signed words can be computed
via the braid group action.

\end{Thm}

In our proof of Theorem \ref{thm:intro-braid-fundamental}, we need
the non-trivial fact that different reduced words provide the same
cluster structure on the quantum unipotent subgroup $\qO[N_{-}]$
(Theorem \ref{thm:identical-cluster-q-qp}). We give a short proof
of this fact by lifting the result at the classical level by Shen--Weng
\cite[Theorem 1.1]{shen2021cluster} to the quantum level. After completing
this paper, the author learned that this fact was already proved by
Fujita--Hernandez--Oh--Oya \cite[Proposition 3.3]{fujita2023isomorphisms}
in type $ADE$ and by Lee--Oh \cite[Proposition 4.5]{lee2024quantum}
in all types, by different methods.

\begin{Rem}
Goncharov and Shen \cite[Section 13]{goncharov2019quantum} have systematically introduced and studied a braid group action on a large class of finite rank (quantum) cluster algebras from higher Teichm\"{u}ller spaces. For cluster algebras arising from double Bott-Samelson cells, we believe that their construction coincides with the braid group action used in this paper, possibly after minor modification. In addition, the braid group action on Grassmannians was introduced by Fraser \cite{fraser2020braid}. 

It is also worth noting that the fundamental variables of cluster algebras from double Bruhat cells coincide with the Reeb chords of positive Legendrian links, see \cite{gao2020augmentations}.

The extensions of cluster algebras from double Bott--Samelson cells are also closely related to the Cauchon--Goodearl--Letzter (CGL) extensions \cite{GY13} \cite{goodearl2018cluster}.
\end{Rem}

Theorem \ref{thm:intro-braid-fundamental} implies that $\bUpClAlg(\rsd)$
is isomorphic to the algebra $\hat{\mathcal{A}}_{\mathfrak{g}}(\beta)$,
which was introduced by Oh--Park \cite{oh2024pbw} when the author
was preparing this paper. Thus, we obtain the following:

\begin{Thm}\label{thm:intro-OH-conj}

The conjectures in \cite{oh2024pbw} are true: $\hat{\mathcal{A}}_{\mathfrak{g}}(\beta)$
is a cluster algebra and has categorification in type $ADE$.

\end{Thm}

We believe the following conjecture is true, where we should use the
braid group action for arbitrary types by Kashiwara--Kim--Oh--Park
\cite{kashiwara2024braid}.

\begin{Conj}

Theorem \ref{thm:intro-braid-fundamental} holds for arbitrary generalized
Cartan matrices.

\end{Conj}

Finally, let $\bClAlg(\dsd)$ denote the cluster algebras associated
with decorated double Bott--Samelson cells \cite{shen2021cluster},
whose initial seeds are still associated with signed words (Section
\ref{subsec:Seeds-associated-with-signed}). In type $ADE$, by extending
$(\bClAlg(\dsd_{i}))_{i\in\N}$ to $\bClAlg(\dsd_{\infty})$, we recover
the cluster algebra $\bClAlg(\sd^{\GHL})$ arising from representations
of shifted quantum affine algebras \cite{geiss2024representations}.
Similarly, we can extend $(\bUpClAlg(\dsd_{i}),\can_{i})_{i\in\N}$
to $(\bUpClAlg(\dsd_{\infty}),\can_{\infty})$

\begin{Thm}[{Proposition \ref{prop:identify-seed-GHL}, Proposition \ref{prop:quantize-seed-GHL}, Theorem \ref{thm:basis-seed-GHL}}]\label{thm:intro-shifted-q-aff-cluster}

The infinite rank cluster algebra $\bClAlg(\dsd_{\infty})$ is identical
with $\bClAlg(\sd^{\GHL})$. It has a natural quantization, which
is extended from the quantization of the seed associated with a signed
word for the double Bruhat cell $\C[G^{w_{0},w_{0}}]$ in \cite{BerensteinZelevinsky05}.
Moreover, $\bUpClAlg(\dsd_{\infty})$ has the common triangular basis
$\can_{\infty}$, whose structure constants are non-negative. 

\end{Thm}

While completing this paper, the author learned from David Hernandez
that Francesca Paganelli is preparing a work \cite{paganelli2025quantum},
in which she uses a different approach to the quantization of cluster
algebras $\bClAlg(\sd^{\GHL})$ arising from shifted quantum affine
algebras for simply-laced types. After the first version of this paper
was posted on the arXiv, another proof of Theorem \ref{thm:intro-OH-conj}
was obtained independently by Kashiwara--Kim--Oh--Park \cite{KKOP25}
using a different method.

\subsection{Contents}

In Section \ref{sec:Preliminaries-on-cluster}, we review basics of
cluster algebras.

In Section \ref{sec:Reviews-on-cluster-words}, we review cluster
algebras associated with signed words, i.e., arising from double Bott--Samelson
cells, following \cite{shen2021cluster}\cite{qin2023analogs}.

In Section \ref{sec:Based-cluster-algebras-infinite-ranks}, we present
a fundamental yet powerful method to extend finite rank based (quantum)
cluster algebras to infinite rank ones.

In Section \ref{sec:Infinite-rank-cluster-q-aff}, we extend finite
rank cluster algebras arising from representations of quantum affine
algebras to recover some infinite rank cluster algebras. Then we discuss
various applications, proving Theorem \ref{thm:intro-KR-poly-is-KL-basis}
and Theorem \ref{thm:intro-A-equal-U-inf-rank}.

In Section \ref{sec:Braid-group-action}, we compute the fundamental
variables via braid group actions, proving Theorem \ref{thm:intro-braid-fundamental}
and Theorem \ref{thm:intro-OH-conj}.

In Section \ref{sec:Cluster-algebras-from-shifted}, we extend cluster
algebras from double Bruhat cells to recover infinite rank cluster
algebras arising from representations of shifted quantum affine algebras,
proving Theorem \ref{thm:intro-shifted-q-aff-cluster}.

In Section \ref{sec:Cluster-algebras-from-dBS}, we briefly review
the decorated double Bott--Samelson cells following \cite{shen2021cluster}.

In Section \ref{sec:Skew-symmetric-bilinear-forms}, we briefly review
the skew-symmetric bilinear forms used for quantization of cluster
algebras arising from quantum affine algebras.

\subsection{Convention}

We will work with $\kk=\Z$ at the classical level and $\kk=\Z[q^{\pm\Hf}]$
at the quantum level, where $q^{\Hf}$ is a formal quantum parameter.
We understand $q^{\Hf}=1$ at the classical level. The set of non-negative
elements in $\kk$ is defined to be $\N$ or $\N[q^{\pm\Hf}]$, respectively.

We will denote $z\sim z'$ if $z=q^{\alpha}z'$ for some $\alpha\in\Q$.
In this case, we say $z$ and $z'$ $q$-commute.

All vectors will be column vectors unless otherwise specified.

Assume we are given a set $I=I_{\ufv}\sqcup I_{\fv}$. If $\sigma$
is a permutation on $I_{\ufv}$, we extend $\sigma$ to a permutation
on $I$, still denoted $\sigma$, such that it acts trivially on $I_{\fv}$.

Assume we are given any $\Q$-matrix $\tB=(b_{ij})_{i\in I,j\in I_{\ufv}}$
such that either $b_{ij}=b_{ji}=0$ or $b_{ij}b_{ji}<0$, $\forall i,j$.
We can associate with $\tB$ a (not necessarily unique) weighted oriented
graph $Q$, called a valued quiver: its set of vertices is $I$; there
is an arrow from $i$ to $j$ with weight $(b_{ij},-b_{ji})$ if $b_{ij}>0$;
the arrows between $i,j\in I_{\fv}$ could be chosen arbitrarily.
When we draw the quiver, a collection of weight $(w_{s},-w_{s}')$
arrows is equivalent to a weight $(\sum_{s}w_{s},-\sum_{s}w_{s}')$
arrow. We will use solid arrows for denoting weight $(1,1)$ arrows
and dashed arrows for weight $(\Hf,\Hf)$ arrows. We often depict
$I_{\ufv}$ by circular nodes and $I_{\fv}$ by rectangular nodes.

We use $\tB^{T}$ to denote the transpose of $\tB$. If $I'\subset I$
and $I'_{\ufv}\subset I_{\ufv}$, we might use $B_{I',I'_{\ufv}}$
to denote the submatrix $(b_{ij})_{i\in I',j\in I'_{\ufv}}$.

\section*{Acknowledgments}

The author would like to thank David Hernandez, Sira Gratz, Jianrong
Li, and Linhui Shen for helpful discussions. He also thanks Ryo Fujita for drawing
his attention to the works \cite{fujita2023isomorphisms} \cite{lee2024quantum},
which contain earlier proofs of Theorem \ref{thm:identical-cluster-q-qp}
by different methods.

\section{Preliminaries on cluster algebras\label{sec:Preliminaries-on-cluster}}

\subsection{Based cluster algebras\label{subsec:Based-cluster-algebras}}

\subsubsection*{Seeds}

Let $I$ denote a given index set with a partition $I=I_{\ufv}\sqcup I_{\fv}$
into its unfrozen and frozen part, respectively. Choose symmetrizers
$d_{i}\in\N_{>0}$, $i\in I$.

For our purpose, we only consider countable sets $I$. See \cite[Remark 2.2 and Remark 3.18]{gratz2015cluster}
for further discussion on the countability.

Let there be given $b_{ij}\in\Q$ for $i,j\in I$, such that $b_{ij}d_{j}=-b_{ji}d_{i}$.
Denote $\tB:=(b_{ik})_{i\in I,k\in I_{\ufv}}$ and $B:=\tB_{I_{\ufv}\times I_{\ufv}}$,
which are assumed to be $\Z$-matrices. We further assume that $\tB$
is locally finite, i.e., for any $j$, only finite many $b_{ij}$
and finitely many $b_{jk}$ are non-zero.

\begin{Lem}\label{lem:locally-finite-product}

Let there be given an $I_{1}\times I_{2}$-matrix $U$ and an $I_{2}\times I_{3}$
matrix $V$, such that $U$ is locally finite or $V$ is locally finite,
then $UV$ is well-defined. If both $U,V$ are locally finite, $UV$
is locally finite as well.

\end{Lem}

\begin{proof}

When $U$ or $V$ is locally finite, for any $(i,k)\in I_{1}\times I_{3}$,
we have $(UV)_{ik}=\sum_{j\in I_{2}}U_{ij}V_{jk}$, which is a finite
sum and thus well-defined. Next, assume both $U$ and $V$ are locally
finite. Then, for any given $i$, $J:=\{j|U_{ij}\neq0\}$ is finite.
Correspondingly, $K:=\{k|j\in J,V_{jk}\neq0\}$ is finite. Therefore,
only finitely many $(UV)_{ik}$ are non-zero. We can similarly show
that, for any given $k$, only finitely many $(UV)_{ik}$ are non-zero.

\end{proof}

Define $\cone:=\oplus_{i\in I}\Z f_{i}$ and $\yCone:=\oplus_{k\in I_{\ufv}}\Z e_{k}$,
where $f_{i}$, $e_{k}$ are understood as the unit vectors. The elements
of $\cone$ is denote by $m=(m_{i})_{i\in I}=\sum m_{i}f_{i}$ where
$m_{i}\in\Z$ and those in $\yCone$ by $n=(n_{k})_{k\in I_{\ufv}}=\sum n_{k}e_{k}$,
$k\in\Z$. Introduce the linear map $p^{*}:\yCone\rightarrow\cone$
such that $p^{*}(n):=\tB n:=\sum_{k\in I_{\ufv}}n_{k}(\sum_{i\in I}b_{ik}f_{i})$.

A compatible Poisson structure is a $\Z$-valued skew-symmetric bilinear
form $\lambda$ on $\cone$ such that $\lambda(f_{i},p^{*}e_{k})=\delta_{ik}\diag_{k}$
for some $\diag_{k}\in\N_{>0}$. If such a $\lambda$ is given, we
define the quantization matrix $\Lambda=(\Lambda_{ij})_{i,j\in I}:=(\lambda(f_{i},f_{j}))_{i,j\in I}$.
The pair $(\tB,\Lambda)$ is called compatible following \cite{BerensteinZelevinsky05}.
Note that such a $\lambda$ might not exist. If $\lambda$ exists,
$p^{*}$ must be injective.

When $|I|<\infty$, we say $\sd$ is of full rank if $\tB$ is of
full rank or, equivalently, $p^{*}$ is injective. In this case, a
compatible Poisson structure $\lambda$ must exist, see \cite{GekhtmanShapiroVainshtein03,GekhtmanShapiroVainshtein05}.

\begin{Def}

The collection $\sd:=(I,I_{\ufv},(d_{i})_{i\in I},\tB,(x_{i})_{i\in I})$
is called a seed, where $x_{i}$ are indeterminates called $x$-variables
or cluster variables. It is further called a quantum seed if we choose
a compatible Poisson structure $\lambda$ for $\sd$.

\end{Def}

The cardinality of $|I|$ is called the rank of $\sd$. The symbols
$I,I_{\ufv},d_{i}$ of a seed $\sd$ will often be omitted for simplicity.

\begin{Def}

We call $x^{m}$ a cluster monomial of $\sd$ if $m\in\oplus_{i\in I}\N f_{i}$
and a localized cluster monomial of $\sd$ if $m\in(\oplus_{k\in I_{\ufv}}\N f_{k})\oplus(\oplus_{j\in I_{\fv}}\Z f_{j})$.

\end{Def}

Take $\kk=\Z$ or $\Z[q^{\pm\Hf}]$, where $q^{\Hf}$ is a formal
quantum parameter. Let $\LP$ denote the Laurent polynomial ring $\kk[\cone]:=\kk[x_{i}^{\pm}]_{i\in I}$,
where we identify $x^{f_{i}}=x_{i}$ and use $\cdot$ to denote its
commutative product. We further introduce the twisted product $*$
for $\LP$:
\begin{align*}
x^{m}*x^{h} & :=q^{\Hf\lambda(m,h)}x^{m+h},\ \forall m,h\in\cone.
\end{align*}
By the multiplication of $\LP$, we mean $*$ unless otherwise specified. 

\begin{Lem}\label{lem:Ore}

$\LP$ satisfies the (left) Ore condition: for any $0\neq a,b\in\LP$,
we have $a\LP\cap b\LP\neq0$.

\end{Lem}

\begin{proof}

For any $0\neq a,b\in\LP$, we can find a subalgebra $\LP':=\kk[x_{i_{1}}^{\pm},\ldots,x_{i_{l}}^{\pm}]$
of $\LP$, such that $a,b\in\LP'$. By \cite{BerensteinZelevinsky05},
we have $a\LP'\cap b\LP'\neq0$.

\end{proof}

By Lemma \ref{lem:Ore}, we can construct the skew-field of fractions
of $\LP$, denoted by $\cF$, see \cite{BerensteinZelevinsky05}.

If $\kk=\Z[q^{\pm\Hf}]$, we introduce the bar-involution $\overline{(\ )}$
on $\LP$, which is the $\Z$-linear map such that $\overline{q^{\alpha}x^{m}}=q^{-\alpha}x^{m}$.

We also introduce $\bLP:=\kk[x_{i}^{\pm}]_{i\in I_{\ufv}}[x_{j}]_{j\in I_{\fv}}$,
which should not be misunderstood as the bar involution acting on
$\LP$. Define $\frGroup$ to be multiplicative group generated by
$x^{h}$, $h\in I_{\fv}$. 

In general, we will use the symbol $(\sd)$ to explicitly remind that
the data is associated with the seed $\sd$.

For any permutation $\sigma$ of $I$ such that $\sigma I_{\ufv}=I_{\ufv}$,
$\sigma I_{\fv}=I_{\fv}$, we have the permuted seed $\sigma\sd$
obtained from $\sd$ by relabeling the vertices via $\sigma$: $x_{\sigma i}(\sigma\sd):=x_{i}$,
$d_{\sigma i}(\sigma\sigma):=d_{i}$, $b_{\sigma i,\sigma j}(\sigma\sd):=b_{ij}$,
$\Lambda_{\sigma i,\sigma j}(\sigma\sd):=\Lambda_{ij}$, etc.

\subsubsection*{Mutations}

Let $\sd$ denote a given seed. Following \cite{fomin2002cluster}
\cite{BerensteinZelevinsky05}, for any $k\in I_{\ufv}$, we have
an operation $\mu_{k}$, called a mutation, such that $\mu_{k}$ produces
a seed $\sd':=\mu_{k}\sd=(I',I'_{\ufv},(d'_{i})_{i\in I},\tB',(x'_{i})_{i\in I})$
from $\sd$ as below. Denote $[\ ]_{+}=\max(\ ,0\}$. For any sign
$\varepsilon\in\pm1$, we introduce an $I\times I$-matrix $\tE_{\varepsilon}$
and an $I_{\ufv}\times I_{\ufv}$-matrix $F_{\varepsilon}$ such that
\begin{align*}
(\tE_{\varepsilon})_{ij} & =\begin{cases}
\delta_{ij} & i\neq k,j\neq k\\
-1 & i=j=k\\{}
[-\varepsilon b_{ik}]_{+} & i\neq k,j=k
\end{cases}\\
(F_{\varepsilon})_{ij} & =\begin{cases}
\delta_{ij} & i\neq k,j\neq k\\
-1 & i=j=k\\{}
[\varepsilon b_{kj}]_{+} & i=k,j\neq k
\end{cases}
\end{align*}

Note that $\tE_{\varepsilon}$ and $F_{\varepsilon}$ are locally-finite.
By Lemma \ref{lem:locally-finite-product}, we can define $\tB':=\tE_{\varepsilon}\tB F_{\varepsilon}$.
If $\sd$ has a quantization matrix $\Lambda$, the quantization matrix
$\Lambda'$ for $\tB'$ is given by $\Lambda'=(\tE_{\varepsilon})^{T}\Lambda\tE_{\varepsilon}$.
Note that $\tB'$ and $\Lambda'$ are well-defined, compatible, and
$\tB'$ is locally finite.

We further connect the cluster variables $x_{i}$ and $x_{i}'$ by
a $\kk$-algebra isomorphism $\mu_{k}^{*}:\cF'\simeq\cF$, such that
\begin{align*}
\mu_{k}^{*}x_{i}'= & \begin{cases}
x_{i} & i\neq k\\
x_{k}^{-1}\cdot(x^{\sum_{j\in I}[-b_{jk}]_{+}f_{j}}+x^{\sum_{i\in I}[b_{ik}]_{+}f_{i}}) & i=k
\end{cases}.
\end{align*}
We will often identify $\cF'$ and $\cF$ via $\mu_{k}^{*}$ and then
omit the symbol $\mu_{k}^{*}$ for simplicity.

Note that $\mu_{k}$ is an involution and does not depend on the choice
of $\varepsilon$. In addition, $x_{j}$, where $j\in I_{\fv}$, are
preserved by mutations. They are called the frozen variables.

In general, for any finite sequence $\uk=(k_{1},k_{2},\ldots,k_{l})$
of letters in $I_{\ufv}$, we denote the mutation sequence $\seq:=\seq_{\uk}:=\mu_{k_{l}}\cdots\mu_{k_{2}}\mu_{k_{1}}$
(read from right to left). For any permutation $\sigma$ of $I$ such
that $\sigma I_{\ufv}=I_{\ufv}$, let $\seq^{\sigma}$ denote the
composition $\sigma\seq$, called a permutation mutation sequence.

For any given initial seed $\sd_{0}$, let $\Delta^{+}:=\Delta_{\sd_{0}}^{+}$
denote the set of seeds obtained from $\sd_{0}$ by any mutation sequences.
We could use $\Delta^{+,\sigma}:=\Delta_{\sd_{0}}^{+,\sigma}$ to
denote the seeds obtained from $\sd_{0}$ by any permutation mutation
sequences.

\subsubsection*{Cluster algebras}

Let there be given an initial seed $\sd_{0}$. 

\begin{Def}

The (partially compactified) cluster algebra $\bClAlg=\bClAlg(\sd_{0})$
is defined to be $\kk$-algebra generated by all the cluster variables
$x_{i}(\sd)$, where $i\in I$, $\sd\in\Delta^{+}$. We define the
(localized) cluster algebra $\clAlg$ to be its localization at the
frozen variables $x_{j}$, $j\in I$.

We define the (partially compactified) upper cluster algebra $\bUpClAlg=\bUpClAlg(\sd_{0})$
as $\cap_{\sd\in\Delta^{+}}\bLP(\sd)$, where we identify different
skew-fields of fractions $\cF(\sd)$ via mutations. The (localized)
upper cluster algebra $\upClAlg$ is its localization at the frozen
variables or, equivalently, $\upClAlg=\cap_{\sd\in\Delta^{+}}\LP(\sd)$.

\end{Def}

By \cite{fomin2002cluster} \cite{BerensteinZelevinsky05}, for any
cluster variable $z$ and any mutation sequence $\seq$, we have $z\in\LP(\seq\sd)$.
Therefore, we have $\clAlg\subset\upClAlg$.

\begin{Thm}[{Starfish Lemma \cite{BerensteinFominZelevinsky05}\cite{BerensteinZelevinsky05}}]\label{thm:star-fish}

If $I$ is finite and $p^{*}$ is injective, $\upClAlg=\LP(\sd)\cap(\cap_{k\in I_{\ufv}(\sd)}\LP(\mu_{k}\sd))$.

\end{Thm}

\subsubsection*{Degrees and pointedness}

Choose and fix a seed $\sd$. Assume $p^{*}$ is injective for the
moment. We introduce $\yCone^{\oplus}:=\oplus_{k\in I_{\ufv}}\N e_{k}$
and $\yCone^{+}:=\oplus_{k\in I_{\ufv}}\N e_{k}\backslash\{0\}$. 

\begin{Def}

$\forall m,h\in\cone(\sd)$, we say $m$ dominates $h$, denoted $h\preceq_{\sd}m$,
if $h\in m+p^{*}\yCone^{\oplus}$.

\end{Def}

Let $\kk[\yCone]$ denote the $\kk$-subalgebra of $\LP$ spanned
by $y^{n}:=x^{p^{*}(n)}$. We denote $y_{k}:=y^{e_{k}}$. 

Then $\kk[\yCone^{\oplus}]$ is a subalgebra of $\kk[\yCone]$. Since
$p^{*}$ is injective, $\kk[\yCone^{+}]$ is a maximal ideal of $\kk[\yCone^{\oplus}]$.
Let $\widehat{\kk[\yCone^{\oplus}]}$ denote the completion of $\kk[\yCone]$
with respect to $\kk[\yCone^{+}]$. We then introduce the following
rings of formal Laurent series:
\begin{align*}
\hLP & :=\LP\otimes_{\kk[\yCone^{\oplus}]}\widehat{\kk[\yCone^{\oplus}]}\\
\widehat{\kk[\yCone]} & :=\kk[\yCone]\otimes_{\kk[\yCone^{\oplus}]}\widehat{\kk[\yCone^{\oplus}]}
\end{align*}

\begin{Def}

An element $z\in\hLP$ is said to have degree $m$ for some $m\in\cone$
if $z=x^{m}\cdot\sum_{n\in\yCone^{\oplus}}c_{n}y^{n}$, $c_{0}\neq0$,
$c_{n}\in\kk$. We denote $\deg^{\sd}z:=m$. Note that $\deg^{\sd}x^{m}\cdot y^{n}=m+p^{*}n\preceq_{\sd}m$.

$z$ is further said to be $m$-pointed if $c_{0}=1$. In this case,
we further define the normalization of $q^{\alpha}z$, $\alpha\in\Q$,
to be $[q^{\alpha}z]:=[q^{\alpha}z]^{\sd}:=z$. 

\end{Def}

By \cite{FominZelevinsky07}\cite{gross2018canonical}\cite{Tran09},
all cluster variables are pointed in $\LP(\sd)$.

\begin{Def}

We say $\sd$ is injective-reachable if there exists a mutation sequence
$\Sigma$ and a permutation $\sigma$ of $I_{\ufv}$, such that, $\forall k\in I_{\ufv}$,
$x_{\sigma k}(\Sigma\sd)$ are $(-f_{k}+p_{k})$-pointed in $\LP(\sd)$
for some $p_{k}\in\Z^{I_{\fv}}$. We denote $\Sigma\sd$ by $\sd[1]$
in this case. 

\end{Def}

The mutation sequence $\Sigma$ is called a green to red sequence
in \cite{keller2011cluster}. We observe that, if $\sd$ is injective-reachable,
$I_{\ufv}$ must be a finite set. Moreover, in this case, all seeds
in $\Delta^{+}$ are injective-reachable, see \cite{qin2017triangular}\cite{muller2015existence}.

\subsubsection*{Tropical points}

For any $k\in I_{\ufv}$, $\sd'=\mu_{k}\sd$, we have the tropical
mutation $\phi_{\sd',\sd}$ from $\cone$ to $\cone'$.\footnote{The notion of cluster tropicalization originates from the work of Fock and Goncharov in their formulation of the Duality Conjecture \cite{FockGoncharov09}.} It is a piecewise
linear map such that, $\forall m=(m_{i})_{i}=\sum m_{i}f_{i}$, its
image $m'=(m'_{i})_{i}=\sum m'_{i}f_{i}$ is given by
\begin{align*}
m'_{i}= & \begin{cases}
-m_{k} & i=k\\
m_{i}+[b_{ik}]_{+}[m_{k}]_{+}+[-b_{ik}]_{+}[-m_{k}]_{+} & i\neq k
\end{cases}.
\end{align*}

In general, for any $\sd'=\seq\sd$, where $\seq$ denotes any mutation
sequence, let $\phi_{\sd',\sd}:\cone\rightarrow\cone'$ denote the
composition of tropical mutations along $\seq$. It only depends on
$\sd',\sd$, see \cite{gross2013birational}.

Let $[m]$ denote the equivalent class of $m$ in $\sqcup_{\sd'\in\Delta^{+}}\cone(\sd')$
under the equivalence relation induced by the identifications $\phi_{\sd',\sd}$,
called the tropical point represented by $m$. Let $\tropSet$ denote
the set of all tropical points.

An element $z\in\upClAlg(\sd)$ is called $[m]$-pointed for $m\in\cone(\sd)$
if it is $\phi_{\sd',\sd}m$-pointed in $\LP(\sd')$ for any $\sd'=\seq\sd$.

\begin{Def}

Let $\Theta$ denote any subset of $\cone(\sd)$ or $\tropSet$. A
subset $\cZ\subset\LP(\sd)$ is called $\Theta$-pointed if it takes
the form $\{\theta_{p}|p\in\Theta\}$ such that $\theta_{p}$ are
$p$-pointed.

\end{Def}

\begin{Thm}[\cite{qin2019bases}]\label{thm:tropical-basis}

Assume that $\sd$ is injective-reachable. If $\cZ$ is $\tropSet$-pointed,
it is a $\kk$-basis of $\upClAlg$.

\end{Thm}

\subsubsection*{Orders of vanishing and optimized seeds}

$\forall z\in\cF=\cF(\sd)$, we can write its reduced form $Z=x_{j}^{\nu_{j}(z)}*P*Q^{-1}$,
where $\nu_{j}(z)\in\Z$ and $P,Q\in\kk[x_{i}]_{i\in I}$ are not
divisible by $x_{j}$. 

\begin{Def}\label{def:order-vanishing}

$\nu_{j}(z)$ is called the order of vanishing of $z$ at $x_{j}=0$. 

\end{Def}

Note that the map $\nu_{j}:z\mapsto\nu_{j}(z)$ is a valuation on
$\cF(\sd)$. Moreover, it is independent of the choice of $\sd\in\Delta^{+}$,
see \cite[Lemma 2.12]{qin2023freezing}.

Observe that we have $\bLP=\{z\in\LP|\nu_{j}(z)\geq0,\forall j\in I_{\fv}\}$
and thus $\bUpClAlg=\{z\in\upClAlg|\nu_{j}(z)\geq0,\forall j\in I_{\fv}\}$.
However, we only know $\bClAlg\subset\{z\in\clAlg|\nu_{j}(z)\geq0,\forall j\in I_{\fv}\}$
in general. So $\clAlg\subset\upClAlg$ and Theorem \ref{thm:star-fish}
imply the following.

\begin{Cor}\label{cor:star-fish-bar}

We have $\bClAlg\subset\bUpClAlg$. Moreover, when $I$ is finite
and $p^{*}$ is injective, we have $\bUpClAlg=\bLP(\sd)\cap(\cap_{k\in I_{\ufv}(\sd)}\bLP(\mu_{k}\sd))$.

\end{Cor}

\begin{Def}

Let $j$ denote any frozen vertex of $\sd$. It is said to be optimized
in $\sd$ if $b_{jk}\geq0$, $\forall k\in I_{\ufv}$. It is called
non-essential if $B_{j,I_{\ufv}}=0$, and essential otherwise.

\end{Def}

We say $j\in I_{\fv}$ can be optimized if it is optimized in some
seed $\sd_{j}\in\Delta^{+}$, and $\sd$ can be optimized if all of
its frozen vertices can be optimized.

Assume that $p^{*}$ is injective and $j\in I_{\fv}$ are optimized
in $\sd_{j}$, respectively. Then, for any $m^{(j)}$-pointed element
$z\in\LP(\sd_{j})$, $m^{(j)}\in\cone(\sd_{j})$, we have $\nu_{j}(z)=(m^{(j)})_{j}$.
Therefore, for any $[m]$-pointed element $z\in\upClAlg$ where $m\in\cone(\sd)$,
we have $\nu_{j}(z)=(\phi_{\sd_{j},\sd}m)_{j}$. We deduce that $z\in\bUpClAlg$
if and only if $(\phi_{\sd_{j},\sd}m)_{j}\geq0$, $\forall j\in I_{\fv}$.

\begin{Prop}[{\cite[Proposition 2.15]{qin2023freezing}}]

Let $\cZ$ be a $\tropSet$-pointed $\kk$-basis of $\upClAlg$. If
$\sd$ can be optimized, $\cZ\cap\bUpClAlg$ is a basis of $\bUpClAlg$.

\end{Prop}

\subsubsection*{Common triangular bases}

Assume $I_{\ufv}$ is finite and $\sd$ is injective-reachable for
the moment. Let $\alg$ denote $\clAlg$ or $\upClAlg$. Let $\can$
denote a given $\kk$-basis of $\alg$.

\begin{Def}

The basis $\can$ is called the triangular basis of $\alg$ with respect
to $\sd$ if the following conditions hold:

\begin{enumerate}

\item (Pointedness) It takes the form $\can=\{\can_{m}|m\in\cone(\sd)\}$
such that $\can_{m}$ are $m$-pointed.

\item (Bar-invariance) We have $\overline{\can_{m}}=\can_{m}$.

\item (Cluster compatibility) $\can$ contains the cluster monomials
in $\sd$ and $\sd[1]$.

\item (Triangularity) $\forall x_{i}(\sd)$, $\can_{m}$, we have
the following decomposition
\begin{align}
[x_{i}(\sd)*\can_{m}]^{\sd}= & \can_{m+f_{i}}+\sum_{m'\prec_{\sd}m}b_{m'}\can_{m'},\ \text{for some }b_{m'}\in\mm:=q^{-\Hf}\Z[q^{-\Hf}].\label{eq:triangular-decomposition}
\end{align}

\end{enumerate}

\end{Def}

A decomposition taking the form on the right hand side of (\ref{eq:triangular-decomposition})
is called a $(\prec_{\sd},\mm)$-unitriangular decomposition \cite{qin2017triangular}. 

Assume that $\can$ is the triangular basis. Then $\can_{m}$ is determined
by the normalized ordered products of localized cluster monomials
of the form $[x(\sd)^{m}*x(\sd[1])^{m'}]^{\sd}$, where $m_{k}m_{k'}=0$,
$\forall k\in I_{\ufv}$, via a Kazhdan-Lusztig type algorithm, see
\cite{qin2017triangular}. In particular, the triangular basis is
unique if it exists.

Note that $\can$ is closed under the $\frGroup$ commutative product:
\begin{align*}
\forall p\in\frGroup,\  & p\cdot\can_{m}\in\can.
\end{align*}

\begin{Def}

If $\can$ is the triangular basis of $\alg$ with respect to all
of its seed, it is called the common triangular basis.

\end{Def}

Assume $\can$ is the triangular basis with respect to $\sd$. By
\cite{qin2020dual}, it is the common triangular basis if and only
if it contains all cluster monomials. In this case, it is further
$\tropSet$-pointed. Moreover, Theorem \ref{thm:tropical-basis} implies
that it is a basis of $\upClAlg$, i.e., we must have $\alg=\upClAlg$.

Let $\alg$ denote $\clAlg$ or $\upClAlg$, and $\overline{\alg}$
denote $\bClAlg$ or $\bUpClAlg$ respectively. Assume that $\can$
is the triangular basis (resp. common triangular basis) of $\alg$.
If $\can\cap\overline{\alg}$ is a basis of $\overline{\alg}$, we
call it the triangular basis (resp. common triangular basis) of $\overline{\alg}$.

\subsubsection*{Based cluster algebras}

Let $\alg$ denote $\bClAlg$, $\clAlg$, $\bUpClAlg$, or $\upClAlg$.
Let there be given a $\kk$-basis $\base$ for $\alg$. We recall
based cluster algebra introduced in \cite{qin2023analogs}.

\begin{Def}

The pair $(\alg,\base)$ is called a based cluster algebra if the
following conditions hold:

\begin{enumerate}

\item $\base$ contains all cluster monomials.

\item $\forall j\in I_{\fv}$, we have $x_{j}\cdot\base\subset\base$.

\item Any $b\in\base$ is contained in $x^{m}\cdot\kk[\yCone]$ for
some $m\in\cone$.

\item At the quantum level, we further require $\overline{b}=b$,
$\forall b\in\base$.

\end{enumerate}

\end{Def}

Let $\F$ be a given field. Let $\cT$ be a $\F$-linear tensor category
in the sense of \cite[Section A.1]{kang2018symmetric}. Then its object
have finite lengths, and its tensor functor $(\ )\otimes(\ )$ is
an exact bifunctor. Let $[X]$ denote the isoclass of an object $X\in\cT$.
Then its Grothendieck ring $K_{0}(\cT)$ is unital and associative,
whose multiplication is induced from the tensor product. Note that
we have $[X\otimes Y]=\sum_{S}c_{XY}^{S}[S]$, $\forall X,Y\in\cT$,
where $S$ appearing are simple objects and $c_{XY}^{S}\in\N$.

When we work with classical cluster algebras, we assume $K_{0}(\cT)$
is a commutative $\Z$-algebra and denote $K=K_{0}(\cT)$.

When we work with quantum cluster algebras such that $\kk=\Z[q^{\pm\Hf}]$,
we make one of the following assumptions.

\begin{enumerate}

\item Assume that $K_{0}(\cT)$ is a $\Z[q^{\pm}]$-algebra, such
that $[q^{\pm}S]$ is the isoclass of a simple object whenever $[S]$
is. Define $K=K_{0}(\cT)\otimes\kk$.

\item Assume that $K_{0}(\cT)$ is a commutative $\Z$-algebra. Moreover,
assume that we can associate to $K_{0}(\cT)\otimes_{\Z}\Z[q^{\pm\Hf}]$
a $q$-twisted multiplication $*$ such that $[X]*[Y]=\sum_{S}c(q^{\Hf})_{XY}^{S}[S]$
with $c(q^{\Hf})_{XY}^{S}\in\N[q^{\pm\Hf}]$, $c(1)_{XY}^{S}=c_{XY}^{S}$.
Denote $K=K_{0}(\cT)$.

\end{enumerate}

In either case, $K$ is called the deformed Grothendieck ring associated
with $\cT$.

\begin{Def}\label{def:categorification}

We say $\cT$ categorifies $\alg$ if there is a $\kk$-algebra isomorphism
$\kappa:\alg\simeq K$, such that, for any cluster monomial $z$,
$\kappa z\sim[S]$ for some simple object $S$.

We say $\cT$ categories a based cluster algebra $(\alg,\base)$ if
$\cT$ categorifies $\alg$ and, $\forall b\in\base$, $\kappa b\sim[S]$
for some simple object $S$.

\end{Def}

\subsection{Cluster embeddings and freezing\label{subsec:Sub-seeds}}

\subsubsection*{Cluster embeddings}

Let $\sd$ and $\sd'$ denote two seeds.

\begin{Def}[\cite{qin2023analogs}]

A cluster embedding $\iota$ from $\sd$ to $\sd'$ is an embedding
$\iota:I'\rightarrow I$ such that $\iota I'_{\ufv}\subset I$, $d'_{i}=d_{\iota i}$,
and $b'_{i,k}=b_{\iota i,\iota k}$, $\forall$$i\in I'$, $k\in I'_{\ufv}$.
When $\sd$ and $\sd'$ are quantum seeds, we further require $\Lambda'_{ij}=\Lambda_{\iota i,\iota j}$,
$\forall i,j\in I$.

\end{Def}

Let $\iota$ be a cluster embedding from $\sd$ to $\sd'$. Take any
mutation sequence $\seq=\mu_{k_{l}}\cdots\mu_{k_{1}}$ on $I_{\ufv}'$
and denote $\iota\seq:=\mu_{\iota k_{l}}\cdots\mu_{\iota k_{1}}$.
Then $\iota$ is also a cluster embedding from the classical seed
$(\iota\seq)\sd'$ to $\seq\sd$.

\begin{Def}

$\sd'$ is a good sub seed of $\sd$ via the cluster embedding $\iota$
if $\tB_{(I\backslash\iota I')\times\iota I_{\ufv}'}=0$.

\end{Def}

Note that a cluster embedding $\iota$ induces an inclusion $\iota:\cF(\sd')\hookrightarrow\cF(\sd)$
such that $\iota(x'_{i})=x_{\iota i}$, $\forall i\in I$. When $\sd'$
is a good sub seed of $\sd$ via $\iota$, $\iota$ is a cluster embedding
from the (quantum) seed $(\iota\seq)\sd'$ to $\seq\sd$, and $\seq\sd'$
is also a good sub seed of $(\iota\seq)\sd$ via $\iota$. Moreover,
we have $\iota(x_{i}(\seq\sd'))=x_{\iota i}((\iota\seq)\sd)$ in $\cF(\sd)$.

\begin{Prop}[{\cite{qin2023analogs}}]\label{prop:inclusion-good-sub-cl-alg}

Assume $\sd'$ is a good sub seed of $\sd$. We have $\iota(\bClAlg(\sd'))\subset\bClAlg(\sd)$
and $\iota(\bLP(\seq\sd'))\subset\bLP(\seq\sd)$ for any mutation
sequence $\seq$. Moreover, if $|I(\sd)|<\infty$ and $\sd$ is of
full rank, we have $\iota(\bUpClAlg(\sd'))\subset\bUpClAlg(\sd)$.

\end{Prop}

\begin{Eg}

Any permutation $\sigma$ on $I$, such that $\sigma I_{\ufv}=I_{\ufv}$,
is a cluster embedding from $\sd$ to $\sigma\sd$.

\end{Eg}

\subsubsection*{Freezing}

Let $\sd$ be a given seed. Choose any subset $F$ of $I$. By freezing
the vertices in $F$, we obtain a new seed $\frz_{F}\sd$, such that
$I_{\ufv}(\frz_{F}\sd)=I_{\ufv}\backslash F$.

Note that $\frz_{F}\sd$ is a good sub seed of $\sd$ via the cluster
embedding $\iota:I\rightarrow I$, which is the identity map.

\section{Reviews on cluster algebras associated with signed words\label{sec:Reviews-on-cluster-words}}

\subsection{Seeds associated with signed words\label{subsec:Seeds-associated-with-signed}}

\subsubsection*{Signed words}

Let $J$ denote a finite subset of $\Z_{>0}$. Let $\ubi$ denote
a finite sequence $(\bi_{1},\ldots,\bi_{l})$ of letters in $\pm J$,
called a signed word, where $l\in\N$. Denote its length by $l(\ubi):=l$.
For any $k\in[1,l]$, define its successor and predecessor to be 
\begin{align*}
k[1] & :=\min(\{k'\in[k+1,l],|\bi_{k'}|=|\bi_{k}|\}\cup\{+\infty\}),\\
k[-1] & :=\max(\{k'\in[1,k-1]\},|\bi_{k'}|=|\bi_{k}|\}\cup\{-\infty\}).
\end{align*}
We then define $k[d\pm1]:=k[d][\pm1]$ inductively when $k[d]\in\Z$,
$d\in\Z$. 

For $a\in J$, we define the orders 
\begin{align*}
O^{\ubi}([j,k];a) & :=|\{s\in[j,k]\ |\ |\bi_{s}|=a\},\\
O^{\ubi}(a) & :=O^{\ubi}([1,l];a),\\
o_{+}^{\ubi}(k) & :=O^{\ubi}([k+1,l];|\bi_{k}|),\\
o_{-}^{\ubi}(k) & :=O^{\ubi}([1,k-1];|\bi_{k}|).
\end{align*}
Denote $k^{\max}:=k[o_{+}^{\ubi}(k)]$ and $k^{\min}:=k[o_{-}^{\ubi}(k)]$. 

We denote $\ubi_{[j,k]}:=(\bi_{j},\ldots,\bi_{k})$, $-\ubi:=(-\bi_{1},\ldots,-\bi_{l})$,
$\ubi\op:=(\bi_{l},\ldots,\bi_{1})$, and the support $\supp\ubi:=\{a\in J|O^{\ubi}(a)>0\}$.
For any $s\in\N$, we let $\ubi^{s}$ denote the signed word $(\ubi,\ubi,\ldots,\ubi)$
where $\ubi$ appears $s$ times.

We further introduce the vertex set $I(\ubi):=\{\binom{a}{d}^{\ubi}|a\in\supp\ubi,d\in[0,O^{\ubi}(a)-1]\}$.
We will identify it with $[1,l]$ via the isomorphism $[1,l]\simeq I(\ubi)$
such that $k$ is identified with $\binom{|\bi_{k}|}{o_{-}^{\ubi}(k)}^{\ubi}$.
Then the natural order on $[1,l]$ induces the order $<_{\ubi}$ on
$I(\ubi)$. Define the sign $\varepsilon_{k}:=\varepsilon_{\binom{|\bi_{k}|}{o_{-}^{\ubi}(k)}^{\ubi}}:=\sign(\bi_{k})$.

Note that, when $\ubi'$ is another signed word with $O^{\ubi}(a)=O^{\ubi'}(a)$,
$\forall a\in J$, we can naturally identify $\binom{a}{d}^{\ubi}$
with $\binom{a}{d}^{\ubi'}$. However, $<_{\ubi}$ and $<_{\ubi'}$
are different in general.

We often omit the superscript $\ubi$ when the context is clear. And
we will use the symbols $k$ and $\binom{a}{d}$ interchangeably.

\subsubsection*{Positive braids}

We choose and fix a generalized Cartan matrix $C=(C_{ab})_{a,b\in J}$,
i.e., it satisfies $C_{aa}=2$, $C_{ab}\leq0$ for $a\neq b$, and
there exist $\symm_{a}\in\N_{>0}$ for $a\in J$ such that $\symm_{a}C_{ab}=\symm_{b}C_{ba}$.
Let $\symm$ denote the least common multiplier of $\{\symm_{a},a\in J\}$
and denote $\symm_{a}^{\vee}:=\frac{\symm}{\symm_{a}}$. 

The monoid of positive braids, denoted $\Br^{+}$, is generated by
$\sigma_{a}$, $\forall a\in J$, such that 
\begin{eqnarray*}
\sigma_{a}\sigma_{b} & =\sigma_{b}\sigma_{a} & \text{if }C_{ab}C_{ba}=0\\
\sigma_{a}\sigma_{b}\sigma_{a} & =\sigma_{b}\sigma_{a}\sigma_{b} & \text{if }C_{ab}C_{ba}=1\\
(\sigma_{a}\sigma_{b})^{2} & =(\sigma_{b}\sigma_{a})^{2} & \text{if }C_{ab}C_{ba}=2\\
(\sigma_{a}\sigma_{b})^{3} & =(\sigma_{b}\sigma_{a})^{3} & \text{if }C_{ab}C_{ba}=3
\end{eqnarray*}
Let $\Br$ denote the group associated with $\Br^{+}$ by adjoining
the inverses $\sigma_{a}^{-1}$, $\forall a\in J$. Let $e$ denote
the identity element. For any word $\ueta=(\eta_{1},\ldots,\eta_{l})$
of the letters of $J$, we define $\beta_{\ueta}:=\sigma_{\eta_{1}}\cdots\sigma_{\eta_{l}}\in\Br^{+}$.
Denote $l(\beta_{\ueta}):=l(\ueta)=l$.

The Weyl group $W$ is the quotient of $\Br$ by the relations $\sigma_{a}^{2}=e$.
The image of $\beta\in\Br$ will be denoted by $[\beta]$ in $W$.
We denote $w_{\ueta}:=[\beta_{\ueta}]$. The length $l(w)$ of $w\in W$
is defined to be $\min\{l(\ueta)\ |\text{ any word \ensuremath{\ueta} satisfies }w_{\ueta}=w\}$. 

\subsubsection*{Seeds associated with signed words}

Let there be given any signed word $\ubi$. Following \cite[Section 3.1]{shen2021cluster}, we associate with $\ubi$ a seed $\sd=\rsd(\ubi)$, which was constructed via the amalgamation procedure
introduced by Fock and Goncharov \cite{FockGoncharov06}. Let us define the seed as follows. Denote  $I=I(\ubi)\simeq[1,l]$,
$I_{\fv}=\{\binom{a}{O_{+}(a)-1}|a\in\supp\ubi\}\simeq\{k^{\max}|k\in[1,l]\}$,
$d_{\binom{a}{d}}:=\symm_{a}^{\vee}$. As in \cite[(6.1)]{qin2023analogs},
define $\tB=(b_{jk})_{j\in I,k\in I_{\ufv}}$ such that
\begin{align}
b_{jk} & =\begin{cases}
\varepsilon_{k} & k=j[1]\\
-\varepsilon_{j} & j=k[1]\\
\varepsilon_{k}C_{|\bi_{j}|,|\bi_{k}|} & \varepsilon_{j[1]}=\varepsilon_{k},\ j<k<j[1]<k[1]\\
\varepsilon_{k}C_{|\bi_{j}|,|\bi_{k}|} & \varepsilon_{k}=-\varepsilon_{k[1]},\ j<k<k[1]<j[1]\\
-\varepsilon_{j}C_{|\bi_{j}|,|\bi_{k}|} & \varepsilon_{k[1]}=\varepsilon_{j},\ k<j<k[1]<j[1]\\
-\varepsilon_{j}C_{|\bi_{j}|,|\bi_{k}|} & \varepsilon_{j}=-\varepsilon_{j[1]},\ k<j<j[1]<k[1]\\
0 & \text{otherwise}
\end{cases}.\label{eq:dBS_B_matrix}
\end{align}
Note that $p^{*}$ is injective \cite{qin2023analogs}. And we can
associate a compatible Poisson structure with $\rsd(\ubi)$.

Next, assume that $\ubi_{[1,|J|]}$ is a Coxeter word, i.e., $O([1,|J|];a)=1$,
$\forall a\in J$. Let us denote $\ubi'=\ubi_{[|J|+1,l]}$. Let $F$
denote $\{\binom{a}{0}\in I(\ubi)|a\in\supp\ubi\}$. We define the
seed $\dsd(\ubi'):=\frz_{F}(\rsd(\ubi))$. It only depends on $\ubi'$.
We denote the elements $\binom{a}{0}^{\ubi}\in F$ by $\binom{a}{-1}^{\ubi'}$.
Then $I(\dsd(\ubi'))$ is given 
\begin{align*}
\ddI(\ubi') & :=\{\binom{a}{-1}^{\ubi'}|a\in J\}\sqcup I(\ubi').
\end{align*}
For any choice of Coxeter word $\uc$, we could extend $\phi:I(\ubi')\simeq[1,l(\ubi')]$
to $\ddI(\ubi')\simeq[1-|J|,l(\ubi')]$ such that $\binom{c_{k}}{-1}^{\ubi'}$
are identified with $k-|J|$.

For $\sd=\rsd(\ubi)$ or $\dsd(\ubi)$, we denote $\binom{a}{d}^{\ubi}\in I(\sd)$
by $\binom{a}{d}^{\sd}$ as well. 

\begin{Eg}\label{eg:SL3-w0-w0}

Take $C=\left(\begin{array}{cc}
2 & -1\\
-1 & 2
\end{array}\right)$. The longest Weyl group element $w_{0}$ has a reduced word $s_{1}s_{2}s_{1}$.
Choose the signed word $\ubi=(1,-1,2,-2,1,-1)$. A quiver for $\dsd(\ubi)$
is depicted in Figure \ref{fig:G-u-v}, where we choose the Coxeter
word $\uc=(1,2)$ and identify $\ddI(\ubi)\simeq[-1,l(\ubi)]$. It
is known that $\dsd(\ubi)$ is a seed for the cluster structure on
the (quantized) coordinate ring of the double Bruhat cell $SL_{3}^{w_{0},w_{0}}$.

\end{Eg}

\begin{figure}[h]

\caption{A quiver for $\dsd(1,-1,2,-2,1,-1)$ }
\label{fig:G-u-v}

\begin{tikzpicture}  [node distance=48pt,on grid,>={Stealth[length=4pt,round]},bend angle=45, inner sep=0pt,      pre/.style={<-,shorten <=1pt,>={Stealth[round]},semithick},    post/.style={->,shorten >=1pt,>={Stealth[round]},semithick},  unfrozen/.style= {circle,inner sep=1pt,minimum size=1pt,draw=black!100,fill=red!50},  frozen/.style={rectangle,inner sep=1pt,minimum size=12pt,draw=black!75,fill=cyan!50},   point/.style= {circle,inner sep=0pt, outer sep=1.5pt,minimum size=1.5pt,draw=black!100,fill=black!100},   boundary/.style={-,draw=cyan},   internal/.style={-,draw=red},    every label/.style= {black}]        
\node[frozen] (q-1)  at (3.5,0.5) {-1}; \node[frozen] (q-2)  at (3,-0.5) {0}; \node[unfrozen] (q1) at (2.5,0.5) {1}; \node[unfrozen] (q2) at (1.5,0.5) {2}; \node[unfrozen] (q3) at (1,-0.5) {3}; \node[frozen] (q4) at (0,-0.5) {4}; \node[unfrozen] (q5) at (-0.5,0.5) {5}; \node[frozen] (q6) at (-1.5,0.5) {6}; \draw[->,teal]  (q-1) edge (q1); \draw[->,teal]  (q2) edge (q1); \draw[->,teal]  (q2) edge (q5); \draw[->,teal]  (q6) edge (q5); \draw[->,teal]  (q5) edge (q4); \draw[->,teal]  (q4) edge (q3); \draw[->,teal]   (q-2) edge (q3); \draw[->,teal]  (q3) edge (q2); \draw[->,teal]   (q1) edge (q-2); \draw[->,dotted,teal]  (q-2) edge (q-1); \draw[->,dotted,teal]  (q4) edge (q6); \end{tikzpicture}

\end{figure}

\subsection{Operations on signed words\label{subsec:Operations-on-signed-words}}

Let $\ubi$ denote a signed word. It is a shuffle of $-\uzeta$ and
$\ueta$, where $\uzeta$ and $\ueta$ are words in $J$. Denote $\uxi:=(\uzeta\op,\ueta)$. 

Denote $\rsd=\rsd(\ubi)$. Following \cite[Section 2.3, Proposition 3.7 ]{shen2021cluster},
we introduce the following operations on $\ubi$, which will produce
new signed words $\ubi'$ and new seeds $\rsd':=\rsd(\ubi')$.

\begin{enumerate}

\item (Left reflection) change $\ubi$ to $\ubi':=(-\bi_{1},\ubi_{[2,l]})$
. In this case, $\rsd'=\rsd$.

\item (Flips) Assume that $\ubi_{[j,j+1]}=(\varepsilon a,-\varepsilon b)$,
where $\varepsilon\in\{\pm1\}$, $a,b\in J$. Change $\ubi$ to $\ubi':=(\ubi_{[1,j-1]},-\varepsilon b,\varepsilon a,\ubi_{[j+2,l]})$.
In this case, $\rsd'=\rsd$ if $a\neq b$ and $\rsd'=\mu_{j}\rsd$
if $a=b$.

\item (Braid moves) Assume $\ubi_{[j,k]}$ and $\ugamma$ are two
words in $J$ such that $\beta_{\ubi_{[j,k]}}=\beta_{\ugamma}$ for
some word $\ugamma$. Change $\ubi$ to $\ubi':=(\ubi_{[1,j-1]},\ueta',\ubi_{[k+1,l]})$.
In this case, there is a sequence of mutations $\seq_{\ubi',\ubi}$
acting on $U:=\{r\in[j,k]|r[1]\leq k\}$ and a permutation $\sigma$
on $[j,k]$, such that $\sigma U=U$ and $\rsd'=\sigma\seq_{\ubi',\ubi}\rsd$.
We use $\seq^{\sigma}:=\seq_{\ubi',\ubi}^{\sigma}$ to denote $\sigma\seq_{\ubi',\ubi}$.

\end{enumerate}

We observe that, for any shuffle $\ubi'$ of $-\uzeta,\ueta$, $\rsd(\ubi')$
can be obtained from $\rsd(\ubi)$ by flips; the seed $\rsd(\uxi)$
can be obtained from $\rsd(\ubi)$ by left reflections and flips.

Next, let us discuss the mutation sequence associated with composition
of operations. Let $\ubi^{(i)}$, $i\in[1,3]$, denote three signed
words. Let $\sd^{(i)}$ denote $\rsd(\ubi^{(i)})$ or $\dsd(\ubi^{(i)})$.
When $\sd^{(i)}=\rsd(\ubi^{(i)})$, we assume $\ubi^{(i)}$ are connected
by left reflections, flips, and braid moves; when $\sd^{(i)}=\dsd(\ubi^{(i)})$,
we assume $\ubi^{(i)}$ are connected by flips and braid moves. Let
$\seq_{\ubi^{(j)},\ubi^{(i)}}^{\sigma}$ denote any chosen permutation
mutation sequences associated with the above operations, such that
$\sd^{(j)}=\seq_{\ubj^{(j)},\ubi^{(i)}}^{\sigma}\sd^{(i)}$. Denote
$\clAlg^{(i)}:=\clAlg(\sd^{(i)})$. We are interested in the following
diagram

\begin{align}
\begin{array}{ccc}
\clAlg^{(3)} & \overset{(\seq_{\ubi^{(3)},\ubi^{(1)}}^{\sigma})^{*}}{\xrightarrow{\sim}} & \clAlg^{(1)}\\
\simeqd(\seq_{\ubi^{(3)},\ubi^{(2)}}^{\sigma})^{*} &  & \parallel\\
\clAlg^{(2)} & \overset{(\seq_{\ubi^{(2)},\ubi^{(1)}}^{\sigma})^{*}}{\xrightarrow{\sim}} & \clAlg^{(1)}
\end{array}.\label{eq:mutation-different-words}
\end{align}

At the quantum level, choose any quantization matrix $\Lambda^{(1)}$
for $\sd^{(1)}$. Let $\phi_{\ubj^{(i)},\ubi^{(j)}}^{\sigma}$ denote
the permutation and tropical mutation sequence associated with $\seq_{\ubj^{(j)},\ubi^{(i)}}^{\sigma}$,
such that $\phi_{\ubj^{(i)},\ubi^{(j)}}^{\sigma}f_{k}=\deg^{\sd^{(i)}}(\seq_{\ubj^{(j)},\ubi^{(i)}}^{\sigma})^{*}x_{k}(\sd^{(j)})$.
Choose the quantization for $\clAlg^{(i+1)}$ such that it is induced
from that of $\clAlg^{(i)}$ by $\phi_{\ubj^{(i+1)},\ubi^{(i)}}^{\sigma}$:
\begin{align*}
\Lambda_{k,h}^{(i+1)} & :=\Lambda^{(i)}(\phi_{\ubi^{(i)},\ubi^{(i+1)}}^{\sigma}f_{k},\phi_{\ubi^{(i)},\ubi^{(i+1)}}^{\sigma}f_{h}).
\end{align*}
Then $(\seq_{\ubi^{(i+1)},\ubi^{(i)}}^{\sigma})^{*}$ are algebra
homomorphisms. We will see $(\seq_{\ubi^{(3)},\ubi^{(1)}}^{\sigma})^{*}$
is an algebra homomorphism in the proof of Lemma \ref{lem:words-mutation-seq-connect-seeds}.

\begin{Lem}\label{lem:words-mutation-seq-connect-seeds}

Diagram (\ref{eq:mutation-different-words}) is commutative.

\end{Lem}

\begin{proof}

First work at the classical level $\kk=\Z$. Then $\sd^{(i)}$ can
be realized as seeds of the same cluster structure on the coordinate
ring of double Bott--Samelson cells. More precisely, we have algebra
isomorphisms $\kappa^{(i)}:\clAlg^{(i)}\otimes\C\simeq\C[X_{\beta_{\ueta}}^{\beta_{\uzeta}}]$,
where $X$ denotes the double Bott--Samelson cell $\dBS$ or $\ddBS$,
such that $\kappa^{(j)}=\kappa^{(i)}(\seq_{\ubj^{(j)},\ubi^{(i)}}^{\sigma})^{*}$,
see Section \ref{sec:Cluster-algebras-from-dBS}, (\ref{eq:ddBS-unique-cluster}),
and (\ref{eq:dBS-unique-cluster}). We deduce that (\ref{eq:mutation-different-words})
is commutative.

Next, work at the quantum level $\kk=\Z[q^{\pm\Hf}]$. We have $\phi_{\ubj^{(1)},\ubi^{(2)}}^{\sigma}\phi_{\ubj^{(2)},\ubi^{(3)}}^{\sigma}=\phi_{\ubj^{(1)},\ubi^{(3)}}^{\sigma}$
by the result at the classical level. Then $\Lambda^{(3)}$ is induced
from $\Lambda^{(1)}$ via $\phi_{\ubi^{(1)},\ubi^{(3)}}^{\sigma}$.
Thus $(\seq_{\ubi^{(3)},\ubi^{(1)}}^{\sigma})^{*}$ is an algebra
homomorphism. 

Now, we take any cluster variables $x_{k}(\sd^{(3)})$ of $\sd^{(3)}$.
Note that its images $(\seq_{\ubi^{(3)},\ubi^{(1)}}^{\sigma})^{*}(x_{k}(\sd^{(3)}))$
and $(\seq_{\ubi^{(2)},\ubi^{(1)}}^{\sigma})^{*}(\seq_{\ubi^{(3)},\ubi^{(2)}}^{\sigma})^{*}(x_{k}(\sd^{(3)}))$
are quantum cluster variables of $\clAlg^{(1)}$. They have the same
degree by the result at the classical level. Therefore, they are the
same quantum cluster variable. The desired claim follows.

\end{proof}

\subsection{Interval variables and $T$-systems}

Let $\ueta$ denote a word of letters from $J$ and denote $\rsd:=\rsd(\ueta)$.
Then the seed $\rsd$ is injective-reachable, see \cite[Proposition 8.7]{qin2023freezing}
or \cite{shen2021cluster}. Moreover, $\rsd$ has a distinguished
green to red sequence $\Sigma$ defined as below, which was introduced in \cite[Section 4]{shen2021cluster}.

For $j\leq k\in[1,l]\simeq I(\ueta)$ such that $\eta_{j}=\eta_{k}$,
denote $\seq_{[j,k]}:=\mu_{k}\cdots\mu_{j[1]}\mu_{j}$. Define $\Sigma_{k}:=\seq_{[k^{\min},k^{\min}[o_{+}(k)-1]]}$,
where $\Sigma_{k^{\max}}$ are understood as the identity. Then $\Sigma$
is defined as $\Sigma:=\Sigma_{l}\cdots\Sigma_{2}\Sigma_{1}$.

The cluster variables of the seeds $\rsd'$ appearing along the mutation
sequence $\Sigma$ starting from $\rsd$ are called the interval variables
of $\clAlg(\rsd)$, and the cluster variables of the form $x_{k^{\min}}(\rsd')$
are called the fundamental variables of $\clAlg(\rsd)$. More precisely,
we parameterize them as follows.

For any $r\in[0,l]$, $a\in J$, we define $r_{a}:=O([1,r];a)$ and
$\rsd\{r\}:=\Sigma_{r}\cdots\Sigma_{1}\rsd$, where $\rsd\{0\}:=\rsd$.
Then we define the interval variables $W_{[\binom{a}{r_{a}},\binom{a}{r_{a}+d}]}(\rsd):=x_{\binom{a}{d}}(\rsd\{r\})$
for any $0\leq d<O(a)-r_{a}$. Equivalently, for any $j=j^{\min}\in I$,
$0\leq d<O(\eta_{j})-r_{\eta_{j}}$, we define the interval variable
$W_{[j[r_{\eta_{j}}],j[r_{\eta_{j}}+d]]}(\rsd):=x_{j[d]}(\rsd\{r\})$.
The fundamental variables are defined as $W_{\binom{a}{d}}:=W_{[\binom{a}{d},\binom{a}{d}]}$,
$\forall\binom{a}{d}\in I(\ueta)$, or, equivalently, $W_{k}:=W_{[k,k]}$,
$\forall k\in[1,l]$.

Denote $\beta_{[j,k]}:=\deg^{\rsd}W_{[j,k]}$ and $\beta_{k}:=\deg^{\rsd}W_{k}$.
By \cite[Lemma 8.4]{qin2023freezing}, we have $\beta_{[j,k]}=f_{k}-f_{j[-1]}$,
where we understand $f_{\pm\infty}=0$.

\begin{Eg}\label{eg:A2-interval-variable}

Take $C=\left(\begin{array}{cc}
2 & -1\\
-1 & 2
\end{array}\right)$ and $\uc=(1,2)$. Consider $\rsd(\uc^{3})$. In Figure \ref{fig:A2-fundamental-variables},
we draw the corresponding fundamental variables $W_{\binom{a}{d}}$
on the vertices $\binom{a}{d}\in I(\uc^{3})\simeq[1,6]$.
\begin{figure}[h]
\caption{Fundamental variables for $\rsd(1,2,1,2,1,2)$}
\label{fig:A2-fundamental-variables}

\subfloat[]{\begin{tikzpicture}  [scale=2,node distance=48pt,on grid,>={Stealth[round]},bend angle=45,      pre/.style={<-,shorten <=1pt,>={Stealth[round]},semithick},    post/.style={->,shorten >=1pt,>={Stealth[round]},semithick},  unfrozen/.style= {circle,inner sep=1pt,minimum size=12pt,draw=black!100,fill=red!100},  frozen/.style={rectangle,inner sep=1pt,minimum size=12pt,draw=black!75,fill=cyan!100},   point/.style= {circle,inner sep=1pt,minimum size=5pt,draw=black!100,fill=black!100},   boundary/.style={-,draw=cyan},   internal/.style={-,draw=red},    every label/.style= {black}] \node (v1) at (0,1) {$W_{\binom{1}{0}}$}; \node (v2) at (-0.5,0) {$W_{\binom{2}{0}}$}; \node (v3) at (-1,1) {$W_{\binom{1}{1}}$}; \node (v4) at (-1.5,0) {$W_{\binom{2}{1}}$}; \node (v5) at (-2,1) {$W_{\binom{1}{2}}$}; \node (v6) at (-2.5,0) {$W_{\binom{2}{2}}$};
\end{tikzpicture}}\hfill{}\subfloat[]{\begin{tikzpicture}  [scale=2,node distance=48pt,on grid,>={Stealth[round]},bend angle=45,      pre/.style={<-,shorten <=1pt,>={Stealth[round]},semithick},    post/.style={->,shorten >=1pt,>={Stealth[round]},semithick},  unfrozen/.style= {circle,inner sep=1pt,minimum size=12pt,draw=black!100,fill=red!100},  frozen/.style={rectangle,inner sep=1pt,minimum size=12pt,draw=black!75,fill=cyan!100},   point/.style= {circle,inner sep=1pt,minimum size=5pt,draw=black!100,fill=black!100},   boundary/.style={-,draw=cyan},   internal/.style={-,draw=red},    every label/.style= {black}] \node (v1) at (0,1) {$W_{1}$}; \node (v2) at (-0.5,0) {$W_{2}$}; \node (v3) at (-1,1) {$W_{3}$}; \node (v4) at (-1.5,0) {$W_{4}$}; \node (v5) at (-2,1) {$W_{5}$}; \node (v6) at (-2.5,0) {$W_{6}$};
\end{tikzpicture}}
\end{figure}

Identify $I(\uc^{3})$ with $[1,6]$. We have $\Sigma=\mu_{2}\mu_{1}(\mu_{4}\mu_{2})(\mu_{3}\mu_{1})=\Sigma_{4}\Sigma_{3}\Sigma_{2}\Sigma_{1}$,
where $\Sigma_{5},\Sigma_{6}$ are trivial. The interval variables,
including the fundamental variables, are $W_{[1,j]}=x_{j}$ for $j\in\{1,3,5\}$,
$W_{[2,k]}=x_{k}$ for $k\in\{2,4,6\}$, and 
\begin{align*}
W_{3} & =x_{1}^{-1}\cdot x_{3}\cdot(1+y_{1})=x_{1}^{-1}\cdot x_{3}+x_{1}^{-1}\cdot x_{2}\\
W_{4} & =x_{2}^{-1}\cdot x_{4}(1+y_{2}+y_{1}\cdot y_{2})=x_{2}^{-1}\cdot x_{4}+x_{1}^{-1}\cdot x_{2}^{-1}\cdot x_{3}+x_{1}^{-1}\\
W_{5} & =x_{3}^{-1}\cdot x_{5}\cdot(1+y_{3}+y_{2}\cdot y_{3})=x_{3}^{-1}\cdot x_{5}+x_{2}^{-1}\cdot x_{3}^{-1}\cdot x_{1}\cdot x_{4}+x_{2}^{-1}\\
W_{6} & =x_{4}^{-1}\cdot x_{6}\cdot(1+y_{4}+y_{3}\cdot y_{4})=x_{4}^{-1}\cdot x_{6}+x_{3}^{-1}\cdot x_{4}^{-1}\cdot x_{2}\cdot x_{5}+x_{3}^{-1}\cdot x_{1}\\
W_{[3,5]} & =x_{1}^{-1}\cdot x_{5}\cdot(1+y_{1}+y_{1}\cdot y_{3})=x_{1}^{-1}\cdot x_{5}+x_{1}^{-1}\cdot x_{3}^{-1}\cdot x_{2}\cdot x_{5}+x_{3}^{-1}\cdot x_{4}\\
W_{[4,6]} & =x_{2}^{-1}\cdot x_{6}\cdot(1+y_{2}+y_{1}\cdot y_{2}+y_{2}\cdot y_{4}+y_{1}\cdot y_{2}\cdot y_{4}+y_{1}\cdot y_{2}\cdot y_{3}\cdot y_{4}).
\end{align*}

\end{Eg}

\begin{Prop}[{\cite[Proposition 8.6]{qin2023analogs}}]\label{prop:T-systems}

$\forall r\in[0,l-1]$, denote $a=\eta_{r+1}$. The interval variables
satisfy the following equations, called the $T$-systems:
\begin{align}
W_{[\binom{a}{r_{a}},\binom{a}{r_{a}+s}]}*W_{[\binom{a}{r_{a}+1},\binom{a}{r_{a}+s+1}]}= & q^{\alpha}[W_{[\binom{a}{r_{a}+1},\binom{a}{r_{a}+s}]}*W_{[\binom{a}{r_{a}},\binom{a}{r_{a}+s+1}]}]\label{eq:T-systems}\\
 & +q^{\alpha'}[\prod_{\binom{b}{r_{b}+d}}W_{[\binom{b}{r_{b}},\binom{b}{r_{b}+d}]}^{-C_{ba}}],\nonumber 
\end{align}
where $\binom{b}{r_{b}+d}$ appearing satisfy $b\neq a$, $\binom{b}{r_{b}+d}<\binom{a}{r_{a}+s+1}<\binom{b}{r_{b}+d+1}$,
$\alpha=\Hf\lambda(\beta_{[\binom{a}{r_{a}},\binom{a}{r_{a}+s}]},\beta_{[\binom{a}{r_{a}+1},\binom{a}{r_{a}+s+1}]})$,
$\alpha'=\Hf\lambda(\beta_{[\binom{a}{r_{a}},\binom{a}{r_{a}+s}]},-\sum_{\binom{b}{r_{b}+d}}C_{ba}\beta_{[\binom{b}{r_{b}},\binom{b}{r_{b}+d}]})$,
and $\alpha>\alpha'$.

\end{Prop}
In the classical case $q=1$, the $T$-systems \eqref{eq:T-systems} were proved in \cite[Proposition 3.25(1)]{shen2021cluster}.

Equivalently, $\forall1\leq j\leq j[s]<l$, the $T$-systems can be
written as 
\begin{align}
W_{[j,j[s]]}*W_{[j[1],j[s+1]]} & =q^{\alpha}[W_{[j[1],j[s]]}*W_{[j,j[s+1]]}]+q^{\alpha'}[\prod W_{[i,i[d]]}^{-C_{\eta_{i},\eta_{j}}}],\label{eq:T-system-other}
\end{align}
where $[i,i[d]]$ appearing satisfy $i=i^{\min}[O([1,j-1];\eta_{i})]$,
$i[d]<j[s+1]<i[d+1]$.

\begin{Eg}

Let $\uc$ denote a Coxeter word and choose $\ueta=\uc^{2i}$. For
any $k\in[1,|J|]$, $r\in[0,2i-1]$, $s\in[0,2i-r-2]$, denote $a=c_{k}$
and $\binom{a}{r}=k+r|J|$ under the identification $I(\uc^{2i})\simeq[1,2i\cdot|J|]$.
Then the $T$-system takes the form
\begin{align}
W_{[\binom{a}{r},\binom{a}{r+s}]}*W_{[\binom{a}{r+1},\binom{a}{r+s+1}]}= & q^{\alpha}[W_{[\binom{a}{r+1},\binom{a}{r+s}]}*W_{[\binom{a}{r},\binom{a}{r+s+1}]}]\label{eq:T-systems-Coxeter}\\
 & +q^{\alpha'}[\prod_{h\in[1,k-1]}W_{[\binom{c_{h}}{r+1},\binom{c_{h}}{r+1+s}]}^{-C_{c_{h},a}}*\prod_{h\in[k+1,|J|]}W_{[\binom{c_{h}}{r},\binom{c_{h}}{r+s}]}^{-C_{c_{h},a}}],\nonumber 
\end{align}
where $\alpha,\alpha'\in\Q$. Equivalently, we can write 
\begin{align}
W_{[k[r],k[r+s]]}*W_{[k[1],k[s+1]]}= & q^{\alpha}[W_{[k[r+1],k[r+s]]}*W_{[k[r],k[r+s+1]]}]\label{eq:T-system-other-Coxeter}\\
 & +q^{\alpha'}[\prod_{h\in[1,k-1]}W_{[h[r+1],h[r+1+s]]}^{-C_{c_{h},a}}*\prod_{h\in[k+1,|J|]}W_{[h[r],h[r+s]]}^{-C_{c_{h},a}}].\nonumber 
\end{align}

\end{Eg}

\subsection{Standard bases and Kazhdan-Lusztig bases}

For any $w=(w_{1},\ldots,w_{l})\in\N^{l}$, we define the standard
monomial $\stdMod(w):=[W_{1}^{w_{1}}*\cdots*W_{l}^{w_{l}}]^{\rsd}$.
Let $<_{\lex}$ denote the lexicographical order on $\N^{l}$ and
$<_{\rev}$ the reverse lexicographical order, i.e., $w<_{\rev}u$
if and only if $w\op<_{\lex}u\op$.

\begin{Thm}[{\cite{qin2023freezing}}]\label{thm:dBS_PBW}

(1) The set $\stdMod:=\{\stdMod(w)|w\in\N^{[1,l]}\}$ is a $\kk$-basis
of $\bUpClAlg(\rsd)$, called the standard basis. Particularly, we
have $\bClAlg(\rsd)=\bUpClAlg(\rsd)$.

(2) The standard monomials satisfy the analog of the Levendorskii--Soibelman
straightening law:
\begin{align}
W_{k}W_{j}-q^{\lambda(\deg W_{k},\deg W_{j})}W_{j}W_{k} & \in\sum_{w\in\N^{[j+1,k-1]}}\kk W(w),\ \forall j\leq k.\label{eq:LS-law}
\end{align}

(3) Let $\{\can(w)|w\in\N^{l}]$ denote the Kazhdan-Lusztig basis
associated with $\stdMod$ sorted by $<_{\rev}$, i.e., $\can(w)$
satisfies 
\begin{align*}
\overline{\can(w)} & =\can(w),\\
\can(w) & =M(w)+\sum_{w'<_{\rev}w}b_{w'}M(w'),\ b_{w'}\in q^{-\Hf}\Z[q^{\Hf}].
\end{align*}
Then $\can(w)$ equals the common triangular basis element $\can_{\sum_{k=1}^{l}w_{k}\beta_{k}}$
of for $\bUpClAlg(\rsd)$. Moreover, the statement still holds if
when we replace $<_{\rev}$ by $<_{\lex}$.

\end{Thm}

\subsection{Cluster embeddings associated with subwords\label{subsec:Cluster-embeddings-subwords}}

Let $\ubi$ denote a signed word in $\pm J$. Choose $1\leq j\leq k\leq l$.
Define the new word $\ubi'=(\bi'_{1},\ldots,\bi'_{k-j+1})$ to be
$\ubi_{[j,k]}=(\bi_{j},\bi_{j+1},\ldots,\bi_{k})$. Denote $\dsd':=\dsd(\ubi')$,
$\ddI':=I(\dsd')$, $I'_{\ufv}:=I_{\ufv}(\dsd')$, and $I':=I(\rsd(\ubi'))$.

\begin{Def}\label{def:calibration-words}

Define the embedding $\iota_{\ubi,\ubi'}:\ddI'\hookrightarrow\ddI$,
such that $\iota\binom{a}{d}^{\ubi'}=\binom{a}{d+O([1,j-1];a)}^{\ubi}$\@.

\end{Def}

Abbreviate $\iota_{\ubi,\ubi'}$ by $\iota$. Note that $\iota(I'_{\ufv})\subset I_{\ufv}$,
$\iota(I')\subset I$. Under the identification $I'\simeq[1,k-j+1]$
and $I\simeq[1,l]$, we have $\iota(s)=s+j-1$, $\forall s\in[1,k-j+1]$.
In particular, $s<s'$ if and only if $\iota(s)<\iota(s')$ for $s,s'\in I'$.

\begin{Lem}[{\cite[Lemma 6.6]{qin2023analogs}}]\label{lem:calibration-word}

(1) $\iota$ is a cluster embedding from $\dsd(\ubi')$ to $\dsd(\ubi)$,
such that $\dsd(\ubi')$ is a good sub seed of $\dsd(\ubi)$ via $\iota$.

(2) $\iota$ restricts to a cluster embedding from the classical seed
$\rsd(\ubi')$ to $\rsd(\ubi)$, still denoted $\iota$. If $\ubi'=\ubi_{[1,k]}$,
$\rsd(\ubi')$ is further a good sub seed of $\rsd(\ubi)$ via $\iota$.

\end{Lem}

Let $\ueta$ denote a word in $J$.

\begin{Lem}[{\cite[Lemma 8.5]{qin2023analogs}}]\label{lem:embed-interval-variables}

The inclusion $\iota_{[j,k]}:\bClAlg(\rsd(\ueta_{[j,k]}))\hookrightarrow\bClAlg(\rsd(\ueta_{[j,l]},-\ueta_{[1,j-1]}\op))=\bClAlg(\rsd(\ueta))$
sends the interval variables $W_{[\binom{a}{d}^{\ueta_{[j,k]}},\binom{a}{d'}^{\ueta_{[j,k]}}]}(\rsd(\ueta_{[j,k]}))$
to $W_{[\binom{a}{d+O([1,j-1];a])},\binom{a}{d'+O([1,j-1];a])}]}(\rsd(\ueta))$.

\end{Lem}

\subsection{Cluster structures from quantum affine algebras\label{subsec:Cluster-structures-quantum-affine}}

\global\long\def\mybinom#1#2{\genfrac{\langle}{\rangle}{0pt}{}{#1}{#2}}%

Choose a Coxeter word $\uc=(c_{1},\ldots,c_{|J|})$. It determines
a skew-symmetric $J\times J$-matrix $B^{\Delta}$ with $B_{c_{j},c_{k}}^{\Delta}:=\sign(C_{c_{j},c_{k}})\in\{0,-1\}$
for $j<k$. A letter $c_{k}$ is called a sink if $C_{c_{k},c_{j}}\geq0$,
$\forall j<k$; it is called a source if $C_{c_{k},c_{j}}\geq0$,
$\forall j>k$. Let $J^{\pm}:=J^{\pm}(\uc)$ denote the set of sink
letters and the set of source letters, respectively. 

Following \cite{KimuraQin14}, define $\W:=\{\mybinom ad|a\in J,d\in\Z\}$.
We use $e_{\mybinom ad}$ to denote the $\mybinom ad$-th unit vector
of $\N^{\W}$.

For any sink letter $a\in J^{+}$, we define a new Coxeter word $\mu_{a}\uc:=(\uc\backslash\{a\},a)$,
where $\mu_{a}$ is called a mutation. Choose any $\uc$-adapted word
$\ugamma=(\gamma_{i})$, i.e., $\gamma_{i}$ is a sink of $\mu_{\gamma_{i-1}}\cdots\mu_{\gamma_{1}}\uc$,
$\forall i$. Define $\hbeta_{k}:=\hbeta_{k}(\ugamma):=e_{\mybinom{\gamma_{k}}{o_{-}(\gamma_{k})}}$.

Define $\epsilon_{c_{j},c_{k}}:=1$ if $j<k$, $\epsilon_{c_{j},c_{j}}=0$,
and $\epsilon_{c_{j},c_{k}}:=-1$ if $j>k$. Assume that we can choose
a $\Z$-valued function $\xi$ on $J$, called a height function,
such that $\xi_{c_{j}}=\xi_{c_{k}}+\epsilon_{c_{j},c_{k}}$ when $C_{c_{j},c_{k}}<0$
and $j<k$. Note that such a $\xi$ exists if the quiver associated
with $B^{\Delta}$ is a tree. Define $J_{\Z}(\xi):=\{(a,p)|a\in J,p\in\xi_{a}+2\Z\}$,
which will be identified with $\W$ such that $(a,\xi_{a}-2d)$ equals
$\mybinom ad$. Then we can identify $\N^{\W}$ with $\N^{J_{\Z}(\xi)}$
and $e_{\mybinom ad}$ with $e_{(a,\xi_{a}-2d)}$.

\begin{Eg}\label{eg:A2-121212}

Continue Example \ref{eg:A2-interval-variable}, where $C=\left(\begin{array}{cc}
2 & -1\\
-1 & 2
\end{array}\right)$, $\uc=(1,2)$, and $\ugamma=\uc^{3}=(1,2,1,2,1,2)$. We draw the
quiver for $\rsd(\ugamma)$ in Figure \ref{fig:quiver-aff}, where
the elements of $I(\rsd(\ugamma))$ are also denoted as $k\in[1,6]$,
$\mybinom ad\in\W$, or $(a,\xi_{a}-2d)\in J_{\Z}(\xi)$. The set
of sink letters is $J^{+}=\{1\}$, $\xi_{2}=\xi_{1}-1$, and $\mu_{1}\uc=(2,1)$.
Note that our quiver is opposite to that of \cite{HernandezLeclerc09}.

\begin{figure}[h]
\caption{The quiver for $\rsd(1,2,1,2,1,2)$}
\label{fig:quiver-aff}\subfloat[]{

\begin{tikzpicture}  [node distance=48pt,on grid,>={Stealth[round]},bend angle=45,      pre/.style={<-,shorten <=1pt,>={Stealth[round]},semithick},    post/.style={->,shorten >=1pt,>={Stealth[round]},semithick},  unfrozen/.style= {circle,inner sep=1pt,minimum size=12pt,draw=black!100,fill=red!100},  frozen/.style={rectangle,inner sep=1pt,minimum size=12pt,draw=black!75,fill=cyan!100},   point/.style= {circle,inner sep=1pt,minimum size=5pt,draw=black!100,fill=black!100},   boundary/.style={-,draw=cyan},   internal/.style={-,draw=red},    every label/.style= {black}] \node[unfrozen] (v1) at (0,1) {1}; \node[unfrozen] (v2) at (-0.5,0) {2}; \node[unfrozen] (v3) at (-1,1) {3}; \node[unfrozen] (v4) at (-1.5,0) {4}; \node[frozen] (v5) at (-2,1) {5}; \node[frozen] (v6) at (-2.5,0) {6}; \draw[<-]  (v1) edge (v2); \draw[<-]  (v2) edge (v3); \draw[<-]  (v3) edge (v4); \draw[<-]  (v4) edge (v5); \draw[<-,dashed]  (v5) edge (v6); \draw[<-]  (v5) edge (v3); \draw[<-]  (v3) edge (v1); \draw[<-]  (v6) edge (v4); \draw[<-]  (v4) edge (v2); \end{tikzpicture}}\subfloat[]{

\begin{tikzpicture}  [scale=1.2,node distance=48pt,on grid,>={Stealth[round]},bend angle=45,      pre/.style={<-,shorten <=1pt,>={Stealth[round]},semithick},    post/.style={->,shorten >=1pt,>={Stealth[round]},semithick},  unfrozen/.style= {circle,inner sep=1pt,minimum size=12pt,draw=black!100,fill=red!100},  frozen/.style={rectangle,inner sep=1pt,minimum size=12pt,draw=black!75,fill=cyan!100},   point/.style= {circle,inner sep=1pt,minimum size=5pt,draw=black!100,fill=black!100},   boundary/.style={-,draw=cyan},   internal/.style={-,draw=red},    every label/.style= {black}] \node (v1) at (0,1) {$\mybinom{1}{0}$}; \node (v2) at (-0.5,0) {$\mybinom{2}{0}$}; \node (v3) at (-1,1) {$\mybinom{1}{1}$}; \node (v4) at (-1.5,0) {$\mybinom{2}{1}$}; \node (v5) at (-2,1) {$\mybinom{1}{2}$}; \node (v6) at (-2.5,0) {$\mybinom{2}{2}$}; \draw[<-]  (v1) edge (v2); \draw[<-]  (v2) edge (v3); \draw[<-]  (v3) edge (v4); \draw[<-]  (v4) edge (v5); \draw[<-,dashed]  (v5) edge (v6); \draw[<-]  (v5) edge (v3); \draw[<-]  (v3) edge (v1); \draw[<-]  (v6) edge (v4); \draw[<-]  (v4) edge (v2); \end{tikzpicture}}\subfloat[]{ \begin{tikzpicture}  [scale=2,node distance=48pt,on grid,>={Stealth[round]},bend angle=45,      pre/.style={<-,shorten <=1pt,>={Stealth[round]},semithick},    post/.style={->,shorten >=1pt,>={Stealth[round]},semithick},  unfrozen/.style= {circle,inner sep=1pt,minimum size=12pt,draw=black!100,fill=red!100},  frozen/.style={rectangle,inner sep=1pt,minimum size=12pt,draw=black!75,fill=cyan!100},   point/.style= {circle,inner sep=1pt,minimum size=5pt,draw=black!100,fill=black!100},   boundary/.style={-,draw=cyan},   internal/.style={-,draw=red},    every label/.style= {black}] \node (v1) at (0,0.5) {$(1,\xi_1)$}; \node (v2) at (-0.5,0) {$(2,\xi_2)$}; \node (v3) at (-1,0.5) {$(1,\xi_1-2)$}; \node (v4) at (-1.5,0) {$(2,\xi_2-2)$}; \node (v5) at (-2,0.5) {$(1,\xi_1-4)$}; \node (v6) at (-2.5,0) {$(2,\xi_2-4)$}; \draw[<-]  (v1) edge (v2); \draw[<-]  (v2) edge (v3); \draw[<-]  (v3) edge (v4); \draw[<-]  (v4) edge (v5); \draw[<-,dashed]  (v5) edge (v6); \draw[<-]  (v5) edge (v3); \draw[<-]  (v3) edge (v1); \draw[<-]  (v6) edge (v4); \draw[<-]  (v4) edge (v2); \end{tikzpicture}}
\end{figure}

\end{Eg}

Assume $C$ to be of type $ADE$ from now on. Let $\qAff$ denote
the associated quantum affine algebra, where $\varepsilon\in\C^{\times}$
is not a root of unity. $\forall a\in J,\varepsilon'\in\C^{\times}$,
it has a finite dimensional simple module $L_{a,\varepsilon'}$ called
a fundamental module. We denote $L(e_{\mybinom ad}):=L_{a,\varepsilon^{\xi_{a}-2d}}$.
The composition factors of the tensor products of $L(e_{\mybinom ad})$
could be parameterized by $L(w)$, for $w\in\oplus_{\mybinom ad\in\W}\N e_{\mybinom ad}$.
See \cite{Nakajima01}\cite{Nakajima04}\cite{HernandezLeclerc09}
for more details.

Let $\cC_{\Z}(\xi)$ denote the monoidal category consisting of the
finite dimensional modules of $\qAff$ which have the composition
factors $L(w)$, $w\in\oplus_{\mybinom ad\in\W}\N e_{\mybinom ad}$.
Define $\cC_{\ugamma}(\xi)$ to be the monoidal subcategory of $\cC_{\Z}(\xi)$
whose objects have the composition factors $L(w)$, for $w\in\oplus_{k}\N\hbeta_{k}$.
For any $j\leq k$ such that $\gamma_{j}=\gamma_{k}$, $L(\hbeta_{[j,k]})$
are called the Kirillov--Reshetikhin modules, where we denote $\hbeta_{[j,k]}:=\hbeta_{j}+\hbeta_{j[1]}+\cdots+\hbeta_{k}$.
Its deformed Grothendieck ring, which is a $\kk$-algebra $K$, can
be constructed from graded quiver varieties \cite{Nakajima04}\cite{VaragnoloVasserot03}\cite[Sections 7, 8.4]{qin2017triangular}. 

By \cite[Theorem 8.4.3]{qin2017triangular}, there is a $\kk$-algebra
isomorphism $\kappa:\bClAlg(\rsd(\ugamma))\simeq K$, such that $\kappa(x_{i}(\rsd(\ugamma)))=[L(\hbeta_{[i^{\min},i]})]$,
where the quantization of $K$ and $\rsd(\ugamma)$ are chosen as
in \cite[Section 7.3]{qin2017triangular}. Recall that we have $\beta_{[j,k]}:=\deg^{\rsd(\ugamma)}W_{[j,k]}=f_{k}-f_{j[-1]}$.

\begin{Thm}[{\cite[Theorem 1.2.1(II)]{qin2017triangular}}]\label{thm:basis-for-HL}

The cluster algebra $\bClAlg(\rsd(\ugamma))$ has the common triangular
basis $\can$. Moreover, $\cC_{\ugamma}(\xi)$ categorifies $(\bClAlg(\rsd(\ugamma)),\can)$
such that $\kappa\can_{m}=[L(w)]$, where $w=\sum_{i=1}^{l}w_{i}\hbeta_{i}$
and $m=(m_{i})_{i\in[1,l]}$ satisfies $\sum_{i}m_{i}\beta_{[i^{\min},i]}=\sum_{i}w_{i}\beta_{i}$. 

\end{Thm}

\section{Based cluster algebras of infinite rank\label{sec:Based-cluster-algebras-infinite-ranks}}

In this section, we consider cluster algebras of infinite rank which
appear as colimits. See \cite{gratz2015cluster} for a general discussion.
We will see that results for the finite rank case might be extended
to the infinite rank case. 

\subsection{Extension of cluster algebras to infinite rank\label{subsec:Extension-to-infinite}}

Assume that we have a chain of seeds $(\sd_{i})_{i\in\N}$ such that
$\sd_{i}$ is a good seed of $\sd_{i+1}$ via a cluster embedding
$\iota_{i}$ (Section \ref{subsec:Sub-seeds}) and $I_{\infty}$ is
the colimit of $\cdots I(\sd_{i})\xhookrightarrow{\iota_{i}}I(\sd_{i+1})\xhookrightarrow{\iota_{i+1}}I(\sd_{i+2})\cdots$.
We will omit $\iota_{i}$ and view $I(\sd_{i})$ as subsets of $I_{\infty}$.
Then we have inclusions:
\begin{align*}
\bClAlg_{0} & \subset\bClAlg_{1}\subset\cdots
\end{align*}
where $\bClAlg_{i}:=\bClAlg(\sd_{i})$ and $x_{j}(\sd_{i})$ are identified
with $x_{j}(\sd_{i+1})$. Define the seed $\sd_{\infty}$, also denoted
$\cup_{i}\sd_{i}$, such that $I(\sd_{\infty})=I_{\infty}$, $\tB(\sd_{\infty})_{j,k}=\tB(\sd_{i})_{j,k}$,and
$x_{j}(\sd_{\infty})=x_{j}(\sd_{i})$ for any $j,k\in I(\sd_{i})\subset I_{\infty}$.
In the quantum case, we assume the above inclusions hold and choose
the quantization matrix for $\Lambda(\sd_{\infty})$ such that $\Lambda(\sd_{\infty})_{j,k}=\Lambda(\sd_{i})_{j,k}$
for $j,k\in I(\sd_{i})$. 

Recall that $\bLP(\sd_{\infty})=\cup\bLP(\sd_{i})$. We define its
subalgebra $\bClAlg_{\infty}:=\cup_{i}\bClAlg_{i}$

\begin{Lem}

The cluster algebra $\bClAlg(\sd_{\infty})$ coincides with $\bClAlg_{\infty}=\cup_{i}\bClAlg_{i}$
as a subalgebra of $\bLP(\sd_{\infty})$.

\end{Lem}

\begin{proof}

Note that every cluster variable of $\bClAlg(\sd_{i})$ is a cluster
variable of $\bClAlg_{\infty}$. Therefore, we have $\bClAlg_{\infty}\subset\bClAlg(\sd_{\infty})$. 

Conversely, take any element $z$ of $\bClAlg(\sd_{\infty})$. It
is a polynomial of the cluster variables in $\seq\sd_{\infty}$ for
some mutation sequence $\seq$. These cluster variables are also cluster
variables of $\seq\sd_{i}$ for some $\sd_{i}$. Therefore, $z$ is
also contained in $\bClAlg(\sd_{i})$. We deduce that $\bClAlg(\sd_{\infty})\subset\bClAlg_{\infty}$.

\end{proof}

If each $\bClAlg_{i}$ has a $\kk$-basis $\base_{i}$ such that the
above inclusions restrict to $\base_{i}\subset\base_{i+1}$, define
the colimit $\base_{\infty}:=\cup\base_{i}$. It is a $\kk$-basis
of $\bClAlg_{\infty}$. In this case, if $\base_{i+1}$ is the common
triangular basis, then $\base_{i}=\base_{i+1}\cap\bClAlg_{i}$ is
the common triangular basis as well, see \cite[Theorem 3.16, Corollary 3.18]{qin2023analogs}.
We call $\base_{\infty}$ the common triangular basis for $\bClAlg_{\infty}$
if all $\base_{i}$ are common triangular bases.

Next, assume there is a monoidal category $\cT$ with monoidal full
subcategories $(\cT_{i})_{i\in\N}$, such that $\cT_{i}\subset\cT_{i+1}$
and $\cT=\cup_{i}\cT_{i}$. Assume the associated deformed Grothendieck
rings satisfy $K_{i}\subset K_{i+1}$ and choose $K:=\cup K_{i}$
to be the deformed Grothendieck ring for $\cT$. Further assume that
each $\cT_{i}$ categorifies $\bClAlg_{i}$ (resp. categories $(\bClAlg_{i},\base_{i})$
under the condition $\base_{i}=\base_{i+1}\cap\bClAlg_{i}$), such
that we have the commutative diagram $\begin{array}{ccc}
\bClAlg_{i+1} & \simeq & K_{i+1}\\
\cup &  & \cup\\
\bClAlg_{i} & \simeq & K_{i}
\end{array}$. Then $\cT$ categorifies $\bClAlg_{\infty}$ (resp. $(\bClAlg_{\infty},\base_{\infty})$)
in the sense of Definition \ref{def:categorification}.

\subsection{Extension of upper cluster algebras to infinite rank}

In this subsection, we make the following assumption:

\begin{Assumption}\label{assumption:finite-full-rk}

Assume that $|I(\sd_{i})<\infty|$ and $\sd_{i}$ are of full rank
for all $i$.

\end{Assumption}

By Proposition \ref{prop:inclusion-good-sub-cl-alg}, we have a chain
of inclusions of upper cluster algebras $\cdots\subset\bUpClAlg(\sd_{i})\subset\bUpClAlg(\sd_{i+1})\subset\cdots$.
Denote the colimit $\bUpClAlg_{\infty}:=\cup_{i}\bUpClAlg(\sd_{i})$.

\begin{Lem}\label{lem:bLP-local}

For any seed $\sd_{i}$, we have $\cF(\sd_{i})\cap\bLP(\sd_{\infty})=\bLP(\sd_{i})$.

\end{Lem}

\begin{proof}

Take any $z\in\cF(\sd_{i})\cap\bLP(\sd_{\infty})$. It takes the form
$Q^{-1}*P\in\cF(\sd_{i})$ for $Q\in\kk[x_{k}]_{k\in I_{\ufv}(\sd_{i})}$,
$P\in R_{i}:=\kk[x_{h}]_{h\in I(\sd_{i})}$. On the other hand, since
$z\in\bLP(\sd_{\infty})$, we can find $\sd_{j}$ with $j\geq i$,
such that $z\in\bLP(\sd_{j})$. Write the reduced form $z=P'*(Q')^{-1}$
for $Q'$ a monomial in $x_{k'}$, where $k'\in I_{\ufv}(\sd_{j})$,
and $P'\in R_{j}:=\kk[x_{h'}]_{h'\in I(\sd_{j})}$, such that $Q'$
and $P'$ do not have common divisors. Note that $R_{i}\subset R_{j}$.

Since the cluster variables $x_{h'}$ $q$-commute, we have $Q*P'=P*Q'=Q'*\tP$
in $R_{j}$ for some $\tP\in R_{i}$. Note that $x_{k'}$, where $k'\in I_{\ufv}(\sd_{j})$,
are prime elements in $R_{j}$ in the sense that $R_{j}/(x_{k'})$
are domains. Recall that $Q'$ and $P'$ are coprime. Inductively,
we deduce that $Q$ is divisible by $Q'$, i.e., we can write $Q=Q'*Q''$
in $R_{j}$. It follows that $Q'$ is a monomial of $x_{k}$, $k\in I_{\ufv}(\sd_{i})$. 

Now we have $Q''*P'=\tP$. Since $\tP\in R_{i}$, we must have $Q'',P'\in R_{i}$
as well. Therefore, $z=(Q')^{-1}*P'$ belongs to $\bLP(\sd_{i})$.

\end{proof}

\begin{Lem}\label{lem:extend-bUpClAlg}

We have $\bUpClAlg(\sd_{\infty})=\bUpClAlg_{\infty}$.

\end{Lem}

\begin{proof}

(i) Take any $i\in\N$ and work in the skew-fields of fractions $\cF(\sd_{i})\subset\cF(\sd_{\infty})$.
We claim that $\bUpClAlg(\sd_{i})\subset\bUpClAlg(\sd_{\infty})$.
To see this, take any mutation sequence $\seq$ on $\sd_{i}$. Then
we have $\bLP(\seq\sd_{i})\subset\bLP(\seq\sd_{\infty})$. The claim
follows.

As a consequence, we deduce that $\bUpClAlg_{\infty}\subset\bUpClAlg(\sd_{\infty})$.

(ii) Conversely, Take any $z\in\bUpClAlg(\sd_{\infty})$. We can find
$i$ such that $z\in\bLP(\sd_{i})$.

Take any mutation sequence $\seq$ on $\sd_{i}$. We have $z\in\cF(\seq\sd_{i})$.
On the other hand, $z\in\bUpClAlg(\sd_{\infty})\subset\bLP(\seq\sd_{\infty})$.
By Lemma \ref{lem:bLP-local}, we have $z\in\bLP(\seq\sd_{i})$. Since
$\seq$ is arbitrary, we have $z\in\bUpClAlg(\sd_{i})$.

\end{proof}

\begin{Cor}\label{cor:extend-bA-equal-bU}

Under Assumption \ref{assumption:finite-full-rk}, if $\bClAlg_{i}=\bUpClAlg_{i}$
for $i$ large enough, we have $\bClAlg(\sd_{\infty})=\bUpClAlg(\sd_{\infty})$.

\end{Cor}

We extend Theorem \ref{thm:star-fish} to infinite rank.

\begin{Thm}\label{thm:start-fish-inf}

Under Assumption \ref{assumption:finite-full-rk}, we have $\bUpClAlg(\sd_{\infty})=\bLP(\sd_{\infty})\cap(\cap_{k\in I_{\ufv}(\sd_{\infty})}\bLP(\mu_{k}\sd_{\infty}))$.

\end{Thm}

\begin{proof}

Note that we have $\bUpClAlg(\sd_{\infty})\subset\bLP(\sd_{\infty})\cap(\cap_{k\in I_{\ufv}(\sd_{\infty})}\bLP(\mu_{k}\sd_{\infty}))$
by definition.

Conversely, take any $z\in\bLP(\sd_{\infty})\cap(\cap_{k\in I_{\ufv}(\sd_{\infty})}\bLP(\mu_{k}\sd_{\infty}))$.
Note that $z\in\bLP(\sd_{i})$ for some $\sd_{i}$. For any $k\in I_{\ufv}(\sd_{i})$,
we have $z\in\cF(\mu_{k}\sd_{i})$. Since $z$ is also contained in
$\bLP(\mu_{k}\sd_{\infty})$, it belongs to $\bLP(\mu_{k}\sd_{i})$
by Lemma \ref{lem:bLP-local}. Therefore, $z$ is contained in $\bLP(\sd_{i})\cap(\cap_{k\in I_{\ufv}(\sd_{i})}\bLP(\mu_{k}\sd_{i}))$,
which equals $\bUpClAlg_{i}$ under Assumption \ref{assumption:finite-full-rk}
(Corollary \ref{cor:star-fish-bar}). We deduce that $\bLP(\sd_{\infty})\cap(\cap_{k\in I_{\ufv}(\sd_{\infty})}\bLP(\mu_{k}\sd_{\infty}))$
is contained in $\bUpClAlg_{\infty}$, which coincides with $\bUpClAlg(\sd_{\infty})$
by Lemma \ref{lem:extend-bUpClAlg}.

\end{proof}

Categorifications of based upper cluster algebras are defined similar
to those in Section \ref{subsec:Extension-to-infinite}.

\subsection{Quantization}

Let $\sd$ denote a seed and $\sd'$ a good sub seed of $\sd$. Let
$\Lambda'$ denote any quantization matrix for $\sd'$. We have the
following natural and interesting question:
\begin{center}
\emph{When can $\Lambda'$ be extended to a quantization matrix $\Lambda$
for $\sd$?}
\par\end{center}

We will provide an answer for the following very special case, which
will suffice for our purpose in Section \ref{sec:Cluster-algebras-from-shifted}.

Let $\sd$ denote a seed of finite rank. Assume that we have the partition
$I=I_{1}\sqcup I_{2}\sqcup I_{3}$, such that $I_{\ufv}=I_{1}\sqcup I_{2}$,
$I_{\fv}=I_{3}$. Let $D$ denote the diagonal matrix whose diagonal
entries are $d_{j}$, $j\in I$. Using block matrices with respect
to $I_{i}$, we denote $\tB D=\left(\begin{array}{cc}
B_{11} & B_{12}\\
B_{21} & B_{22}\\
B_{31} & B_{32}
\end{array}\right)$, where $B_{ij}$ denotes $(\tB D)_{I_{i}\times I_{j}}$. Note that
$\left(\begin{array}{cc}
B_{11} & B_{12}\\
B_{21} & B_{22}
\end{array}\right)$ is skew-symmetric.

\begin{Def}[{\cite[Definition 2.1.10]{qin2020dual}}]

An $I_{\ufv}\times I_{\ufv}$ matrix $B=(b_{ij})$ is called connected
if, $\forall i,j\in I_{\ufv}$, there exists finitely many $i_{s}\in I_{\ufv}$,
$s\in[0,l]$, such that $i_{0}=i$, $i_{l}=j$, and $b_{i_{s},i_{s+1}}\neq0$,
$\forall s\in[0,l-1]$.

\end{Def}

Assume that $\tB_{I_{\ufv}\times I_{\ufv}}$ and $B_{11}$ are connected.
Assume that $B_{31}=0$. Then the seed $\sd'$ obtained from $\sd$
by freezing $I_{2}$ and then removing the frozen vertices $I_{3}$
is a good sub seed of $\sd$. Note that $\tB(\sd')D_{(I_{1}\sqcup I_{2})\times(I_{1}\sqcup I_{2})}=\left(\begin{array}{c}
B_{11}\\
B_{21}
\end{array}\right)$. Assume that $\sd'$ has a quantization matrix $\Lambda':=\left(\begin{array}{cc}
\Lambda_{11} & \Lambda_{12}\\
\Lambda_{21} & \Lambda_{22}
\end{array}\right)$. Since $B_{11}$ is connected, we have $\Lambda'\left(\begin{array}{c}
B_{11}\\
B_{21}
\end{array}\right)=\alpha\left(\begin{array}{c}
-\Id_{I_{1}}\\
0
\end{array}\right)$ for some $\alpha\in\Q_{>0}$, where $\Id_{I_{1}}$ is the identity
matrix on $I_{1}$, see \cite[Section 2.1]{qin2020dual}.

\begin{Lem}\label{lem:extend-quantization}

Further assume $B_{32}$ to be of full rank and $|I_{2}|=|I_{3}|$.
Then we can uniquely extend $\Lambda'$ to a quantization matrix $\Lambda$
for $\sd$.

\end{Lem}

\begin{proof}

Denote $\Lambda=\left(\begin{array}{ccc}
\Lambda_{11} & \Lambda_{12} & Z_{13}\\
\Lambda_{21} & \Lambda_{22} & Z_{23}\\
-Z_{13}^{T} & -Z_{23}^{T} & Z_{33}
\end{array}\right)$ for some $I_{i}\times I_{j}$ matrices $Z_{ij}$. Since $\tB_{I_{\ufv}\times I_{\ufv}}$
is connected, $\Lambda$ is a quantization matrix if and only if we
have $\Lambda\tB D=\alpha\left(\begin{array}{cc}
-\Id_{I_{1}} & 0\\
0 & -\Id_{I_{2}}\\
0 & 0
\end{array}\right)$. Combining with $\Lambda'\left(\begin{array}{c}
B_{11}\\
B_{21}
\end{array}\right)=\alpha\left(\begin{array}{c}
-\Id_{I_{1}}\\
0
\end{array}\right)$ and $B_{31}=0$, this is equivalent to the conditions 
\begin{align*}
\left(\begin{array}{cc}
\Lambda_{11} & \Lambda_{12}\end{array}\right)\left(\begin{array}{c}
B_{12}\\
B_{22}
\end{array}\right)+Z_{13}B_{32} & =0\\
\left(\begin{array}{cc}
\Lambda_{21} & \Lambda_{22}\end{array}\right)\left(\begin{array}{c}
B_{12}\\
B_{22}
\end{array}\right)+Z_{23}B_{32} & =-\alpha\Id_{I_{2}}\\
\left(\begin{array}{cc}
-Z_{13}^{T} & -Z_{23}^{T}\end{array}\right)\left(\begin{array}{c}
B_{12}\\
B_{22}
\end{array}\right)+Z_{33}B_{32} & =0\\
\left(\begin{array}{cc}
-Z_{13}^{T} & -Z_{23}^{T}\end{array}\right)\left(\begin{array}{c}
B_{11}\\
B_{21}
\end{array}\right) & =0
\end{align*}

Since $|I_{2}|=|I_{3}|$ and $B_{32}$ is of full rank, the values
of $Z_{ij}$ are uniquely determined by the first three equations.
It remains to verify the last one, which is equivalent to $\left(\begin{array}{cc}
B_{11}^{T} & B_{21}^{T}\end{array}\right)\left(\begin{array}{c}
-Z_{13}\\
-Z_{23}
\end{array}\right)B_{32}=0$. 

Substituting $Z_{ij}B_{32}$ by the first two equations and using
$\left(\begin{array}{cc}
B_{11}^{T} & B_{21}^{T}\end{array}\right)\Lambda'=\left(\begin{array}{cc}
\alpha\Id_{I_{1}} & 0\end{array}\right)$, we can calculate the left hand side as follows: 
\begin{align*}
\left(\begin{array}{cc}
B_{11}^{T} & B_{21}^{T}\end{array}\right)\left(\begin{array}{c}
-Z_{13}\\
-Z_{23}
\end{array}\right)B_{32}= & \left(\begin{array}{cc}
B_{11}^{T} & B_{21}^{T}\end{array}\right)(\left(\begin{array}{cc}
\Lambda_{11} & \Lambda_{12}\\
\Lambda_{21} & \Lambda_{22}
\end{array}\right)\left(\begin{array}{c}
B_{12}\\
B_{22}
\end{array}\right)+\left(\begin{array}{c}
0\\
\alpha\Id_{I_{2}}
\end{array}\right))\\
= & \alpha B_{12}+\alpha B_{21}^{T}\\
= & 0.
\end{align*}

\end{proof}

\section{Applications: infinite rank cluster algebras from quantum affine
algebras\label{sec:Infinite-rank-cluster-q-aff}}

Following the convention in Section \ref{subsec:Cluster-structures-quantum-affine},
we let $C$ denote a $J\times J$ generalized Cartan matrix and choose
a Coxeter word $\uc$. Define the classical seeds $\sd_{i}:=\rsd(\uc^{i})$,
$i\in\N$. We will quantize and extend $\bClAlg(\sd_{i})=\bUpClAlg(\sd_{i})$
to infinite rank. We will also introduce seeds $\ssd_{i}$ and $\usd_{i}$
and extend $\bClAlg(\ssd_{i})=\bClAlg(\usd_{i})=\bUpClAlg(\ssd_{i})=\bUpClAlg(\usd_{i})$
to infinite rank. Note that we still have $\bClAlg=\bUpClAlg$ after
extension by Corollary \ref{cor:extend-bA-equal-bU}, i.e., Theorem
\ref{thm:intro-A-equal-U-inf-rank} is true.

When $C$ is of finite type, we will compare our cluster algebras
and interval variables with the those arising from the quantum virtual
Grothendieck rings $\mathfrak{K}_{q}(\mathfrak{g})$, $\mathfrak{K}_{q,\xi}(\mathfrak{g})$,
and the KR-polynomials in \cite{jang2023quantization}. We will not
go into details with their constructions and instead refer the reader
to \cite{jang2023quantization} for further information.

\subsection{Limits of Coxeter words\label{subsec:Limits-of-Coxeter}}

View $\uc^{i}$ as $(\uc^{i+1})_{[1,|J|\cdot i]}$. Then $\sd_{i}$
is a good sub seed of $\sd_{i+1}$ via the cluster embedding $\iota'$
sending sending $\binom{a}{d}^{\sd_{i}}$ to $\binom{a}{d}^{\sd_{i+1}}$
(Definition \ref{def:calibration-words}). We identify $\binom{a}{d}^{\sd_{i}}$
and $\binom{a}{d}^{\sd_{i+1}}$ with $\mybinom ad\in\W$. Then we
obtain the colimit $I_{\infty}=\{\mybinom ad|a\in J,d\in\N\}$ and
$\sd_{\infty}:=\cup_{i}\sd_{i}$. We make any of the following choices
to quantize this chain of good sub seeds:
\begin{enumerate}
\item When $C$ is symmetric, we can quantize $\sd_{i}$ using the bilinear
form $\cN$ in \cite[(46), Remark 7.3.1]{qin2017triangular}, where
our $\mybinom ad$ is denoted by $(a,-2d)$, see Section \ref{sec:Skew-symmetric-bilinear-forms}. 
\item When $C$ is not of finite type, $\uc^{i}$ is reduced. Following
the convention in \cite{Kimura10}, we can quantize $\sd_{i}:=\rsd(\uc^{i})$
using the bilinear form $N_{\ow}$ in \cite[Section 8.1]{qin2020dual}
such that $\bClAlg(\sd_{i})$ is identified with the quantum unipotent
subgroup $\qO[N(w_{\uc^{i}})]$.
\item When $C$ is of arbitrary finite type, choose any height function
$\xi$. Recall that we can identify our vertex $\mybinom ad\in\W$
with $(a,\xi_{a}-2d)\in J_{\Z}(\xi)$. Let $\ucN$ denote the bilinear
form on $\oplus_{(a,p)\in J_{\Z}(\xi)}\Z e_{(a,p)}$ such that $\ucN(e_{(a,p)},e_{(b,s)})$
is given by \cite[(4.1)]{jang2023quantization}, see Section \ref{sec:Skew-symmetric-bilinear-forms}
for details. We could quantize $\sd_{i}$ by choosing $\Lambda(\sd_{i})_{\mybinom ad,\mybinom bs}:=\ucN(\sum_{j=0}^{d}e_{(a,\xi_{a}-2j)},\sum_{j=0}^{s}e_{(b,\xi_{b}-2j)})$,
see \cite[Theorem 8.1]{jang2023quantization}\cite{kashiwara2023q}\cite{fujita2023isomorphisms}.
\\
Note that our matrix $\tB(\sd_{\infty})$ is identified with the $B$-matrix
$^{\cS}\tB=(b_{u,v})_{u,v\in\cS}$, $\cS=\{(a,\xi_{a}-2d)|d\in\N\}$,
in \cite[(2.14)]{jang2023quantization}. In this way, we could identify
our quantum seed $\sd_{\infty}$ with the quantum seed for the quantum
cluster algebra $\mathfrak{K}_{q,\xi}(\mathfrak{g})$ in \cite[Theorem 8.8]{jang2023quantization},
because they share the same $B$-matrix $\tB(\sd_{\infty})$ and quantization
matrix $\Lambda(\sd_{\infty})$. Then our initial cluster variables
$x_{\mybinom ad}(\sd_{\infty})=W_{[\mybinom a0,\mybinom ad]}(\sd_{i})$
are identified with the initial variables (KR-polynomials) of $\mathfrak{K}_{q,\xi}(\mathfrak{g})$,
denoted $F_{q}(\underline{m}^{(a)}[\xi_{a}-2d,\xi_{a}])$.
\end{enumerate}
So we obtain inclusions of (quantum) cluster algebras $\iota':\bClAlg(\sd_{i})\subset\bClAlg(\sd_{i+1})$,
identifying $W_{[\mybinom ad,\mybinom a{d'}]}(\sd_{i})$ with $W_{[\mybinom ad,\mybinom a{d'}]}(\sd_{i+1})$
(Lemma \ref{lem:embed-interval-variables}). Let $\stdMod_{i}$ and
$\can_{i}$ denote the standard basis and the common triangular basis
of $\bClAlg(\sd_{i})$, then $\stdMod=\cup_{i}\stdMod_{i}$ and $\can=\cup_{i}\can_{i}$
are bases of $\bClAlg(\sd_{\infty})$, called the standard basis and
the common triangular basis, respectively.

\begin{Rem}\label{rem:semi-inf-q-affine-categorification}

When $C$ is of type $ADE$, by \cite{qin2017triangular}, $(\bClAlg(\sd_{i}),\can_{i})$
is categorified by $\cC_{\uc^{i}}(\xi)$, see Section \ref{subsec:Cluster-structures-quantum-affine}.
So $(\bClAlg(\sd_{\infty}),\can_{\infty})$ is categorified by $\cC^{-}(\xi):=\cup_{i}\cC_{\uc^{i}}(\xi)$,
where $\cC^{-}(\xi)$ was introduced in \cite{hernandez2013cluster}.

\end{Rem}

\begin{Eg}\label{eg:Lambda}

Let us continue Example \ref{eg:A2-121212}, where $\tB=\left(\begin{array}{cccc}
0 & -1 & 1 & 0\\
1 & 0 & -1 & 1\\
-1 & 1 & 0 & -1\\
0 & -1 & 1 & 0\\
0 & 0 & -1 & 1\\
0 & 0 & 0 & -1
\end{array}\right)$. For $C(z):=\left(\begin{array}{cc}
z+z^{-1} & -1\\
-1 & z+z^{-1}
\end{array}\right)$, its inverse is $\tC=\left(\begin{array}{cc}
(z-z^{5})\sum_{d\geq0}z^{6d} & (z^{2}-z^{4})\sum_{d\geq0}z^{6d}\\
(z^{2}-z^{4})\sum_{d\geq0}z^{6d} & (z-z^{5})\sum_{d\geq0}z^{6d}
\end{array}\right)=:\sum_{m\in\Z}\tC(m)z^{m}$, where $\tC(m)$ are $\Z$-matrices. Note that $\tC^{T}=\tC$ in
this example.

We could compute the skew-symmetric bilinear form $\cN$ defined in
Section \ref{sec:Skew-symmetric-bilinear-forms} or, equivalently,
the matrix $\cN=\sum_{m\in\Z}\cN(m)z^{m}$ where $\cN(m)$ are $\Z$-matrices.
We have $\cN_{ab}(p-s):=\cN(e_{(a,p)},e_{(b,s)}):=\tC_{ab}(p-1-s)-\tC_{ba}(s-1-p)-\tC_{ab}(p+1-s)+\tC_{ba}(s+1-p)$.
We obtain the matrix $\cN=z\tC-z^{-1}\tC^{T}(z^{-1})-z^{-1}\tC+z\tC^{T}(z^{-1})$,
which satisfies $\cN(z^{-1})^{T}=-\cN$. Explicit calculation shows
that 
\begin{align*}
\cN= & \left(\begin{array}{cc}
-1+z^{2}+z^{4}-z^{6} & -z+2z^{3}-z^{5}\\
-z+2z^{3}-z^{5} & -1+z^{2}+z^{4}-z^{6}
\end{array}\right)\sum_{d\geq0}z^{6d}\\
 & +\left(\begin{array}{cc}
1-z^{-2}-z^{-4}+z^{-6} & z^{-1}-2z^{-3}+z^{-5}\\
z^{-1}-2z^{-3}+z^{-5} & 1-z^{-2}-z^{-4}+z^{-6}
\end{array}\right)\sum_{d\geq0}z^{-6d}.
\end{align*}
Choose $\xi_{1}=0$, $\xi_{2}=-1$, and identify $(a,\xi_{a}-2d)$
with $a+2d\in[1,6]$, where $a\in[1,2]$, $d\in[0,2]$. We obtain
the following matrix
\begin{align*}
(\cN(e_{j},e_{k}))_{j,k\in[1,6]} & =\left(\begin{array}{cccccc}
0 & -1 & 1 & 2 & 1 & -1\\
1 & 0 & -1 & 1 & 2 & 1\\
-1 & 1 & 0 & -1 & 1 & 2\\
-2 & -1 & 1 & 0 & -1 & 1\\
-1 & -2 & -1 & 1 & 0 & -1\\
1 & -1 & -2 & -1 & 1 & 0
\end{array}\right),
\end{align*}
\begin{align*}
(\Lambda_{jk})_{j,k\in[1,6]} & =(\cN(\beta_{[j^{\min},j]},\beta_{[k^{\min},k]}))_{j,k\in[1,6]}=\left(\begin{array}{cccccc}
0 & -1 & 1 & 1 & 2 & 0\\
1 & 0 & 0 & 1 & 2 & 2\\
-1 & 0 & 0 & 1 & 2 & 2\\
-1 & -1 & -1 & 0 & 0 & 2\\
-2 & -2 & -2 & 0 & 0 & 0\\
0 & -2 & -2 & -2 & 0 & 0
\end{array}\right).
\end{align*}
We could verify that $(\Lambda\tB)_{jk}=-2\delta_{j,k}$.

\end{Eg}

\subsection{Limits of signed words\label{subsec:Limits-of-signed}}

Denote $\sd_{2i}:=\rsd(\uc^{2i})$ as before. Denote $\sd'_{2i}:=\rsd(\uc^{2i-1},-\uc\op)$.
When we view $\uc^{2i}$ as the subword $(\uc^{2i+1},-\uc\op)_{[1,2i|J|]}$,
$\sd_{2i}$ becomes a good sub seed of $\sd'_{2i+2}$ via the cluster
embedding $\iota$ sending $j\in[1,2i|J|]$ to $j$. So we have $\sd_{2i}\xhookrightarrow{\iota}\sd_{2i+2}'\xrightarrow{(\Sigma_{|J|}\cdots\Sigma_{1})^{-1}}\sd_{2i+2}$.
Note that the inclusion $\bClAlg(\sd_{2i})\subset\bClAlg(\sd_{2i+2})$
sends $W_{[j,k]}(\sd_{2i})$ to $W_{[j[1],k[1]]}(\sd_{2i+2})$, i.e.,
it is associated with the cluster embedding viewing $\uc^{2i}$ as
the subword $(\uc^{2i+2})_{[|J|+1,(2i+1)|J|]}$ (Lemma \ref{lem:embed-interval-variables}). 

We choose quantization as before. Then it is straightforward to check
that the quantization for $\sd_{2i+2}$ restricts to that of $\sd_{2i}$,
i.e., $\Lambda(\sd_{2i})(\deg^{\sd_{2i}}W_{[j^{\min},j]}(\sd_{2i}),\deg^{\sd_{2i}}W_{[k^{\min},k]}(\sd_{2i}))$
equals $\Lambda(\sd_{2i+2})(\deg^{\sd_{2i+2}}W_{[j^{\min}[1],j[1]]}(\sd_{2i+2}),\deg^{\sd_{2i+2}}W_{[k^{\min}[1],k[1]]}(\sd_{2i+2}))$,
$\forall j,k\in[1,2i|J|]$. From now on, we assume $\bClAlg_{i}:=\bClAlg(\sd_{2i})$
are quantum cluster algebras.

Define $\ubi^{(i)}:=(\uc,-\uc\op)^{i}$ and $\ssd_{i}:=\rsd(\ubi^{(i)})$.
Then we have $\ssd_{i}=\seq\sd_{2i}$ for the mutation sequence $\seq:=\seq_{\ubi^{(i)},\uc^{2i}}$
associated with the left reflections and flips in Section \ref{subsec:Operations-on-signed-words}.
Note that $\ssd_{i+1}=\seq\sd_{2i+2}'$. Therefore, $\ssd_{i}$ is
a good sub seed of $\ssd_{i+1}$ via the same cluster embedding $\iota$.
So we have the following commutative diagram (where the mutation isomorphisms
are omitted):
\begin{align*}
\begin{array}{cccccc}
\iota: & \bClAlg(\sd_{2i}) & \hookrightarrow & \bClAlg(\sd'_{2i+2}) & = & \bClAlg(\sd_{2i+2})\\
 & \parallel &  & \parallel\\
\iota: & \bClAlg(\ssd_{i}) & \hookrightarrow & \bClAlg(\ssd_{i+1})
\end{array}
\end{align*}

For any $1\leq j\leq2i|J|$, define $d(j)$ and $a(j)$ such that
$j-d(j)\cdot|J|\in[1,|J|]$ and $a(j)=(\uc^{i})_{j}$, then $j$ is
identified with $\binom{a(j)}{d(j)}^{\sd_{2i}}\in I(\sd_{2i})$. Note
that $\iota x_{j}(\sd_{2i})=x_{j}(\sd'_{2i+2})$. Using the definitions
of the interval variables, we have the following.

\begin{Lem} \label{lem:interval-for-seeds}

(1) $\forall j\in[1,2i|J|]$, we have $x_{j}(\sd'_{2i})=W_{[j^{\min}[1],j[1]]}(\sd_{2i})$.

(2) For any $j\in[1,2i|J|]$, we have $x_{j}(\ssd_{i})=W_{[j^{\min}[i-d],j^{\min}[i+d]]}(\sd_{2i})$
if $d(j)=2d$ and $x_{j}(\ssd_{i})=W_{[j^{\min}[i-d-1],j^{\min}[i+d]]}(\sd_{2i})$
if $d(j)=2d+1$. 

\end{Lem}

\begin{proof}

(1) Note that $\sd'_{2i}=\Sigma_{c_{|J|}}\cdots\Sigma_{c_{1}}\sd_{2i}$.
The claim follows from the definition of the interval variables.

(2) The case for $\ssd_{0}$ is trivial. Assume the claim has been
verified for $\ssd_{i-1}$, we will prove it for $\ssd_{i}$. Note
that we have $\ssd_{i}=\seq^{(i)}\sd_{2i}$ and similarly $\ssd_{i-1}=\seq^{(i-1)}\sd_{2i-2}$,
where $\seq^{(i)}$ and $\seq^{(i-1)}$ are associated with flips
and left reflections, see Section \ref{subsec:Operations-on-signed-words}.
In addition, $\seq^{(i)}=\seq^{(i-1)}\Sigma_{c_{|J|}}\cdots\Sigma_{c_{1}}$.
Note that $\seq^{(i-1)}$does not mutate on $j>(2i-2)|J|$. Consider
the previous commutative diagram
\begin{align*}
\begin{array}{cccc}
\iota: & \bClAlg(\sd_{2i-2}) & \hookrightarrow & \bClAlg(\sd'_{2i})\\
 & \parallel &  & \parallel\\
\iota: & \bClAlg(\ssd_{i-1}) & \hookrightarrow & \bClAlg(\ssd_{i})
\end{array}.
\end{align*}
For any $j\leq(2i-2)|J|$, we deduce the desired formula of $x_{j}(\ssd_{i})$
by using the induction hypothesis for $x_{j}(\ssd_{i-1})$ and then
applying the inclusion $\bClAlg(\sd_{2i-2})\subset\bClAlg(\sd_{2i})$
to it. If $(2i-2)|J|<j\leq(2i-1)|J|$, then $x_{j}(\ssd_{i})$ is
the cluster variable $x_{j}(\Sigma_{c_{|J|}}\cdots\Sigma_{c_{1}}\sd_{2i})=W_{[j^{\min}[1],j^{\max}]}(\sd_{2i})$,
which satisfies the desired formula since $d(j)=2(i-1)$ and $j^{\max}=j^{\min}[2i-1]$.
If $j>(2i-1)|J|$, $x_{j}(\ssd_{i})$ is the cluster variable $x_{j}(\sd_{2i})=W_{[j^{\min},j^{\max}]}(\sd_{2i})$,
which satisfies the desired formula since $d(j)=2(i-1)+1$.

\end{proof}

We can view $I(\sd_{2i})=I(\ssd_{i})$ as a subset of $\W=\{\mybinom ad|a\in J,d\in\Z\}$
such that $\binom{a}{d}^{\sd_{2i}}=\binom{a}{d}^{\ssd_{i}}$ is identified
with $\mybinom a{d-i+1}$. Then we have $I(\sd_{2i})=\{\mybinom ad|d\in[-i+1,i]\}$,
and $\iota W_{[\mybinom ad,\mybinom a{d'}]}(\sd_{2i})=W_{[\mybinom ad,\mybinom a{d'}]}(\sd_{2i+2})$,
which are denoted by $W_{[\mybinom ad,\mybinom a{d'}]}$. Then Lemma
\ref{lem:interval-for-seeds}(2) is equivalent to the following.

\begin{Lem}\label{lem:interval-for-ssd}

For any $j\in[1,2i|J|]\simeq I(\ssd_{i})\subset\W$, we have $x_{j}(\ssd_{i})=W_{[\mybinom a{-d+1},\mybinom a{d+1}]}$
if $j=\binom{a}{2d}^{\ssd_{i}}$ and $x_{j}(\ssd_{i})=W_{[\mybinom a{-d},\mybinom a{d+1}]}$
if $j=\binom{a}{2d+1}^{\ssd_{i}}$.

\end{Lem}

\begin{Eg}

In Figure \ref{eg:A2-12121212-recenter}, we draw the quiver for $\sd_{4}=\rsd((1,2)^{4})$
and represent its nodes in different ways. Note that $\ssd_{2}=\rsd((1,2,-2,-1,1,2,-2,-1))=\mu_{2}\mu_{1}(\mu_{6}\mu_{4}\mu_{2})(\mu_{5}\mu_{3}\mu_{1})\sd_{4}=\mu_{2}\mu_{1}\Sigma_{2}\Sigma_{1}\sd_{4}$.
In Figure \ref{fig:A2-12(21)12(21)-recenter}, we draw the quiver
with initial cluster variables for $\sd_{4}$ and $\ssd_{2}$.

\end{Eg}

\begin{figure}[h]
\caption{The quiver for $\sd_{4}=\rsd(\uc^{4})$}
\label{eg:A2-12121212-recenter}

\subfloat[Nodes $I(\sd_{4})$]{

 \begin{tikzpicture}  [scale=1.2,node distance=48pt,on grid,>={Stealth[round]},bend angle=45,      pre/.style={<-,shorten <=1pt,>={Stealth[round]},semithick},    post/.style={->,shorten >=1pt,>={Stealth[round]},semithick},  unfrozen/.style= {circle,inner sep=1pt,minimum size=12pt,draw=black!100,fill=red!100},  frozen/.style={rectangle,inner sep=1pt,minimum size=12pt,draw=black!75,fill=cyan!100},   point/.style= {circle,inner sep=1pt,minimum size=5pt,draw=black!100,fill=black!100},   boundary/.style={-,draw=cyan},   internal/.style={-,draw=red},    every label/.style= {black}] \node (v1) at (0,1) {$\binom{1}{0}^{\sd_4}$}; \node (v2) at (-0.5,0) {$\binom{2}{0}^{\sd_4}$}; \node (v3) at (-1,1) {$\binom{1}{1}^{\sd_4}$}; \node (v4) at (-1.5,0) {$\binom{2}{1}^{\sd_4}$}; \node (v5) at (-2,1) {$\binom{1}{2}^{\sd_4}$}; \node (v6) at (-2.5,0) {$\binom{2}{2}^{\sd_4}$}; \node (v8) at (-3,1) {$\binom{1}{3}^{\sd_4}$}; \node (v7) at (-3.5,0) {$\binom{2}{3}^{\sd_4}$}; \draw[<-]  (v1) edge (v2); \draw[<-]  (v2) edge (v3); \draw[<-]  (v3) edge (v4); \draw[<-]  (v4) edge (v5); \draw[<-]  (v5) edge (v6); \draw[<-]  (v5) edge (v3); \draw[<-]  (v3) edge (v1); \draw[<-]  (v6) edge (v4); \draw[<-]  (v4) edge (v2);
\draw[<-]  (v7) edge (v6); \draw[<-]  (v8) edge (v5); \draw[<-]  (v6) edge (v8); \draw[<-,dashed]  (v8) edge (v7); \end{tikzpicture}}$\quad$\subfloat[{Nodes $[1,8]$}]{

 \begin{tikzpicture}  [scale=1.2,node distance=48pt,on grid,>={Stealth[round]},bend angle=45,      pre/.style={<-,shorten <=1pt,>={Stealth[round]},semithick},    post/.style={->,shorten >=1pt,>={Stealth[round]},semithick},  unfrozen/.style= {circle,inner sep=1pt,minimum size=12pt,draw=black!100,fill=red!100},  frozen/.style={rectangle,inner sep=1pt,minimum size=12pt,draw=black!75,fill=cyan!100},   point/.style= {circle,inner sep=1pt,minimum size=5pt,draw=black!100,fill=black!100},   boundary/.style={-,draw=cyan},   internal/.style={-,draw=red},    every label/.style= {black}] \node (v1) at (0,1) {$1$}; \node (v2) at (-0.5,0) {$2$}; \node (v3) at (-1,1) {$3$}; \node (v4) at (-1.5,0) {$4$}; \node (v5) at (-2,1) {$5$}; \node (v6) at (-2.5,0) {$6$}; \node (v8) at (-3,1) {$7$}; \node (v7) at (-3.5,0) {$8$}; \draw[<-]  (v1) edge (v2); \draw[<-]  (v2) edge (v3); \draw[<-]  (v3) edge (v4); \draw[<-]  (v4) edge (v5); \draw[<-]  (v5) edge (v6); \draw[<-]  (v5) edge (v3); \draw[<-]  (v3) edge (v1); \draw[<-]  (v6) edge (v4); \draw[<-]  (v4) edge (v2);
\draw[<-]  (v7) edge (v6); \draw[<-]  (v8) edge (v5); \draw[<-]  (v6) edge (v8); \draw[<-,dashed]  (v8) edge (v7); \end{tikzpicture}}$\quad$\subfloat[Nodes in $\W$]{

\begin{tikzpicture}  [scale=1.2,node distance=48pt,on grid,>={Stealth[round]},bend angle=45,      pre/.style={<-,shorten <=1pt,>={Stealth[round]},semithick},    post/.style={->,shorten >=1pt,>={Stealth[round]},semithick},  unfrozen/.style= {circle,inner sep=1pt,minimum size=12pt,draw=black!100,fill=red!100},  frozen/.style={rectangle,inner sep=1pt,minimum size=12pt,draw=black!75,fill=cyan!100},   point/.style= {circle,inner sep=1pt,minimum size=5pt,draw=black!100,fill=black!100},   boundary/.style={-,draw=cyan},   internal/.style={-,draw=red},    every label/.style= {black}] \node (v1) at (0,1) {$\mybinom{1}{-1}$}; \node (v2) at (-0.5,0) {$\mybinom{2}{-1}$}; \node (v3) at (-1,1) {$\mybinom{1}{0}$}; \node (v4) at (-1.5,0) {$\mybinom{2}{0}$}; \node (v5) at (-2,1) {$\mybinom{1}{1}$}; \node (v6) at (-2.5,0) {$\mybinom{2}{1}$}; \node (v8) at (-3,1) {$\mybinom{1}{2}$}; \node (v7) at (-3.5,0) {$\mybinom{2}{2}$}; \draw[<-]  (v1) edge (v2); \draw[<-]  (v2) edge (v3); \draw[<-]  (v3) edge (v4); \draw[<-]  (v4) edge (v5); \draw[<-]  (v5) edge (v6); \draw[<-]  (v5) edge (v3); \draw[<-]  (v3) edge (v1); \draw[<-]  (v6) edge (v4); \draw[<-]  (v4) edge (v2); \draw[<-]  (v7) edge (v6); \draw[<-]  (v8) edge (v5); \draw[<-]  (v6) edge (v8); \draw[<-,dashed]  (v8) edge (v7); \end{tikzpicture}}
\end{figure}

\begin{figure}[h]
\caption{Quiver and initial cluster variables}
\label{fig:A2-12(21)12(21)-recenter}

\subfloat[$\sd_{4}=\rsd(1,2,1,2,1,2,1,2)$]{

  \begin{tikzpicture}  [scale=2,node distance=48pt,on grid,>={Stealth[round]},bend angle=45,      pre/.style={<-,shorten <=1pt,>={Stealth[round]},semithick},    post/.style={->,shorten >=1pt,>={Stealth[round]},semithick},  unfrozen/.style= {circle,inner sep=1pt,minimum size=12pt,draw=black!100,fill=red!100},  frozen/.style={rectangle,inner sep=1pt,minimum size=12pt,draw=black!75,fill=cyan!100},   point/.style= {circle,inner sep=1pt,minimum size=5pt,draw=black!100,fill=black!100},   boundary/.style={-,draw=cyan},   internal/.style={-,draw=red},    every label/.style= {black}] \node (v1) at (0,1) {$W_{\mybinom{1}{-1}}$}; \node (v2) at (-0.5,0) {$W_{\mybinom{2}{-1}}$}; \node (v3) at (-1,1) {$W_{[\mybinom{1}{-1},\mybinom{1}{0}]}$}; \node (v4) at (-1.5,0) {$W_{[\mybinom{2}{-1},\mybinom{2}{0}]}$}; \node (v5) at (-2,1) {$W_{[\mybinom{1}{-1},\mybinom{1}{1}]}$}; \node (v6) at (-2.5,0) {$W_{[\mybinom{2}{-1},\mybinom{2}{1}]}$}; \node (v8) at (-3,1) {$W_{[\mybinom{1}{-1},\mybinom{1}{2}]}$}; \node (v7) at (-3.5,0) {$W_{[\mybinom{2}{-1},\mybinom{2}{2}]}$}; \draw[<-]  (v1) edge (v2); \draw[<-]  (v2) edge (v3); \draw[<-]  (v3) edge (v4); \draw[<-]  (v4) edge (v5); \draw[<-]  (v5) edge (v6); \draw[<-]  (v5) edge (v3); \draw[<-]  (v3) edge (v1); \draw[<-]  (v6) edge (v4); \draw[<-]  (v4) edge (v2); \draw[<-]  (v7) edge (v6); \draw[<-]  (v8) edge (v5); \draw[<-]  (v6) edge (v8); \draw[<-,dashed]  (v8) edge (v7); \end{tikzpicture}}$\quad$\subfloat[$\ssd_{2}=\rsd(1,2,-2,-1,1,2,-2,-1)$]{

\begin{tikzpicture}  [scale=2,node distance=48pt,on grid,>={Stealth[round]},bend angle=45,      pre/.style={<-,shorten <=1pt,>={Stealth[round]},semithick},    post/.style={->,shorten >=1pt,>={Stealth[round]},semithick},  unfrozen/.style= {circle,inner sep=1pt,minimum size=12pt,draw=black!100,fill=red!100},  frozen/.style={rectangle,inner sep=1pt,minimum size=12pt,draw=black!75,fill=cyan!100},   point/.style= {circle,inner sep=1pt,minimum size=5pt,draw=black!100,fill=black!100},   boundary/.style={-,draw=cyan},   internal/.style={-,draw=red},    every label/.style= {black}] \node (v1) at (-1,1) {$W_{\mybinom{1}{1}}$}; \node (v2) at (-1,0) {$W_{\mybinom{2}{1}}$}; \node (v3) at (-2,1) {$W_{[\mybinom{1}{0},\mybinom{1}{1}]}$}; \node (v4) at (-2,0) {$W_{[\mybinom{2}{0},\mybinom{2}{1}]}$}; \node (v5) at (-3,1) {$W_{[\mybinom{1}{0},\mybinom{1}{2}]}$}; \node (v6) at (-3,0) {$W_{[\mybinom{2}{0},\mybinom{2}{2}]}$}; \node (v8) at (-4,1) {$W_{[\mybinom{1}{-1},\mybinom{1}{2}]}$}; \node (v7) at (-4,0) {$W_{[\mybinom{2}{-1},\mybinom{2}{2}]}$}; \draw[->]  (v2) edge (v1); \draw[->]  (v1) edge (v4); \draw[->]  (v4) edge (v3); \draw[->]  (v6) edge (v5); \draw[->]  (v5) edge (v7); \draw[->,dashed]  (v7) edge (v8); \draw[->]  (v8) edge (v5); \draw[->]  (v3) edge (v5); \draw[->]  (v3) edge (v1); \draw[->]  (v5) edge (v4); \draw[->]  (v4) edge (v6); \draw[->]  (v4) edge (v2); \draw[<-]  (v6) edge (v7);\end{tikzpicture}}
\end{figure}

Since $\ssd_{i}$ are good sub seeds of $\ssd_{i+1}$, we can take
the colimit $\ssd_{\infty}:=\cup_{i}\ssd_{i}$, where $I(\ssd_{\infty})=\{\mybinom ad\in\W|a\in J,d\in\Z\}\simeq\cup_{i}[1,2i]=\N_{>0}$.
Let $\stdMod_{i}$ and $\can_{i}$ denote the standard basis and the
common triangular basis of $\bClAlg_{i}$, then $\stdMod=\cup_{i}\stdMod_{i}$
and $\can=\cup_{i}\can_{i}$ are bases of $\bClAlg_{\infty}:=\bClAlg(\ssd_{\infty})$,
called the standard basis and the common triangular basis, respectively.

\begin{Rem}

When $C$ is of type $ADE$, choose a height function $\xi$. Then
$\bClAlg_{i}$ is categorified by $\cT_{i}:=\cC_{\uc^{2i}}(\xi+2i-2)$,
such that $W_{\mybinom ad}$ corresponds to the fundamental module
$L(a,\xi_{a}-2d)$. It follows that $(\bClAlg_{\infty},\can_{\infty})$
is categorified by $\cC_{\Z}:=\cup_{i}\cC_{\uc^{2i}}(\xi+2i-2)$.
Note that, in all finite types, $\cC_{\Z}$ as introduced in \cite{HernandezLeclerc09}
categorifies a cluster algebra by \cite{kashiwara2021monoidal}.

\end{Rem}

\subsection{Cluster algebras from virtual quantum Grothendieck rings\label{subsec:Cluster-algebras-identify-Kq}}

The authors in \cite{jang2023quantization} showed that the virtual
quantum Grothendieck ring $\mathfrak{K}_{q}$ in \cite{kashiwara2023q}
takes the form $\bClAlg(\usd)$ for some seed $\usd$. We will identify
$\bClAlg(\usd)$ with our $\bClAlg_{\infty}=\bClAlg(\ssd_{\infty})$
introduced in Section \ref{sec:Infinite-rank-cluster-q-aff}. In this
subsection, we will use the bilinear form $\ucN$ for quantization.

Assume $C$ is of finite type. We can choose the Coxeter word $\uc$
of the form $\uc:=(\uc^{+},\uc^{-})$, such that $J=J^{+}\sqcup J^{-}$,
where $J^{\pm}$ denote of the set of sink letters and source letters
respectively, and $\uc^{\pm}$ are words in $J^{\pm}$. This choice
corresponds to a bipartite orientation of the associated Dynkin diagram.
Choose the height function $\xi$ such that $\xi(J^{+})=s$ and $\xi(J^{-})=s-1$,
$s\in\Z$.

As before, we take the inclusion $I(\sd_{2i})\subset\W\simeq J_{\Z}(\xi)=\cup_{a\in J}\{a\}\times(\xi_{a}+2\Z)$,
identifying $\binom{a}{d}^{\ubi}$ with $\mybinom a{d-i+1}\in\W$
and $(a,\xi_{a}-2d+2i-2)\in J_{\Z}(\xi)$. Quantize $\sd_{2i}$ using
$\ucN$ as before.

Introduce $\ubi=(\uc^{+},-(\uc^{-})\op,-(\uc^{+})\op,\uc^{-})$. Denote
$\usd_{i}:=\rsd(\ubi^{i})$. Then $\usd_{i}$ equals $\seq^{([1,j])}\ssd_{i}$,
where $\ssd_{i}=\rsd((\uc,-\uc\op)^{i})$, $\seq^{([1,j])}:=\prod_{j=1}^{i}\seq^{(j)}$,
and $\seq^{(j)}:=\prod_{a\in J^{-}}\mu_{\binom{a}{2j-2}}$ is the
mutation sequence associated with left reflections and flips changing
$(\ubi^{i})_{[2(j-1)|J|+1,2j|J|]}=\ubi$ to $((\uc,-\uc\op)^{i})_{[2(j-1)|J|+1,2j|J|]}=(\uc,-\uc\op)$.
Note that $\seq^{(j)}$ and $\seq^{(j')}$ commute. So we obtain the
commutative diagram via the cluster embedding $\iota=\iota_{[1,2i|J|]}$
of subwords, see Section \ref{subsec:Cluster-embeddings-subwords}:
\begin{align*}
\begin{array}{ccccc}
\ssd_{i} & \xhookrightarrow{\iota} & \ssd_{i+1} & = & \ssd_{i+1}\\
\downarrow\seq^{([1,i])} &  & \downarrow\seq^{([1,i])} &  & \downarrow\seq^{([1,i+1])}\\
\usd_{i} & \xhookrightarrow{\iota} & \rsd(\ubi^{i},\uc,-\uc\op) & \xrightarrow{\seq^{(i+1)}} & \usd_{i+1}
\end{array}
\end{align*}

Recall that $\ssd_{i}$ is a good sub seed of $\ssd_{i+1}$ and $I(\ssd_{i})=I(\usd_{i})$.
Since $b_{jk}(\ssd_{i})=0$, $\forall j\in I(\ssd_{i+1})\backslash I(\ssd_{i})$,
$k\in I_{\ufv}(\ssd_{i})$, and $\seq^{(i+1)}$ only mutates at $I(\ssd_{i+1})\backslash I(\ssd_{i})$,
$\usd_{i}$ is also a good sub seed of $\usd_{i+1}$ via the cluster
embedding $\iota$. Applying $\seq^{([1,i])}$ to the cluster variables
of $\ssd_{i}$, we deduce the following from Lemma \ref{lem:interval-for-ssd}.

\begin{Lem}\label{lem:initial-cluster-usd}

For any $d\in[0,i-1]$, if $a\in J^{-}$, we have $x_{\binom{a}{2d}^{\usd_{i}}}(\usd_{i})=W_{[\mybinom a{-d},\mybinom ad]}$
and $x_{\binom{a}{2d+1}^{\usd_{i}}}(\usd_{i})=W_{[\mybinom a{-d},\mybinom a{d+1}]}$;
and if $a\in J^{+}$, we have $x_{\binom{a}{2d}^{\usd_{i}}}(\usd_{i})=W_{[\mybinom a{-d+1},\mybinom a{d+1}]}$
and $x_{\binom{a}{2d+1}^{\usd_{i}}}(\usd_{i})=W_{[\mybinom a{-d},\mybinom a{d+1}]}$.

\end{Lem}

\begin{proof}

Recall that $\usd_{i}=\prod_{a\in J^{-},d\in[0,i-1]}\mu_{\binom{a}{2d}}\ssd_{i}$.
We can deduce the formulae from Lemma \ref{lem:interval-for-ssd}.
More precisely, it remains to compute $x_{\binom{a}{2d}}(\usd_{i})=x_{\binom{a}{2d}}(\mu_{\binom{a}{2d}}\ssd_{i})$
for $a\in J^{-}$. 

We can compute the $B$-matrix for $\ssd_{i}$ explicitly. Then we
deduce the exchange relation $x_{\binom{a}{2d}}(\usd_{i})*x_{\binom{a}{2d}}(\ssd_{i})=q^{\gamma}x_{\binom{a}{2d-1}}(\ssd_{i})*x_{\binom{a}{2d+1}}(\ssd_{i})+q^{\gamma'}\prod_{b\in J^{+}}x_{\binom{b}{2d}}(\ssd_{i})^{|C_{ba}|}$
for some $\gamma,\gamma'\in\Q$. On the other hand, by the $T$-system
(\ref{eq:T-systems-Coxeter}), we have $W_{[\mybinom a{-d},\mybinom ad]}*W_{[\mybinom a{-d+1},\mybinom a{d+1}]}=q^{\alpha}W_{[\mybinom a{-d+1},\mybinom ad]}*W_{[\mybinom a{-d},\mybinom a{d+1}]}+q^{\alpha'}\prod_{b\in J^{+}}W_{[\mybinom b{-d+1},\mybinom b{d+1}]}^{|C_{ba}|}$
for some $\alpha,\alpha'\in\Q$. Combining with Lemma \ref{lem:interval-for-ssd},
we deduce $x_{\binom{a}{2d}}(\usd_{i})=W_{[\mybinom a{-d},\mybinom ad]}$
at the classical level. It follows that $x_{\binom{a}{2d}}(\usd_{i})=W_{[\mybinom a{-d},\mybinom ad]}$
as a quantum cluster variable.

\end{proof}

\begin{Eg}

In type $A_{2}$, take $\uc=(1,2)$, $\usd_{2}=\rsd(1,-2,-1,2,1,-2,-1,2)$.
Denote $I(\usd_{2})=I(\ssd_{2})=I(\sd_{4})\simeq[1,8]$ as usual.
We have $\usd_{2}=\seq^{(2)}\seq^{(1)}\ssd$, where $\seq^{(1)}=\mu_{2}$
and $\seq^{(2)}=\mu_{6}$. Its quiver with initial cluster variables
is depicted in Figure \ref{fig:A2-1(21)21(21)2-recenter}.

\begin{figure}[h]
\caption{Quiver and initial cluster variables for $\usd_{2}$}
\label{fig:A2-1(21)21(21)2-recenter}

\subfloat[Quiver]{\begin{tikzpicture}  [scale=1.5,node distance=48pt,on grid,>={Stealth[round]},bend angle=45,      pre/.style={<-,shorten <=1pt,>={Stealth[round]},semithick},    post/.style={->,shorten >=1pt,>={Stealth[round]},semithick},  unfrozen/.style= {circle,inner sep=1pt,minimum size=12pt,draw=black!100,fill=red!100},  frozen/.style={rectangle,inner sep=1pt,minimum size=12pt,draw=black!75,fill=cyan!100},   point/.style= {circle,inner sep=1pt,minimum size=5pt,draw=black!100,fill=black!100},   boundary/.style={-,draw=cyan},   internal/.style={-,draw=red},    every label/.style= {black}] \node[unfrozen] (v1) at (-1,1) {$1$}; \node[unfrozen] (v2) at (-1,0) {$2$}; \node[unfrozen] (v3) at (-2,1) {$3$}; \node[unfrozen] (v4) at (-2,0) {$4$}; \node[unfrozen] (v5) at (-3,1) {$5$}; \node[unfrozen] (v6) at (-3,0) {$6$}; \node[frozen] (v8) at (-4,1) {$7$}; \node[frozen] (v7) at (-4,0) {$8$};
\draw[->]  (v4) edge (v3); \draw[->]  (v3) edge (v1); \draw[->]  (v1) edge (v2); \draw[->]  (v2) edge (v4); \draw[->]  (v3) edge (v5); \draw[->]  (v5) edge (v6); \draw[->]  (v6) edge (v4); \draw[->]  (v6) edge (v7); \draw[->]  (v8) edge (v5); \draw[->,dashed]  (v7) edge (v8); \end{tikzpicture}}$\quad$\subfloat[Initial cluster variables]{

\begin{tikzpicture}  [scale=2,node distance=48pt,on grid,>={Stealth[round]},bend angle=45,      pre/.style={<-,shorten <=1pt,>={Stealth[round]},semithick},    post/.style={->,shorten >=1pt,>={Stealth[round]},semithick},  unfrozen/.style= {circle,inner sep=1pt,minimum size=12pt,draw=black!100,fill=red!100},  frozen/.style={rectangle,inner sep=1pt,minimum size=12pt,draw=black!75,fill=cyan!100},   point/.style= {circle,inner sep=1pt,minimum size=5pt,draw=black!100,fill=black!100},   boundary/.style={-,draw=cyan},   internal/.style={-,draw=red},    every label/.style= {black}] \node (v1) at (-1,1) {$W_{\mybinom{1}{1}}$}; \node (v2) at (-1,0) {$W_{\mybinom{2}{0}}$}; \node (v3) at (-2,1) {$W_{[\mybinom{1}{0},\mybinom{1}{1}]}$}; \node (v4) at (-2,0) {$W_{[\mybinom{2}{0},\mybinom{2}{1}]}$}; \node (v5) at (-3,1) {$W_{[\mybinom{1}{0},\mybinom{1}{2}]}$}; \node (v6) at (-3,0) {$W_{[\mybinom{2}{-1},\mybinom{2}{1}]}$}; \node (v8) at (-4,1) {$W_{[\mybinom{1}{-1},\mybinom{1}{2}]}$}; \node (v7) at (-4,0) {$W_{[\mybinom{2}{-1},\mybinom{2}{2}]}$};
\draw[->]  (v4) edge (v3); \draw[->]  (v3) edge (v1); \draw[->]  (v1) edge (v2); \draw[->]  (v2) edge (v4); \draw[->]  (v3) edge (v5); \draw[->]  (v5) edge (v6); \draw[->]  (v6) edge (v4); \draw[->]  (v6) edge (v7); \draw[->]  (v8) edge (v5); \draw[->,dashed]  (v7) edge (v8); \end{tikzpicture}}
\end{figure}

\end{Eg}

Denote $\binom{a}{d}^{\usd_{i}}$ by $\binom{a}{d}^{\usd_{\infty}}$.
Then $I(\usd_{\infty})=\cup_{i}I(\usd_{i})=\{\binom{a}{d}^{\usd_{\infty}}|a\in J,d\in\N\}$,
which is identical to $\N_{>0}$ such that $\binom{c_{j}}{d}^{\usd_{\infty}}=j+d|J|$.
Denote $\usd_{\infty}=\cup\usd_{i}$ as in Section \ref{sec:Based-cluster-algebras-infinite-ranks}.
We have $\bClAlg(\usd_{\infty})=\cup_{i}\bClAlg(\usd_{i})=\cup_{i}\bClAlg(\ssd_{i})=\bClAlg(\ssd_{\infty})=:\bClAlg_{\infty}$. 

We could identify our quantum seed $\usd_{\infty}$ with the quantum
seed for $\mathfrak{K}_{q}(\mathfrak{g})$ in \cite{jang2023quantization},
denoted $\usd$. To see this, take any $s\in\Z$, denote $\xi_{a}=s+1$
if $a\in J^{+}$ and $\xi_{0}=s$ if $a\in J^{-}$, and identify $J_{\Z}(\xi)\simeq\W$
as before. Identify our vertex $\binom{a}{d}^{\usd_{\infty}}\in I(\usd_{\infty})$
with $(a,\xi'_{a}-2d)$ in \cite[(9.1)]{jang2023quantization}, where
$\xi'_{a}=s$ if $a\in J^{-}$ and $\xi'_{a}=s-1$ if $a\in J^{+}$.
We observe that our $\tB(\usd_{\infty})$ is identified with $\tB(\usd)$
given in \cite[(9.1)]{jang2023quantization}. Moreover, our $\Lambda(\usd_{\infty})$
is identified with $\Lambda(\usd)$ in \cite[Theorem 9.4]{jang2023quantization}.
Therefore, we could identify $\usd_{\infty}$ with $\usd$, such that
our initial cluster variables in $\usd_{\infty}$ of the form $W_{[\mybinom ar,\mybinom a{r+s}]}$
become identified with the initial cluster variables (KR-polynomials)
in $\usd$ \cite[(9.2)]{jang2023quantization}, denoted $F_{q}(m^{(a)}[\xi_{a}-2r-2d,\xi_{a}-2r])$. 

\begin{Eg}\label{eg:B3-three-seeds}

Let us take $C=\left(\begin{array}{ccc}
2 & -1 & 0\\
-1 & 2 & -1\\
0 & -2 & 2
\end{array}\right)$ and $\uc=(2,1,3)$. Then $\uc^{+}=(2)$, $\uc^{-}=(1,3)$. In Figure
\ref{fig:B3-sd4}, we draw the valued quiver and the initial cluster
variables for $\sd_{4}=\rsd(\uc^{4})$. In Figure \ref{fig:B3-new-seeds},
we draw those for $\ssd_{2}=\rsd((\uc,-\uc\op)^{2})=\rsd((2,1,3,-3,-1,-2)^{2})$,
and $\usd_{2}=\rsd(\ubi^{2})$ where $\ubi=(\uc^{+},-(\uc^{-})\op,-(\uc^{+})\op,\uc^{-})=(2,-3,-1,-2,1,3)$.
The upwards solid double headed arrows have weight $(2,1)$, the downward
solid ones have weight $(1,2)$, and the dashed double headed arrows
have half weights $(1,\Hf)$ and $(\Hf,1)$ respectively. Note that
$\ssd_{2}=\mu_{3}\mu_{2}\mu_{1}\Sigma_{3}\Sigma_{2}\Sigma_{1}\sd_{4}$,
$\usd_{2}=\seq^{(2)}\seq^{(1)}\ssd_{2}$, where $\seq^{(1)}$ mutates
$W_{\mybinom 11}$, $W_{\mybinom 31}$ and $\seq^{(2)}$ mutates $W_{[\mybinom 10,\mybinom 12]}$,
$W_{[\mybinom 30,\mybinom 32]}$. This example for $\usd_{2}$ could
be compared with \cite[Examples 9.1(1) and 9.3(1)]{jang2023quantization}.

\begin{figure}[h]
\caption{Quiver and initial cluster variables in type $B_{3}$, $\sd_{4}=\rsd((213)^{4})$}
\label{fig:B3-sd4}

\subfloat[Quiver for $\sd_{4}$]{ \begin{tikzpicture}  [scale=2,node distance=48pt,on grid,>={Stealth[round]},bend angle=45,      pre/.style={<-,shorten <=1pt,>={Stealth[round]},semithick},    post/.style={->,shorten >=1pt,>={Stealth[round]},semithick},  unfrozen/.style= {circle,inner sep=1pt,minimum size=12pt,draw=black!100,fill=red!100},  frozen/.style={rectangle,inner sep=1pt,minimum size=12pt,draw=black!75,fill=cyan!100},   point/.style= {circle,inner sep=1pt,minimum size=5pt,draw=black!100,fill=black!100},   boundary/.style={-,draw=cyan},   internal/.style={-,draw=red},    every label/.style= {black}]
                        \node[unfrozen] (v1) at (2,-1) {1}; \node[unfrozen] (v9) at (1.5,-0.5) {2}; \node[unfrozen] (v2) at (1.5,-1.5) {3}; \node[unfrozen] (v3) at (1,-1) {4}; \node[unfrozen] (v10) at (0.5,-0.5) {5}; \node[unfrozen] (v4) at (0.5,-1.5) {6}; \node[unfrozen] (v5) at (0,-1) {7}; \node[unfrozen] (v11) at (-0.5,-0.5) {8}; \node[unfrozen] (v6) at (-0.5,-1.5) {9}; \node[frozen] (v7) at (-1,-1) {10}; \node[frozen] (v12) at (-1.5,-0.5) {11}; \node[frozen] (v8) at (-1.5,-1.5) {12};
\draw[<<-,blue]  (v1) edge (v2); \draw[<<-,blue]  (v2) edge (v3); \draw[<<-,blue]  (v3) edge (v4); \draw[<<-,blue]  (v4) edge (v5); \draw[<<-,blue]  (v5) edge (v6); \draw [<<-,blue] (v6) edge (v7); \draw[<<-,dashed,blue]  (v7) edge (v8); \draw[<-]  (v1) edge (v9); \draw  [<-](v9) edge (v3); \draw  [<-](v3) edge (v10); \draw  [<-](v10) edge (v5); \draw  [<-](v5) edge (v11); \draw  [<-](v11) edge (v7); \draw  [<-,dashed](v7) edge (v12);
\draw  [<-](v12) edge (v11); \draw  [<-](v11) edge  (v10); \draw  [<-](v10) edge (v9); \draw  [<-](v7) edge (v5); \draw  [<-](v5) edge (v3); \draw  [<-](v3) edge (v1); \draw  [<-](v8) edge (v6); \draw  [<-](v6) edge (v4); \draw  [<-](v4) edge (v2); \end{tikzpicture}}$\quad$\subfloat[Initial cluster variables for $\sd_{4}$]{

 \begin{tikzpicture}  [scale=2,node distance=48pt,on grid,>={Stealth[round]},bend angle=45,      pre/.style={<-,shorten <=1pt,>={Stealth[round]},semithick},    post/.style={->,shorten >=1pt,>={Stealth[round]},semithick},  unfrozen/.style= {circle,inner sep=1pt,minimum size=12pt,draw=black!100,fill=red!100},  frozen/.style={rectangle,inner sep=1pt,minimum size=12pt,draw=black!75,fill=cyan!100},   point/.style= {circle,inner sep=1pt,minimum size=5pt,draw=black!100,fill=black!100},   boundary/.style={-,draw=cyan},   internal/.style={-,draw=red},    every label/.style= {black}]                         \node (v1) at (2,-1) {$W_{\mybinom{2}{-1}}$}; \node (v9) at (1.5,-0.5) {$W_{\mybinom{1}{-1}}$}; \node (v2) at (1.5,-1.6) {$W_{\mybinom{3}{-1}}$}; \node (v3) at (1,-1) {$W_{[\mybinom{2}{-1},\mybinom{2}{0}]}$}; \node (v10) at (0.5,-0.5) {$W_{[\mybinom{1}{-1},\mybinom{1}{0}]}$}; \node (v4) at (0.5,-1.6) {$W_{[\mybinom{3}{-1},\mybinom{3}{0}]}$}; \node (v5) at (0,-1) {$W_{[\mybinom{2}{-1},\mybinom{2}{1}]}$}; \node(v11) at (-0.5,-0.5) {$W_{[\mybinom{1}{-1},\mybinom{1}{1}]}$}; \node (v6) at (-0.5,-1.6) {$W_{[\mybinom{3}{-1},\mybinom{3}{1}]}$}; \node (v7) at (-1,-1) {$W_{[\mybinom{2}{-1},\mybinom{2}{2}]}$}; \node (v12) at (-1.5,-0.5) {$W_{[\mybinom{1}{-1},\mybinom{1}{2}]}$}; \node (v8) at (-1.5,-1.6) {$W_{[\mybinom{3}{-1},\mybinom{3}{2}]}$}; \draw[<<-,blue]  (v1) edge (v2); \draw[<<-,blue]  (v2) edge (v3); \draw[<<-,blue]  (v3) edge (v4); \draw[<<-,blue]  (v4) edge (v5); \draw[<<-,blue]  (v5) edge (v6); \draw [<<-,blue] (v6) edge (v7); \draw[<<-,dashed,blue]  (v7) edge (v8); \draw[<-]  (v1) edge (v9); \draw  [<-](v9) edge (v3); \draw  [<-](v3) edge (v10); \draw  [<-](v10) edge (v5); \draw  [<-](v5) edge (v11); \draw  [<-](v11) edge (v7); \draw  [<-,dashed](v7) edge (v12); \draw  [<-](v12) edge (v11); \draw  [<-](v11) edge  (v10); \draw  [<-](v10) edge (v9); \draw  [<-](v7) edge (v5); \draw  [<-](v5) edge (v3); \draw  [<-](v3) edge (v1); \draw  [<-](v8) edge (v6); \draw  [<-](v6) edge (v4); \draw  [<-](v4) edge (v2); \end{tikzpicture}}
\end{figure}

\begin{figure}[h]
\caption{Quiver and initial cluster variables for other seeds in type $B_{3}$}
\label{fig:B3-new-seeds}

\subfloat[$\ssd_{2}=\rsd((2,1,3,-3,-1,-2)^{2})$]{  \begin{tikzpicture}  [scale=2,node distance=48pt,on grid,>={Stealth[round]},bend angle=45,      pre/.style={<-,shorten <=1pt,>={Stealth[round]},semithick},    post/.style={->,shorten >=1pt,>={Stealth[round]},semithick},  unfrozen/.style= {circle,inner sep=1pt,minimum size=12pt,draw=black!100,fill=red!100},  frozen/.style={rectangle,inner sep=1pt,minimum size=12pt,draw=black!75,fill=cyan!100},   point/.style= {circle,inner sep=1pt,minimum size=5pt,draw=black!100,fill=black!100},   boundary/.style={-,draw=cyan},   internal/.style={-,draw=red},    every label/.style= {black}]                         \node (v1) at (2,-1) {$W_{\mybinom{2}{1}}$};                         \node (v9) at (2,-0.4) {$W_{\mybinom{1}{1}}$};                          \node (v2) at (2,-1.7) {$W_{\mybinom{3}{1}}$};                          \node (v3) at (1,-1) {$W_{[\mybinom{2}{0},\mybinom{2}{1}]}$};                          \node (v10) at (1,-0.4) {$W_{[\mybinom{1}{0},\mybinom{1}{1}]}$};                         \node (v4) at (1,-1.7) {$W_{[\mybinom{3}{0},\mybinom{3}{1}]}$};                         \node (v5) at (0,-1) {$W_{[\mybinom{2}{0},\mybinom{2}{2}]}$};                         \node(v11) at (0,-0.4) {$W_{[\mybinom{1}{0},\mybinom{1}{2}]}$};                          \node (v6) at (0,-1.7) {$W_{[\mybinom{3}{0},\mybinom{3}{2}]}$};                          \node (v7) at (-1,-1) {$W_{[\mybinom{2}{-1},\mybinom{2}{2}]}$};                         \node (v12) at (-1,-0.4) {$W_{[\mybinom{1}{-1},\mybinom{1}{2}]}$};                          \node (v8) at (-1,-1.7) {$W_{[\mybinom{3}{-1},\mybinom{3}{2}]}$}; \draw[->]  (v9) edge (v1);  \draw[->]  (v1) edge (v10); \draw[->]  (v10) edge (v9);  \draw[->]  (v10) edge (v11);  \draw[->]  (v11) edge (v5);  \draw[->]  (v5) edge (v10); \draw[->]  (v5) edge (v12); \draw[->]  (v12) edge (v11); \draw[->,dashed]  (v12) edge (v7); \draw[->]  (v7) edge (v5);  \draw[->]  (v10) edge (v3); \draw[->]  (v3) edge (v5);  \draw[->]  (v3) edge (v1); \draw[->]  (v4) edge (v2); \draw[->]  (v4) edge (v6);  \draw[->]  (v8) edge (v6); \draw[->>,blue,dashed]  (v8) edge (v7);  \draw[->>,blue]  (v5) edge (v6); \draw[->>,blue]  (v4) edge (v3); \draw[->>,blue]  (v2) edge (v1);  \draw[->>,blue]  (v1) edge (v4);  \draw[->>,blue]  (v5) edge (v4);  \draw[->>,blue]  (v5) edge (v8);  \end{tikzpicture}}$\quad$\subfloat[$\usd_{2}=\rsd((2,-3,-1,-2,1,3)^{2})$]{

 \begin{tikzpicture}  [scale=2,node distance=48pt,on grid,>={Stealth[round]},bend angle=45,      pre/.style={<-,shorten <=1pt,>={Stealth[round]},semithick},    post/.style={->,shorten >=1pt,>={Stealth[round]},semithick},  unfrozen/.style= {circle,inner sep=1pt,minimum size=12pt,draw=black!100,fill=red!100},  frozen/.style={rectangle,inner sep=1pt,minimum size=12pt,draw=black!75,fill=cyan!100},   point/.style= {circle,inner sep=1pt,minimum size=5pt,draw=black!100,fill=black!100},   boundary/.style={-,draw=cyan},   internal/.style={-,draw=red},    every label/.style= {black}]                         \node (v1) at (1.5,-1) {$W_{\mybinom{2}{1}}$}; \node (v9) at (2,-0.4) {$W_{\mybinom{1}{0}}$}; \node (v2) at (2,-1.7) {$W_{\mybinom{3}{0}}$}; \node (v3) at (0.5,-1) {$W_{[\mybinom{2}{0},\mybinom{2}{1}]}$}; \node (v10) at (1,-0.4) {$W_{[\mybinom{1}{0},\mybinom{1}{1}]}$}; \node (v4) at (1,-1.7) {$W_{[\mybinom{3}{0},\mybinom{3}{1}]}$}; \node (v5) at (-0.5,-1) {$W_{[\mybinom{2}{0},\mybinom{2}{2}]}$}; \node(v11) at (0,-0.4) {$W_{[\mybinom{1}{-1},\mybinom{1}{1}]}$}; \node (v6) at (0,-1.7) {$W_{[\mybinom{3}{-1},\mybinom{3}{1}]}$}; \node (v7) at (-1.5,-1) {$W_{[\mybinom{2}{-1},\mybinom{2}{2}]}$}; \node (v12) at (-1,-0.4) {$W_{[\mybinom{1}{-1},\mybinom{1}{2}]}$}; \node (v8) at (-1,-1.7) {$W_{[\mybinom{3}{-1},\mybinom{3}{2}]}$}; \draw[->]  (v1) edge (v9); \draw[->]  (v9) edge (v10); \draw  [->](v10) edge (v3); \draw[->]  (v3) edge (v1); \draw [->] (v5) edge (v11); \draw[->]  (v11) edge (v10); \draw [->] (v3) edge (v5); \draw  [->](v7) edge (v5); \draw  [->](v11) edge (v12); \draw  [->](v2) edge (v4); \draw  [->](v6) edge (v4); \draw  [->](v6) edge (v8); \draw  [->>,blue](v1) edge (v2); \draw  [->>,blue](v4) edge (v3); \draw  [->>,blue](v5) edge (v6); \draw  [->>,dashed](v12) edge (v7); \draw  [->>,blue,dashed](v8) edge (v7); \end{tikzpicture}}
\end{figure}

\end{Eg}

\begin{Lem}\label{lem:interval-KR}

The KR-polynomials $F_{q}(w)$ in \cite{jang2023quantization}, where
$w=e_{(a,\xi_{a}-2d)}+e_{(a,\xi_{a}-2d+2)}+\cdots+e_{(a,\xi_{a}-2s)}$
for $s\leq d\in\Z$, $a\in J$, coincides with $W_{[\mybinom as,\mybinom ad]}$
of $\bClAlg_{\infty}$.

\end{Lem}

\begin{proof}

The KR-polynomials $F_{q}(w)$ satisfy a collection of equations in
\cite[Theorem 6.9]{jang2023quantization}, called the $T$-systems.
Particularly, they are quantum cluster variables of $\bClAlg(\usd)$,
see \cite[Proposition 9.16]{jang2023quantization}. On the other hand,
$W_{[\mybinom as,\mybinom ad]}$ are also a quantum cluster variables
of $\bClAlg_{\infty}=\bClAlg(\usd)$. So it suffices to verify the
claim at the classical level.

Consider $\bClAlg(\uc^{2i})\subset\bClAlg_{\infty}$. Consider the
$T$-systems (\ref{eq:T-systems-Coxeter}) for $\bClAlg(\uc^{2i})$.
For $r\in[0,2i-1]$, denote $m:=r-i+1\in[-i+1,i]$. Identify its nodes
$\binom{b}{r}$ with $\mybinom b{r-i+1}=\mybinom am$ and $(b,\xi_{b}-2m)$.
Choose any $k\in[1,|J|]$ and denote $a=c_{k}$. Note that, for $J=J^{+}\sqcup J^{-}$
(i.e., bipartite orientation), we have $\xi_{c_{h}}=\xi_{a}+1$ if
$C_{c_{h},c_{a}}\neq0$ and $h<k$; $\xi_{c_{h}}=\xi_{a}-1$ if $C_{c_{h},c_{a}}\neq0$
and $h>k$. We obtain the following

\begin{align}
 & W_{[(a,\xi_{a}-2m),(a,\xi_{a}-2m-2s)]}*W_{[(a,\xi_{a}-2m-2),(a,\xi_{a}-2m-2s-2)]}\label{eq:T-systems-bipartite}\\
= & q^{\alpha}[W_{[(a,\xi_{a}-2m-2),(a,\xi_{a}-2m-2s)]}*W_{[(a,\xi_{a}-2m),(a,\xi_{a}-2m-2s-2)]}]\nonumber \\
 & +q^{\alpha'}[\prod_{h\neq k}W_{[(c_{h},\xi_{a}-1-2m),(c_{h},\xi_{a}-1-2m-2s)]}^{-C_{c_{h},a}}].\nonumber 
\end{align}
At the classical level, it is identical with the $T$-system for $KR$-polynomials
in \cite[Theorem 6.9]{jang2023quantization}. Recall that we have
identified our interval variables in $\usd_{\infty}$ and with the
KR-polynomials in $\usd$. And these initial cluster variables determine
$F_{q}(w)$ and $W(w)$ by the same sequence of $T$-systems. So $F_{q}(w)=W(w)$
in general.

\end{proof}

Combining Theorem \ref{thm:dBS_PBW} and Lemma \ref{lem:interval-KR},
we obtain that our common triangular basis $\can$ for $\bClAlg_{\infty}$
coincides with the Kazhdan-Lusztig type basis $\{L(w)\}$ in \cite[Theorem 5.27]{jang2023quantization}
(our standard basis differs from that of \cite{jang2023quantization}
by a bar involution). Then Lemma \ref{lem:interval-KR} implies that
 $F_{q}(w)=L(w)$, confirming Theorem \ref{thm:intro-KR-poly-is-KL-basis}.

\section{Applications: fundamental variables via braid group actions\label{sec:Braid-group-action}}

Assume $C$ is of finite type. Take any word $\ueta$ for $\beta\in\Br^{+}$.
We will briefly recall the braid group action in \cite{jang2023braid}
\cite{kashiwara2024braid}, and then use it to compute the fundamental
variables $W_{k}(\rsd(\ueta))$. 

\subsection{Presentations and braid group actions\label{subsec:Presentations-and-braid}}

Choose a Coxeter word $\uc=(c_{1},\ldots,c_{|J|})$ and an associated
height function $\xi$. Recall that $\tau:=w_{\uc}$ denotes the Weyl
group element associated with $\uc$, also called the Auslander--Reiten
translation. Recall that we have $J_{\Z}(\xi):=\{(a,p)|a\in J,p\in\xi_{a}-2\Z\}$.
Let $\Phi^{+}$ denote the set of positive roots of $C$, whose simple
roots are denoted by $\alpha_{b}$, $b\in J$. Denote $q_{a}:=q^{\symm_{a}}$,
$[k]_{q_{a}}:=\frac{q_{a}^{k}-q_{a}^{-k}}{q_{a}-q_{a}^{-1}}$, and
$\sqbinom{n}{k}_{q_{a}}:=\frac{[n]_{q_{a}}!}{[k]_{q_{a}}![n-k]_{q_{a}}!}$.

We recursively define the bijection $\phi^{\xi}:J_{\Z}(\xi)\simeq\Phi^{+}\times\Z$
following \cite{hernandez2015quantum}, such that:
\begin{itemize}
\item $\phi^{\xi}(c_{k},\xi_{c_{k}})=(s_{c_{1}}s_{c_{2}}\cdots s_{c_{k-1}}\alpha_{c_{k}},0)$,
$\forall k\in[1,|J|]$.
\item If $\phi^{\xi}(a,p)=(\gamma,m)$, we have $\phi^{\xi}(a,m\pm2)=\begin{cases}
(\tau^{\mp1}\gamma,m) & \tau^{\mp}\gamma\in\Phi^{+}\\
(-\tau^{\mp}\gamma,m\pm1) & \text{-\ensuremath{\tau^{\mp}\gamma\in\Phi^{+}}}
\end{cases}$.
\end{itemize}
Indeed, one can verify that $\phi^{\xi}$ and $\tau$ are determined
by $\xi$. 

Let $w_{0}$ denote the longest element of the Weyl group. Then there
exists a permutation $\nu$ on $J$, such that $\nu^{2}=\Id$ and
$s_{a}w_{0}=w_{0}s_{\nu(a)}$, $\forall a\in J$. Then, when $\phi^{\xi}(b,p)=(\alpha_{b'},k)$,
we have $\phi^{\xi}(\nu^{i}(b),p+ih)=(\alpha_{b'},k+i)$, where $h$
is the Coxeter number \cite[(38)]{hernandez2015quantum}. 

\begin{Eg}\label{eg:A2-root-nodes}

Take $C=\left(\begin{array}{cc}
2 & -1\\
-1 & 2
\end{array}\right)$. Choose the Coxeter word $\uc=(1,2)$ We have $\tau=w_{\uc}=s_{1}s_{2}$.
In Figure \ref{fig:A2-root-nodes}, we represent $J_{\Z}(\xi)$ using
$\Phi^{+}\times\Z$.
\begin{figure}[h]
\caption{Representing the nodes}
\label{fig:A2-root-nodes}

\subfloat[Nodes in $J_{\Z}(\xi)$]{

 \begin{tikzpicture}  [scale=2,node distance=48pt,on grid,>={Stealth[round]},bend angle=45,      pre/.style={<-,shorten <=1pt,>={Stealth[round]},semithick},    post/.style={->,shorten >=1pt,>={Stealth[round]},semithick},  unfrozen/.style= {circle,inner sep=1pt,minimum size=12pt,draw=black!100,fill=red!100},  frozen/.style={rectangle,inner sep=1pt,minimum size=12pt,draw=black!75,fill=cyan!100},   point/.style= {circle,inner sep=1pt,minimum size=5pt,draw=black!100,fill=black!100},   boundary/.style={-,draw=cyan},   internal/.style={-,draw=red},    every label/.style= {black}] \node (v1) at (0,1) {$(1,\xi_1)$}; \node (v2) at (-0.5,0) {$(2,\xi_2)$}; \node (v3) at (-1,1) {$(1,\xi_1-2)$}; \node (v4) at (-1.5,0) {$(2,\xi_2-2)$}; \node (v5) at (-2,1) {$(1,\xi_1-4)$}; \node (v6) at (-2.5,0) {$(2,\xi_2-4)$};  \end{tikzpicture}}\hfill{}\subfloat[Nodes in $\Phi^{+}\times\Z$]{ \begin{tikzpicture}  [scale=2,node distance=48pt,on grid,>={Stealth[round]},bend angle=45,      pre/.style={<-,shorten <=1pt,>={Stealth[round]},semithick},    post/.style={->,shorten >=1pt,>={Stealth[round]},semithick},  unfrozen/.style= {circle,inner sep=1pt,minimum size=12pt,draw=black!100,fill=red!100},  frozen/.style={rectangle,inner sep=1pt,minimum size=12pt,draw=black!75,fill=cyan!100},   point/.style= {circle,inner sep=1pt,minimum size=5pt,draw=black!100,fill=black!100},   boundary/.style={-,draw=cyan},   internal/.style={-,draw=red},    every label/.style= {black}] \node (v1) at (0,1) {$(\alpha_1,0)$}; \node (v2) at (-0.5,0) {$(\alpha_1+\alpha_2,0)$}; \node (v3) at (-1,1) {$(\alpha_2,0)$}; \node (v4) at (-1.5,0) {$(\alpha_1,-1)$}; \node (v5) at (-2,1) {$(\alpha_1+\alpha_2,-1)$}; \node (v6) at (-2.5,0) {$(\alpha_2,-1)$};\end{tikzpicture}}
\end{figure}

\end{Eg}

Denote $\K=\Q(q^{\Hf})$. Let $\hAq$ be the $\K$-algebra generated
by $y_{a,k}$ for $(a,k)\in J\times\Z$, called Serre generators,
subject to the following relations (see \cite[Theorem 7.3]{hernandez2015quantum}):
\begin{flalign*}
 & \sum_{r,s\geq0,r+s=1-C_{ab}}(-1)^{s}\left[\begin{array}{c}
1-C_{ab}\\
s
\end{array}\right]_{q_{a}}y_{a,k}^{r}y_{b,k}y_{a,k}^{s}=0,\ \forall a\neq b\\
 & y_{a,k}y_{b,k+1}=q^{-(\alpha_{a},\alpha_{b})}y_{b,k+1}y_{a,k}+(1-q^{-(\alpha_{a,}\alpha_{a})})\delta_{a,b}\\
 & y_{a,k}y_{b,d}=q^{(-1)^{k+d}(\alpha_{a},\alpha_{b})}y_{b,d}y_{a,k},\forall d>k+1
\end{flalign*}

Recall that the quantum Grothendieck ring $\Kq$ in \cite{kashiwara2023q}
is isomorphic to $\bClAlg(\usd)$ in \cite{jang2023quantization}
and we have $\bClAlg(\usd)\simeq\bClAlg_{\infty}$, see Section \ref{subsec:Cluster-algebras-identify-Kq}.
By \cite[Theorem 7.2]{jang2023braid}(or \cite{fujita2022isomorphisms}\cite{hernandez2015quantum}),
we have an isomorphism $\tTheta^{\xi}:\hAq\simeq\Kq\otimes\K\simeq\bClAlg(\usd)\otimes\K\simeq\bClAlg_{\infty}\otimes\K$,
called a presentation, such that $\tTheta^{\xi}(y_{b,k})=\can(e_{(a,p)})$,
where $\phi^{\xi}(a,p)=(\alpha_{b},k)$. Denote $\tz_{b,k}^{\xi}:=\can(e_{(a,p)})$,
called the Serre generators.\footnote{Our $\tz_{a,k}$ corresponds to $x_{a,k}$ in \cite{hernandez2015quantum}.
It differs from $z_{a,k}$ in \cite{hernandez2015quantum} by a scalar
multiple, where $z_{a,k}$ is the isoclass of the stalk complex $S_{a}[k]$
in the derived Hall algebra.} Note that $(\tTheta^{\xi})^{-1}$ sends the common triangular basis
$\can$ for $\bClAlg_{\infty}$ to a basis for $\hAq$, denoted $\hdCan$.
Then $\hAq_{\kk}:=\Span_{\kk}\hdCan$ is a $\kk$-algebra, which could
be viewed as the integral form of $\hAq$.

We can choose another Coxeter word $\uc'$ and height function $\xi'$
such that $J_{\Z}(\xi')=J_{\Z}(\xi)$. Then we obtain new bijections
$\phi^{\xi'}:J_{\Z}(\xi)\rightarrow\Phi^{+}\times\Z$ and $\tTheta^{\xi'}:\hAq_{\kk}\simeq\bClAlg_{\infty}$,
such that $\tTheta^{\xi'}(y_{b,k})=\can(e_{(a',p')})$, where $\phi^{\xi'}(a',p')=(\alpha_{b},k)$.
Denote $\tz_{b,k}^{\xi'}:=\can(e_{(a',p')})$. Define the composition
$\tTheta(\xi',\xi):=\tTheta^{\xi'}(\tTheta^{\xi})^{-1}$. Then it
restricts to a permutation on $\can$ for $\bClAlg_{\infty}$, see
\cite[Theorem 7.3]{jang2023braid}\cite{fujita2022isomorphisms}.
Equivalently, we have $\tTheta^{\xi}(\can)=\hdCan=\tTheta^{\xi'}(\can)$.

\begin{align*}
\begin{array}{ccc}
\bClAlg_{\infty} & \overset{\tTheta(\xi',\xi)}{\xrightarrow{\sim}} & \bClAlg_{\infty}\\
\simequ\tTheta^{\xi} &  & \simequ\tTheta^{\xi'}\\
\hAq_{\kk} & = & \hAq_{\kk}
\end{array}
\end{align*}

Particularly, if $a=c_{j}$ for some sink $j\in J^{+}(\uc)$, we can
choose $\uc':=\mu_{a}\uc:=(\uc\backslash\{a\},a)$, and $\xi':=\mu_{a}\xi$
on $J$ such that $\xi'_{b}=\xi_{b}-2\delta_{b,a}$. Then we have
$\tz_{b,k}^{\xi'}=\can(e_{(\phi^{\xi})^{-1}(s_{a}(\alpha_{b}),k)})$
if $b\neq a$, and $\tz_{a,k+1}^{\xi'}=\tz_{a,k}^{\xi}$, see also
\cite[Section 7.2]{jang2023braid}. Denote $(\tz_{b,k}^{\xi})^{(r)}:=\frac{(\tz_{b,k}^{\xi})^{r}}{[r]_{q_{b}}!}$.
We have the following (see \cite[Proof of Proposition 7.4]{jang2023braid}
for the second case):

\begin{align}
\tz_{b,k}^{\xi'} & =\begin{cases}
\tz_{a,k-1}^{\xi} & a=b\\
\frac{1}{q_{a}^{\frac{-C_{ab}}{2}}(q_{a}^{-1}-q_{a})^{-C_{ab}}}\sum_{r+s=-C_{ab}}(-1)^{r}q_{a}^{r}(\tz_{a,k}^{\xi})^{(s)}\tz_{b,k}^{\xi}(\tz_{a,k}^{\xi})^{(r)} & a\neq b
\end{cases}.\label{eq:braid-T-a-category}
\end{align}

\begin{Eg}\label{eg:different-presentation}

Choose $C=\left(\begin{array}{cc}
2 & -1\\
-1 & 2
\end{array}\right)$ and $\uc=(1,2)$. Among the fundamental variables of $\bClAlg_{\infty}$,
we consider those on the $6$ nodes shown in Figure \ref{fig:A2-root-nodes}.
They are the fundamental variables appearing in $\bClAlg(\rsd(\uc^{3}))$
in Examples \ref{eg:A2-interval-variable} and \ref{eg:A2-121212},
whose quantization is given by Example \ref{eg:Lambda}.

Choose $\uc'=\mu_{1}\uc=(2,1)$, $\xi'=\mu_{1}\xi$. Recall that $\xi'_{b}=\xi_{b}-2\delta_{b,1}$.

\begin{figure}[h]
\caption{}
\label{fig:A2-braid-presentations}\subfloat[Serre generators $\tz_{b,k}$]{\begin{tikzpicture}  [scale=2,node distance=48pt,on grid,>={Stealth[round]},bend angle=45,      pre/.style={<-,shorten <=1pt,>={Stealth[round]},semithick},    post/.style={->,shorten >=1pt,>={Stealth[round]},semithick},  unfrozen/.style= {circle,inner sep=1pt,minimum size=12pt,draw=black!100,fill=red!100},  frozen/.style={rectangle,inner sep=1pt,minimum size=12pt,draw=black!75,fill=cyan!100},   point/.style= {circle,inner sep=1pt,minimum size=5pt,draw=black!100,fill=black!100},   boundary/.style={-,draw=cyan},   internal/.style={-,draw=red},    every label/.style= {black}] \node (v1) at (0,1) {$\tz_{1,0}$}; \node (v2) at (-0.5,0) {$W_{\phi^{-1}(\alpha_1+\alpha_2,0)}$}; \node (v3) at (-1,1) {$\tz_{2,0}$}; \node (v4) at (-1.5,0) {$\tz_{1,-1}$}; \node (v5) at (-2,1) {$W_{\phi^{-1}(\alpha_1+\alpha_2,-1)}$}; \node (v6) at (-2.5,0) {$\tz_{2,-1}$}; \end{tikzpicture}}\hfill{}\subfloat[Serre generators $\tz'_{b,k}$]{\begin{tikzpicture}  [scale=2,node distance=48pt,on grid,>={Stealth[round]},bend angle=45,      pre/.style={<-,shorten <=1pt,>={Stealth[round]},semithick},    post/.style={->,shorten >=1pt,>={Stealth[round]},semithick},  unfrozen/.style= {circle,inner sep=1pt,minimum size=12pt,draw=black!100,fill=red!100},  frozen/.style={rectangle,inner sep=1pt,minimum size=12pt,draw=black!75,fill=cyan!100},   point/.style= {circle,inner sep=1pt,minimum size=5pt,draw=black!100,fill=black!100},   boundary/.style={-,draw=cyan},   internal/.style={-,draw=red},    every label/.style= {black}] \node (v1) at (0,1) {$\tz'_{1,1}$}; \node (v2) at (-0.5,0) {$\tz'_{2,0}$}; \node (v3) at (-1,1) {$W_{(\phi')^{-1}(\alpha_1+\alpha_2,0)}$}; \node (v4) at (-1.5,0) {$\tz'_{1,0}$}; \node (v5) at (-2,1) {$\tz'_{2,-1}$}; \node (v6) at (-2.5,0) {$W_{(\phi')^{-1}(\alpha_1+\alpha_2,-1)}$}; \end{tikzpicture}}
\end{figure}

Denote $\tz_{b,k}:=\tz_{b,k}^{\xi}$ and $\tz'_{b,k}:=\tz_{b,k}^{\xi'}$.
Then we have $W_{1}=\tz_{1,0}$, $W_{3}=\tz_{2,0}$, $W_{4}=\tz_{1,-1}$,
$W_{6}=\tz_{2,-1}$, $W_{1}=\tz'_{1,1}$, $W_{2}=\tz'_{2,0}$, $W_{4}=\tz'_{1,0}$,
$W_{5}=\tz'_{2,-1}$, see Figure \ref{fig:A2-braid-presentations}.

Let us verify some relations among the fundamental variables, which
are necessary for $\tTheta^{\xi}:\hAq\simeq\bClAlg_{\infty}\otimes\K$
to be an algebra homomorphism. It is easy to check the following relations:
\begin{align*}
W_{4}*W_{3} & =qW_{3}*W_{4},\\
W_{6}*X_{1} & =qW_{1}*W_{6}.
\end{align*}
Direct computation shows that $W_{4}*W_{1}=q^{-1}(x_{1}\cdot x_{2}^{-1}\cdot x_{4}+x_{2}^{-1}\cdot x_{3})+1$,
from which we deduce 
\begin{align*}
W_{4}*W_{1} & =q^{-2}W_{1}*W_{4}+(1-q^{-2}).
\end{align*}
We can similarly verify that $W_{6}*W_{3}=q^{-2}W_{3}*W_{6}+(1-q^{-2})$.
By direct computations, we obtain the following relations as well:
\begin{align*}
W_{1}^{2}*W_{3}-[2]_{q}W_{1}*W_{3}*W_{1}+W_{1}*W_{3}^{2} & =0,\\
W_{4}^{2}*W_{6}-[2]_{q}W_{4}*W_{6}*W_{4}+W_{6}*W_{4}^{2} & =0.
\end{align*}

Finally, let us verify (\ref{eq:braid-T-a-category}). Note that $W_{1}*W_{3}=q^{\Hf}x_{3}+q^{-\Hf}x_{2}$.
We deduce that $\tz'_{2,0}=W_{2}=\frac{1}{q^{\Hf}(q^{-1}-q)}(W_{1}*W_{3}-qW_{3}*W_{1})=\frac{1}{q^{\Hf}(q^{-1}-q)}(\tz_{1,0}*\tz_{2,0}-q\tz_{2,0}*\tz_{1,0})$.
In addition, we have  $\tz'_{1,1}=W_{1}=\tz_{1,0}$. 

\end{Eg}

Take any $a\in J$. Define $T_{a}$ to the automorphism on $\hat{\mathcal{A}}_{q}(\mathfrak{n})$
such that\footnote{Our $y_{a,k}$ corresponds to $f_{a,-k}$ in \cite{kashiwara2024braid}.}
\begin{align}
T_{a}(y_{b,k}) & =\begin{cases}
y_{a,k-1} & a=b\\
\frac{1}{q_{a}^{\frac{-C_{ab}}{2}}(q_{a}^{-1}-q_{a})^{-C_{ab}}}\sum_{r+s=-C_{ab}}(-1)^{r}q_{a}^{r}y_{a,k}^{(s)}y_{b,k}y_{a,k}^{(r)} & a\neq b
\end{cases}.\label{eq:braid-T-a}
\end{align}
Then, when $a=c_{j}$ for some $j\in J^{+}(\uc)$, $T_{a}$ is determined
by the following commutative diagram by (\ref{eq:braid-T-a-category}):
\begin{align}
\begin{array}{ccc}
\bClAlg_{\infty}\otimes\K & = & \bClAlg_{\infty}\otimes\K\\
\simequ\tTheta^{\xi} &  & \simequ\tTheta^{\mu_{a}\xi}\\
\hAq & \overset{T_{a}}{\xleftarrow{\sim}} & \hAq
\end{array}\label{eq:sink-braid-action}
\end{align}
Note that, when $a\neq b$, the action of $T_{a}$ on $y_{b,k}$ differ
with the action of $T'_{a,-1}$ on $F_{b}$ in \cite[37.1.3]{Lus:intro}
by a multiple in $\Q(q^{\Hf})$. 

By \cite[Theorem 8.1]{jang2023braid}\cite[Theorem 3.1]{kashiwara2024braid},
the braid group $\Br$ acts on $\hAq$ such that $\sigma_{a}$, $a\in J$,
acts by $T_{a}$.

\begin{Eg}\label{eg:A2-braid-fundamental-variables}

Continue Example \ref{eg:different-presentation}. Denote $T'_{a}:=\tTheta^{\xi}T_{a}(\tTheta^{\xi})^{-1}$
so that we compute in the cluster algebra. Denote $\beta:=\frac{1}{q^{\Hf}(q^{-1}-q)}$.
Recall that $\tTheta^{\xi}y_{1,0}=W_{1}$, $\tTheta^{\xi}y_{2,0}=W_{3}$,
$\tTheta^{\xi}y_{1,-1}=W_{4}$, and $\tTheta^{\xi}y_{2,-1}=W_{6}$.

We have seen $T'_{1}W_{3}=\beta(W_{1}*W_{3}-qW_{3}*W_{1})=W_{2}$,
$T'_{1}W_{1}=W_{4}$. We can also compute $T'_{2}W_{1}=\beta(W_{3}*W_{1}-qW_{1}*W_{3})=x_{3}=W_{[1,3]}$.
Let us make more computations.

\begin{align*}
T'_{1}T'_{2}W_{1} & =\beta T'_{1}(W_{3}*W_{1}-qW_{1}*W_{3})=\beta(W_{2}*W_{4}-qW_{4}*W_{2})=W_{3},\\
T'_{1}T'_{2}T'_{1}W_{3} & =\beta T'_{1}T'_{2}(W_{1}*W_{3}-qW_{3}*W_{1})=\beta T'_{1}(x_{3}*W_{6}-qW_{6}*x_{3})=T'_{1}W_{1}=W_{4},\\
T'_{1}T'_{2}T'_{1}T'_{2}W_{1} & =T'_{1}T'_{2}W_{3}=T'_{1}W_{6}=\beta(W_{4}*W_{6}-qW_{6}*W_{4})=W_{5}.
\end{align*}
Finally, using $T'_{1}T'_{2}T'_{1}=T'_{2}T'_{1}T'_{2}$, we obtain
\begin{align*}
T'_{2}T'_{1}W_{3} & =\beta T'_{2}(W_{1}*W_{3}-qW_{3}*W_{1})=\beta(x_{3}*W_{6}-qW_{6}*x_{3})=W_{1},\\
T'_{1}T'_{2}T'_{1}(T'_{2}T'_{1}W_{3}) & =T'_{2}T'_{1}T'_{2}(W_{1})=T'_{2}W_{3}=W_{6}.
\end{align*}

\end{Eg}

\subsection{Canonical cluster structures\label{subsec:Indentical-cluster-str-for-q-Gp}}

We will often abbreviate $(\rsd(\ubi))$ by $(\ubi)$ for simplicity. 

Recall that $\K=\Q(q^{\Hf})$ and $C$ is assumed to be a $J\times J$
Cartan matrix of finite type. Let $\qO[N_{-}]$ denote the quantum
unipotent subgroup associated with the unipotent radical $N_{-}\subset G$,
where $G$ is the associated connected, simply connected, complex
semisimple algebraic group. It is a $\Q(q)$-algebra. It has the dual
canonical basis $\dCan$. Then $\Span_{\kk}\dCan$ is a $\kk$-algebra,
denoted $\kk[N_{-}]$.

Let $\ugamma$ denote a reduced word for $w_{0}$ and $l:=l(w_{0})$.
We define $\Delta:=\beta_{\ugamma}\in\Br^{+}$.

Denote $\rsd=\rsd(\ugamma)$, $\bClAlg=\bClAlg(\rsd)$. By choosing
an appropriate quantization matrix $\Lambda$ for $\rsd$, we have
a $\kk$-algebra isomorphism $\kappa:\bClAlg\simeq\kk[N_{-}]$, sending
the interval variables $W_{[j,k]}$, $1\leq j\leq k\leq l$, to $q^{h_{[j,k]}}D[j,k]$,
where $D[j,k]\in\dCan$, and $h_{[j,k]}\in\Q$ is chosen such that
$q^{h_{[j,k]}}\kappa^{-1}D[j,k]$ is bar-invariant, see \cite[Section 7.4, Lemma 8.2.1]{qin2020dual}\cite{GeissLeclercSchroeer11}\cite[Theorem 7.3]{goodearl2020integral}. 

Particularly, for each $b\in J$, there is a unique $j(b)=\binom{a(b)}{d(b)}^{\ugamma}$
such that $D[j(b),j(b)]$ has weight $-\alpha_{b}$. Denote $q^{h_{[j(b),j(b)]}}D[j(b),j(b)]$
by $q^{h_{b}}D_{b}$, called the Serre generators. Let $\hAq^{[s,r]}$,
$s\leq r$, denote the $\K$-subalgebra of $\hAq$ generated by $y_{b,m}$,
$b\in J$, $m\in[s,r]$. $\forall s\in\Z$, we have the isomorphism
$\qO[N_{-}]\otimes\K\simeq\hAq^{[s,s]}$ such that $\kappa W_{j(b)}=q^{h_{b}}D_{b}$
is identified with $y_{b,s}$. Note that $\hdCan^{[s,r]}:=\hdCan\cap\hAq^{[s,r]}$
is a $\K$-basis of $\hAq^{[s,r]}$. Define the integral form $\hAq_{\kk}^{[s,r]}:=\Span_{\kk}\hdCan^{[s,r]}$,
which is a $\kk$-algebra.

Let $\ugamma'$ denote another reduced word for $w_{0}$. Use $\rsd'$,
$\Lambda'$, $\bClAlg'$, $\kappa'$, $W'$, $D'[j,k]$, $j'(b)$
to denote the associated construction. Note that we have $D'_{b}=D_{b}$
since they are the unique dual canonical basis element with weight
$-\alpha_{b}$.

Denote $\rsd'':=\seq^{\sigma}\rsd$, where $\seq^{\sigma}$ denotes
the permutation mutation sequence $\seq_{\ugamma',\ugamma}^{\sigma}$
in Section \ref{subsec:Operations-on-signed-words}. Use $\Lambda''$,
$\bClAlg''$, $W''$ to denote the associated construction. At the
classical level, the seed $\rsd'$ equals $\rsd''=\seq^{\sigma}\rsd$.
A priori, the quantization matrix $\Lambda''$ for the quantum seed
$\rsd''$ might be different from $\Lambda'$ for $\rsd'$. We have
an isomorphism $(\seq^{\sigma})^{*}:\bClAlg''\simeq\bClAlg$. Consider
the diagram:
\begin{align*}
\begin{array}{ccc}
\bClAlg'' & \stackrel[\text{change}]{\text{quantization}}{\dashrightarrow} & \bClAlg'\\
\simeqd(\seq^{\sigma})^{*} &  & \simeqd\kappa'\\
\bClAlg & \overset{\kappa}{\xrightarrow{\sim}} & \kk[N_{-}]
\end{array}
\end{align*}

\begin{Thm}\label{thm:identical-cluster-q-qp}

We have $\Lambda''=\Lambda'$, i.e., the quantum seed $\rsd':=\rsd(\ubi')$
equals $\rsd'':=\seq^{\sigma}\rsd$ and $\bClAlg''=\bClAlg'$. Moreover,
$\kappa'=\kappa(\seq^{\sigma})^{*}$ or, equivalently, $W_{j(b)}=(\seq^{\sigma})^{*}W'_{j'(b)}$. 

\end{Thm}

\begin{proof}

(i) We claim that $(\seq^{\sigma})^{*}W''_{j'(b)}=W_{j(b)}$. Since
$W''_{j'(b)}$ is a cluster variable of $\bClAlg''$, $(\seq^{\sigma})^{*}W''_{j'(b)}$
is a cluster variable of $\bClAlg$. Therefore, it suffices to verify
this claim at the classical level.

For $\kk=\C$, we have $\tilde{\kappa}:\upClAlg(\dsd(\ugamma))\simeq\C[G^{w_{0},e}]$
and $\tilde{\kappa}':\upClAlg(\dsd(\ugamma'))\simeq\C[G^{w_{0},e}]$
by \cite{BerensteinFominZelevinsky05}. Note that $\dsd(\ugamma')=\seq^{\sigma}\dsd(\ugamma)$
and we have the associated isomorphism $(\seq^{\sigma})^{*}:\upClAlg(\dsd(\ugamma'))\simeq\upClAlg(\dsd(\ugamma))$.
The equality $\tilde{\kappa}'=\tilde{\kappa}(\seq^{\sigma})^{*}$
was conjectured in \cite[Remark 2.14]{BerensteinFominZelevinsky05}
and proved in \cite[Theorem 1.1]{shen2021cluster}, see (\ref{eq:ddBS-unique-cluster}).
When we evaluate the frozen variables $x_{j}$ to $1$ for $j\in I(\dsd(\ugamma))\backslash I(\rsd)$,
$\tilde{\kappa}$ and $\tilde{\kappa}'$ restricts to $\kappa$ and
$\kappa'$ respectively. Moreover, $\rsd''=\rsd'$ and thus $W''_{j'(b)}=W'_{j'(b)}$
at the classical level. So we deduce $\kappa'=\kappa(\seq^{\sigma})^{*}$
and thus $W_{j(b)}=(\seq^{\sigma})^{*}W_{j'(b)}$. So the above claim
is true.

(ii) By (i), we obtain an isomorphism $\kappa''=\kappa(\seq^{\sigma})^{*}:\bClAlg''\simeq\kk[N_{-}]$,
sending $W''_{j'(b)}$ to $q^{h_{b}}D_{b}$. We deduce the $\kk$-algebra
isomorphism $(\kappa')^{-1}\kappa'':\bClAlg''\simeq\bClAlg'$, sending
$W''_{j'(b)}$ to $W'_{j'(b)}$. We claim that $\kappa''W''_{[j,k]}=\kappa'W'_{[j,k]}$.
Particularly, for $x_{i}''=W''_{[i^{\min},i]}$ and $x'_{i}=W'_{[i^{\min},i]}$,
we obtain $\kappa''(x_{i}'')=\kappa'x'_{i}$.

We will prove the claim using arguments similar to those in (i). Since
$W''[j,k]$ is a cluster variable of $\bClAlg''$, $(\seq^{\sigma})^{*}W''_{[j,k]}$
is a cluster variable of $\bClAlg$. By \cite{qin2020dual}, $\kappa''W''_{[j,k]}=\kappa((\seq^{\sigma})^{*}W''_{[j,k]})$
equals $q^{h_{0}}D_{0}$ for some dual canonical basis element $D_{0}$,
where $h_{0}$ is chosen such that $q^{h_{0}}\kappa^{-1}D_{0}$ is
bar-invariant in $\bClAlg$. On the other hand, recall that $\kappa'W'_{[j,k]}=q^{h_{[j,k]}}D'_{[j,k]}$.
It remains to show $q^{h_{0}}D_{0}=q^{h_{[j,k]}}D'[j,k]$.

By \cite{qin2020dual}, $\kappa^{-1}(q^{h_{0}}D_{0})$ and $\kappa^{-1}(q^{h_{[j,k]}}D'[j,k])$
are pointed elements in $\bClAlg\subset\LP(\rsd)$, and they are equal
if and only if they have the same degrees. At the classical level,
we have $W''_{[j,k]}=W'_{[j,k]}$ and thus $D_{0}=\kappa(\seq^{\sigma})W''_{[j,k]}=\kappa'W'_{[j,k]}=D'[j,k]$.
Therefore, they have the same degree at the classical level. We deduce
that $q^{h_{0}}D_{0}=q^{h_{[j,k]}}D'[j,k]$. The desired claim follows. 

Finally, recall that $\Lambda'$ is determined by $x'_{j}*x'_{k}=q^{\Lambda'_{j,k}}x'_{k}*x'_{j}$
and similar for $\Lambda''$. Then Claim (ii) implies that $\Lambda'=\Lambda''$.
The the desired statements follow as consequences.

\end{proof}

We often omit the symbol for mutations among different seeds of the
same cluster algebra. In this convention, Theorem \ref{thm:identical-cluster-q-qp}
could be written as $\kappa'=\kappa$ and $W_{j'(b)}=W_{j(b)}$, i.e,
different choices of the reduced words give the same cluster structure
on $\qO[N_{-}]\otimes\K$.

Next, choose any Coxeter word $\uc$. Let $\ugamma$ denote any chosen
$\uc$-adapted word for $\Delta$. Denote $\nu(\ugamma)=(\nu(\gamma_{1}),\ldots,\nu(\gamma_{l}))$.
Then $\uzeta:=(\ugamma,\nu(\ugamma),\ldots,\nu^{4m-1}(\ugamma))$
is a $\uc$-adapted word for $\Delta^{4m}$. Note that $\uc^{h}$
and $(\ugamma,\nu(\ugamma))$ are connected by braid moves $(a,b)\mapsto(b,a)$
where $C_{ab}=0$. So we have $\rsd(\uzeta)=\rsd(\uc^{2mh})=:\sd_{2mh}$.
Identify $I(\sd_{2mh})\simeq[1,l(\uzeta)]$. Recall that, for $\bClAlg(\uzeta)=\bClAlg(\sd_{2mh})=\bClAlg(\ssd_{mh})=\bClAlg(\usd_{mh})\subset\bClAlg_{\infty}=\bClAlg(\usd_{\infty})$
as in Section \ref{subsec:Limits-of-signed} and Section \ref{subsec:Cluster-algebras-identify-Kq},
$\bClAlg(\uzeta)$ contains $W_{\mybinom ad}$, $d\in[-mh+1,mh]$,
where $\mybinom ad\in\W$ are identified with $(a,\xi_{a}-2d)$. Introduce
$\txi=\xi+2(mh-1)$ and denote $\kappa:=\kappa^{\txi}:=(\tTheta^{\txi})^{-1}$.
Then we obtain the isomorphism 
\begin{align*}
\kappa & :\bClAlg(\sd_{2mh})\xrightarrow{\sim}\hAq_{\kk}^{[-4m+1,0]}
\end{align*}
such that $y_{b,-s}=\kappa W_{j(b)[s]}$, $\forall b\in J,s\in[0,4m-1]$.

Choose other $\uc'$ and $\ugamma'$, we obtain the associated data
such as $\uzeta'$, $\sd'_{2mh}$. Denote $\seq^{\sigma}=\seq_{(\uc')^{2mh},(\uc)^{2mh}}^{\sigma}$.
Choose any $s\in[0,-m+1]$. Define 
\begin{align*}
\ualpha & :=(\uzeta_{[s\cdot l(\Delta)+1,4m\cdot l(\Delta)]},-(\uzeta_{[1,s\cdot l(\Delta)]})\op)=(\nu^{s}(\ugamma),\uzeta_{[(s+1)l(\Delta)+1,4m\cdot l(\Delta)]},-(\uzeta_{[1,s\cdot l(\Delta)]})\op)
\end{align*}
Similarly, define $\ualpha'$ from $\uzeta'$. View $\nu^{s}(\ugamma)$
as $\ualpha_{[1,l(\Delta)]}$ and $\nu^{s}(\ugamma')$ as $\ualpha'_{[1,l(\Delta)]}$.
Then we have the following diagrams, where $\iota_{s}$, $\iota_{s}'$
denote the inclusion in Lemma \ref{lem:calibration-word} induced
by cluster embeddings:
\begin{align*}
\begin{array}{ccccccccccc}
W_{j(b)}(\nu^{s}(\ugamma')) & \in & \bClAlg(\nu^{s}(\ugamma')) & \xhookrightarrow{\iota_{s}'} & \bClAlg(\ualpha') & \overset{\seq_{\ualpha',\uzeta'}^{*}}{\xrightarrow{\sim}} & \bClAlg(\uzeta') & = & \bClAlg(\sd'_{2mh}) & \overset{\kappa'}{\xrightarrow{\sim}} & \hAq_{\kk}^{[-4m+1,0]}\\
 &  & \simeqd(\seq_{\ugamma',\ugamma}^{\sigma})^{*} &  & \simeqd(\seq_{\ualpha',\ualpha}^{\sigma})^{*} &  & \simeqd(\seq^{\sigma})^{*} &  & \simeqd(\seq^{\sigma})^{*} &  & \parallel\\
W_{j(b)}(\nu^{s}(\ugamma)) & \in & \bClAlg(\nu^{s}(\ugamma)) & \xhookrightarrow{\iota_{s}} & \bClAlg(\ualpha) & \overset{\seq_{\ualpha,\uzeta}^{*}}{\xrightarrow{\sim}} & \bClAlg(\uzeta) & = & \bClAlg(\sd_{2mh})) & \overset{\kappa}{\xrightarrow{\sim}} & \hAq_{\kk}^{[-4m+1,0]}
\end{array}
\end{align*}
Note that the three leftmost squares commute (Lemma \ref{lem:words-mutation-seq-connect-seeds}
implies the commutativity of the second square from the left). 

Recall that $W_{j(b)[s]}=\seq_{\ualpha,\uzeta}^{*}\iota_{s}W_{j(b)}(\nu^{s}(\ugamma))$
by definition. So $W_{j(b)}(\nu^{s}(\ugamma))$ has the image $y_{b,-s}$
in $\hAq_{\kk}^{[-4m+1,0]}$. Similarly, $W_{j(b)}(\nu^{s}(\ugamma'))$
has the image $y_{b,-s}$. Recall that $y_{b,-s}$ generate $\hAq^{[s,s]}\simeq\qO[N_{-}]\otimes\K$
over $\K$. Particularly, denoting $\phi:=\kappa\seq_{\ualpha,\uzeta}^{*}\iota_{s}$
and $\phi':=\kappa'\seq_{\ualpha',\uzeta'}^{*}\iota'_{s}$, we have
the following diagram: 
\begin{align}
\begin{array}{ccc}
\bClAlg(\nu^{s}(\ugamma')) & \overset{\phi'}{\xrightarrow{\sim}} & \hAq_{\kk}^{[s,s]}\\
\simeqd(\seq_{\ugamma',\ugamma}^{\sigma})^{*} &  & \parallel\\
\bClAlg(\nu^{s}(\ugamma)) & \overset{\phi}{\xrightarrow{\sim}} & \hAq_{\kk}^{[s,s]}
\end{array}\label{eq:uni-cluster-stru-Aq-s}
\end{align}

\begin{Thm}\label{thm:canonical-cluster-structure}

We have $W_{j(b)[s]}(\sd_{2mh})=(\seq^{\sigma})^{*}W{}_{j'(b)[s]}(\sd'_{2mh})$
or, equivalently, $\kappa'=\kappa(\seq^{\sigma})^{*}$.

\end{Thm}

\begin{proof}

It suffices to show that Diagram (\ref{eq:uni-cluster-stru-Aq-s})
is commutative. We could deduce its commutativity from Theorem \ref{thm:identical-cluster-q-qp}.

More precisely, introduce the isomorphism $\psi:\hAq_{\kk}^{[s,s]}\simeq\kk[N_{-}]$
such that $\psi(y_{b,-s})=q^{h_{b}}D_{b}$. Then we have the diagram
$\begin{array}{ccc}
\bClAlg(\nu^{s}(\ugamma')) & \overset{\psi\phi'}{\xrightarrow{\sim}} & \kk[N_{-}]\\
\simeqd(\seq_{\ugamma',\ugamma}^{\sigma})^{*} &  & \parallel\\
\bClAlg(\nu^{s}(\ugamma)) & \overset{\psi\phi}{\xrightarrow{\sim}} & \kk[N_{-}]
\end{array}$ such that $\psi\phi'(W_{j'(b)}(\nu^{s}(\ugamma'))=q^{h_{b}}D_{b}=\psi\phi(W_{j(b)}(\nu^{s}(\ugamma))$.
Then Theorem \ref{thm:identical-cluster-q-qp} implies that this diagram
is commutative. The commutativity of Diagram (\ref{eq:uni-cluster-stru-Aq-s})
follows as a consequence.

\end{proof}

\subsection{Interval variables via braid group actions\label{subsec:invertal-via-braid}}

By \cite[Lemma 8.16]{qin2023analogs}, for any given word $\ueta$,
we have $\ueta\leq_{R}\Delta^{4m}$ for some $m\in\N$, and we can
choose a word $\uzeta$, such that $\beta_{\ueta}=\Delta^{4m}$ and
$\ueta=\uzeta_{[1,l(\ueta)]}$. Choose any Coxeter word $\uc$. Then
$\uc^{2mh}$ is a word for $\Delta^{4m}$. We denote $\sd_{2mh}=\rsd(\uc^{2mh})$
as before. Recall that we have $\bClAlg(\ueta)\subset\bClAlg(\uzeta)=\bClAlg(\sd_{2mh})\overset{\kappa}{\simeq}\hAq_{\kk}^{[-4m+1,0]}$,
where $\kappa=(\tTheta^{\txi})^{-1}$ such that $\kappa W_{j(b)[s]}(\sd_{2mh})=y_{b,-s}$,
$\forall b\in J,s\in[0,4m-1]$. Note that, by the equality $\bClAlg(\uzeta)=\bClAlg(\sd_{2mh})$,
we identify these two cluster algebras via the isomorphism $(\seq_{\uzeta,\uc^{2mh}}^{\sigma})^{*}$
associated with the permutation mutation sequence $\seq_{\uzeta,\uc^{2mh}}^{\sigma}$.

\begin{Thm}\label{thm:braid-formula-fundamental}

For any $k\in[1,l(\ueta)]$, $W_{k}(\rsd(\ueta))$ in $\bClAlg_{k}(\rsd(\ueta))\subset\bClAlg(\uzeta)=\bClAlg(\sd_{2mh})$
satisfies $\kappa W_{k}(\rsd(\ueta))=(T_{\eta_{1}}\cdots T_{\eta_{k-1}}y_{\eta_{k},0})$.

\end{Thm}

We refer the reader to Example \ref{eg:A2-braid-fundamental-variables}
for an example in type $A_{2}$.

\begin{proof}

By Theorem \ref{thm:canonical-cluster-structure}, if we choose a
different Coxeter word $\uc'$, then $\kappa'=\kappa(\seq^{\sigma})^{*}$,
i.e., the associated $\kappa'$ is identified with $\kappa$ via the
mutation $(\seq_{(\uc')^{2mh},(\uc)^{2mh}}^{\sigma})^{*}$, where
$\rsd((\uc')^{2mh})=\seq_{(\uc')^{2mh},(\uc)^{2mh}}^{\sigma}\sd_{2mh}$.
So the statement holds for $\uc$ if and only if it holds for $\uc'$,
and we can make any choice. We prove the statement by induction on
$l(\ueta)$. 

(1) When $l(\ueta)=1$. We can choose $\uc$ such that $c_{1}=\eta_{1}$.
Then $W_{1}(\rsd(\eta_{1}))=x_{1}(\rsd(\eta_{1}))$ coincides with
$x_{1}(\sd_{2mh})$, which is sent to $y_{\eta_{1},0}$ by $\kappa$.

(2) Assume the claim has been prove for length $l(\ueta)-1$. We choose
$\uc$ such that $c_{1}$ equals $\eta_{1}$ and it is a sink. The
case $k=1$ can be proved as in (1). We now assume $k\geq2$.

Take $\uc'=\mu_{\eta_{1}}\uc$ and $\xi':=\mu_{\eta_{1}}\xi$. We
have the following diagrams:
\begin{align*}
\begin{array}{ccccccc}
\bClAlg(\ueta) & \subset & \bClAlg(\uzeta) & \overset{(\seq^{\sigma})^{*}}{\xrightarrow{\sim}} & \bClAlg(\uc^{2mh})\\
\simequ(\seq^{\sigma})^{*} &  & \simequ(\seq^{\sigma})^{*} &  & \simequ(\seq^{\sigma})^{*}\\
\bClAlg(\ueta_{[2,l(\ueta)]},-c_{1}) &  & \bClAlg(\uzeta_{[2,4ml(\Delta)]},-c_{1}) & \overset{(\seq^{\sigma})^{*}}{\xrightarrow{\sim}} & \bClAlg((\uc^{2mh})_{[2,4ml(\Delta)]},-c_{1})\\
\cup &  & \cup &  & \cup\\
\bClAlg(\ueta_{[2,l(\ueta)]}) & \subset & \bClAlg(\uzeta_{[2,4ml(\Delta)]}) & \overset{(\seq^{\sigma})^{*}}{\xrightarrow{\sim}} & \bClAlg((\uc^{2mh})_{[2,4ml(\Delta)]}) & \subset & \bClAlg((\uc')^{2mh})
\end{array}
\end{align*}
where we view $(\uc^{2mh})_{[2,4ml(\Delta)]}$ as the subword $((\uc')^{2mh})_{[1,4ml(\Delta)-1]}$,
all inclusions are induced from cluster embeddings of subwords (Lemma
\ref{lem:calibration-word}), and $\seq^{\sigma}$ denote the (different)
permutation mutation sequences connecting signed words. The diagram
is commutative: By tracking the fundamental variables via Lemma \ref{lem:embed-interval-variables},
we obtain the commutativity of the left most square; The commutativity
of the lower right square is obvious because the two sequences $\seq^{\sigma}$
appearing are the same; The commutativity of the upper right square
is implied by Lemma \ref{lem:words-mutation-seq-connect-seeds}.

By Lemma \ref{lem:embed-interval-variables}, $W_{k}(\ueta)\in\bClAlg(\ueta)\subset\bClAlg(\uc^{2mh})$
is identified with $W_{k-1}(\ueta_{[2,l(\ueta)]})\in\bClAlg(\ueta_{[2,l(\ueta)]})\subset\bClAlg((\uc')^{2mh})$.

By induction hypothesis, we have $\kappa'W_{k-1}(\ueta_{[2,l(\ueta)]})=T_{\eta_{2}}\cdots T_{\eta_{k-1}}(y_{\eta_{k},0})$
, where we denote $\txi'=\mu_{\eta_{1}}\txi$ and $\kappa':=(\tTheta^{\xi'})^{-1}:\bClAlg((\uc')^{2mh})\simeq\hAq_{\kk}^{[-4m+1,0]}$.
By Diagram (\ref{eq:sink-braid-action}), we have $(\kappa')^{-1}(y_{b,s})=\kappa^{-1}T_{\eta_{1}}y_{b,s}$,
for any $b,s$. Therefore, $\kappa W_{k}(\ueta)=\kappa(\kappa')^{-1}\kappa'W_{k-1}(\ueta_{[2,l(\ueta)]})=\kappa(\kappa')^{-1}T_{\eta_{2}}\cdots T_{\eta_{k-1}}y_{\eta_{k},0}=T_{\eta_{1}}T_{\eta_{2}}\cdots T_{\eta_{k-1}}y_{\eta_{k},0}$.

\end{proof}

By Theorem \ref{thm:braid-formula-fundamental}, $\bClAlg(\rsd(\ueta))\otimes\K$
is isomorphic to $\hat{\mathcal{A}}(\beta_{\ueta})$ recently introduced
in \cite{oh2024pbw}. Particularly, Theorem \ref{thm:intro-OH-conj}
is true: $\hat{\mathcal{A}}(\beta_{\ueta})$ is a cluster algebra
and has monoidal categorification.

\section{Applications: Cluster algebras from shifted quantum affine algebras\label{sec:Cluster-algebras-from-shifted}}

In this section, we use our extension approach in Section \ref{sec:Based-cluster-algebras-infinite-ranks}
to realize and quantize the infinite rank cluster algebras $\bClAlg(\sd^{\GHL})$
introduced in \cite{geiss2024representations}, which arise from representations
of shifted quantum affine algebras. We will assume $C$ is a Cartan
matrix of type $ADE$ as in \cite{geiss2024representations}, though
our approach should work in non-simply laced as well.

\subsection{Recovering the seeds}

Choose any Coxeter word $\uc$ and a height function $\xi$ as in
Section \ref{subsec:Cluster-structures-quantum-affine}. Let $\ueta=(\eta_{1},\ldots,\eta_{l})$,
$\uzeta=(\zeta_{1},\ldots,\zeta_{l})$, be any reduced words for the
longest Weyl group element $w_{0}$. 

Let $\ubi$ denote any shuffle of $(\ueta,-\uzeta)$. For $s\in\N$,
define the word $\ubi^{(s)}:=((\uc)^{s},\ubi,(\nu(\uc))^{s})$ and
the seed $\dsd_{s}:=\dsd(\ubi^{(s)})$. Denote $\phi:I(\ubi^{(s)})\simeq[1,l(\ubi^{(s)})]$
as before. We identify $\ddI(\ubi^{(s)})$ with a subset of $\Z$
such that $\binom{a}{d}$ is sent to $\phi\binom{a}{d}-s\cdot|J|$
if $d\geq0$ and $\binom{c_{k}}{-1}$ are sent to $k-(s+1)\cdot|J|$,
$\forall k\in[1,|J|]$.

Note that $l(\ubi^{(s+1)})=(2s+2)\cdot|J|+l(\ubi)$. When we view
$\ubi^{(s)}$ as the subword $(\ubi^{(s+1)})_{[|J|+1,(2s+1)|J|+l(\ubi)]}$,
$\dsd_{s}$ becomes a good sub seed of $\dsd_{s+1}$. Note that the
corresponding cluster embedding $\iota:I(\dsd_{s})\hookrightarrow I(\dsd_{s+1})$
sends $j\in\Z$ to $j$. We have $\cup_{s}I(\dsd_{s})=\Z$. 

Let $\dsd_{\infty}$ denote the colimit of the chain of seeds $(\dsd_{s})_{s\in I}$,
denoted $\dsd_{\infty}=\dsd(\ubi^{(\infty)})$ where $\ubi^{(\infty)}=(\cdots,\uc,\uc,\ubi,\nu(\uc),\nu(\uc),\ldots)$.
If $\ubi'$ is another choice of the signed word, let $(\ubi')^{(s)}$,
$\dsd'_{s}$, and $\dsd'_{\infty}$ denote the corresponding constructions.
Then we have $\dsd'_{s}=\seq_{\ubi',\ubi}^{\sigma}\dsd_{s}$ and $\dsd'_{\infty}=\seq_{\ubi',\ubi}^{\sigma}\dsd_{\infty}$,
where $\seq_{\ubi',\ubi}^{\sigma}$ consists of a mutation sequence
on $U=\{k\in[1,l(\ubi)]|k[1]\leq l\}$ and a permutation $\sigma$
on $[1,l(\ubi)]$ such that $\sigma(U)=U$. So we have the isomorphism
$(\seq_{\ubi',\ubi}^{\sigma})^{*}:\bClAlg(\dsd'_{\infty})\simeq\bClAlg(\dsd_{\infty})$.

\begin{Prop}\label{prop:identify-seed-GHL}

If $\ueta$ is $\uc$-adapted and $\ubi=(\eta_{1},-\eta_{1},\eta_{2},-\eta_{2},\ldots,\eta_{l},-\eta_{l})$,
we can naturally identify $\dsd_{\infty}=\dsd(\ubi^{(\infty)})$ with
the seed $\sd^{\GHL}$ in \cite[Section 3.4]{geiss2024representations}.

\end{Prop}

\begin{proof}

We could identify $\tB(\dsd_{\infty})$ with the $B$-matrix of the
seed $\sd^{\GHL}$. The claim follows.

\end{proof}

\begin{Eg}

We continue Example \ref{eg:SL3-w0-w0}, where $\uc=(1,2)$ and $\ubi=(1,-1,2,-2,1,-1)$.
A quiver for $\dsd_{1}:=\dsd(\ubi^{(1)})$ is shown in Figure \ref{fig:dsd-ubi-1},
where we view $\ddI(\ubi^{(1)})$ as a subset of $\Z$. A quiver for
$\dsd_{\infty}$ is shown in Figure \ref{fig:dsd-inf}, where we identify
$\ddI(\ubi^{(\infty)})$ with $\Z$. Note that our quivers are opposite
to those of \cite{geiss2024representations}.

\end{Eg}

\begin{figure}[h]

\caption{A quiver for $\dsd_{1}=\dsd(1,2,1,-1,2,-2,1,-1,2,1)$}
\label{fig:dsd-ubi-1}

\begin{tikzpicture}  [scale=1,node distance=48pt,on grid,>={Stealth[length=4pt,round]},bend angle=45, inner sep=0pt,      pre/.style={<-,shorten <=1pt,>={Stealth[round]},semithick},    post/.style={->,shorten >=1pt,>={Stealth[round]},semithick},  unfrozen/.style= {circle,inner sep=1pt,minimum size=1pt,draw=black!100,fill=red!50},  frozen/.style={rectangle,inner sep=1pt,minimum size=12pt,draw=black!75,fill=cyan!50},   point/.style= {circle,inner sep=0pt, outer sep=1.5pt,minimum size=1.5pt,draw=black!100,fill=black!100},   boundary/.style={-,draw=cyan},   internal/.style={-,draw=red},    every label/.style= {black}]        
\node[unfrozen] (q-1)  at (3.5,0.5) {-1}; \node[unfrozen] (q-2)  at (3,-0.5) {0}; \node[unfrozen] (q1) at (2.5,0.5) {1}; \node[unfrozen] (q2) at (1.5,0.5) {2}; \node[unfrozen] (q3) at (1,-0.5) {3}; \node[unfrozen] (q4) at (0,-0.5) {4}; \node[unfrozen] (q5) at (-0.5,0.5) {5}; \node[unfrozen] (q6) at (-1.5,0.5) {6}; \draw[->,teal]  (q-1) edge (q1); \draw[->,teal]  (q2) edge (q1); \draw[->,teal]  (q2) edge (q5); \draw[->,teal]  (q6) edge (q5); \draw[->,teal]  (q5) edge (q4); \draw[->,teal]  (q4) edge (q3); \draw[->,teal]   (q-2) edge (q3); \draw[->,teal]  (q3) edge (q2); \draw[->,teal]   (q1) edge (q-2); \draw[->,teal]  (q-2) edge (q-1);
\node[frozen] (q-3) at (5.5,0.5) {-3}; \node[frozen] (q-4) at (5,-0.5) {-2}; \node[frozen](q9) at (-2,-0.5) {7}; \node[frozen] (q10) at (-2.5,0.5) {8}; \draw[->,teal]  (q-1) edge (q-4); \draw [->,teal] (q-4) edge (q-2); \draw[->,teal]  (q-3) edge (q-1); \draw[->,teal]  (q4) edge (q9); \draw[->,teal]  (q9) edge (q6); \draw [->,teal] (q6) edge (q10); \draw[->,dotted,teal]  (q-4) edge (q-3); \draw [->,dotted,teal] (q10) edge (q9); \end{tikzpicture}

\end{figure}

\begin{figure}[h]
\caption{A quiver for $\dsd_{\infty}=\dsd(\ldots,1,2,1,2,1,-1,2,-2,1,-1,2,1,2,1,\ldots)$}
\label{fig:dsd-inf}

 \begin{tikzpicture}  [scale=1,node distance=48pt,on grid,>={Stealth[length=4pt,round]},bend angle=45, inner sep=0pt,      pre/.style={<-,shorten <=1pt,>={Stealth[round]},semithick},    post/.style={->,shorten >=1pt,>={Stealth[round]},semithick},  unfrozen/.style= {circle,inner sep=1pt,minimum size=1pt,draw=black!100,fill=red!50},  frozen/.style={rectangle,inner sep=1pt,minimum size=12pt,draw=black!75,fill=cyan!50},   point/.style= {circle,inner sep=0pt, outer sep=1.5pt,minimum size=1.5pt,draw=black!100,fill=black!100},   boundary/.style={-,draw=cyan},   internal/.style={-,draw=red},    every label/.style= {black}]        
\node[unfrozen] (q-1)  at (3.5,0.5) {-1}; \node[unfrozen] (q-2)  at (3,-0.5) {0}; \node[unfrozen] (q1) at (2.5,0.5) {1}; \node[unfrozen] (q2) at (1.5,0.5) {2}; \node[unfrozen] (q3) at (1,-0.5) {3}; \node[unfrozen] (q4) at (0,-0.5) {4}; \node[unfrozen] (q5) at (-0.5,0.5) {5}; \node[unfrozen] (q6) at (-1.5,0.5) {6}; \draw[->,teal]  (q-1) edge (q1); \draw[->,teal]  (q2) edge (q1); \draw[->,teal]  (q2) edge (q5); \draw[->,teal]  (q6) edge (q5); \draw[->,teal]  (q5) edge (q4); \draw[->,teal]  (q4) edge (q3); \draw[->,teal]   (q-2) edge (q3); \draw[->,teal]  (q3) edge (q2); \draw[->,teal]   (q1) edge (q-2); \draw[->,teal]  (q-2) edge (q-1);
\node[unfrozen] (q-3) at (5.5,0.5) {-3}; \node[unfrozen] (q-4) at (5,-0.5) {-2}; \node[unfrozen](q9) at (-2,-0.5) {7}; \node[unfrozen] (q10) at (-2.5,0.5) {8}; \draw[->,teal]  (q-1) edge (q-4); \draw [->,teal] (q-4) edge (q-2); \draw[->,teal]  (q-3) edge (q-1); \draw[->,teal]  (q4) edge (q9); \draw[->,teal]  (q9) edge (q6); \draw [->,teal] (q6) edge (q10); \draw[->,teal]  (q-4) edge (q-3); \draw [->,teal] (q10) edge (q9); \node[unfrozen] (q-5) at (6,-0.5) {-4}; \node[unfrozen] (q-6) at (6.5,0.5) {-5}; \node at (7,0) {$\ldots$}; \node[unfrozen] (q11) at (-3,-0.5) {9}; \node[unfrozen] (q12) at (-3.5,0.5) {10}; \node at (-4,0) {$\ldots$}; \draw [->,teal]   (q10) edge (q12); \draw [->,teal]   (q12) edge (q11); \draw [->,teal]   (q11) edge (q10); \draw  [->,teal]  (q9) edge (q11); \draw [->,teal]   (q-6) edge (q-3); \draw  [->,teal]  (q-5) edge (q-4); \draw [->,teal]   (q-3) edge (q-5); \draw  [->,teal]  (q-5) edge (q-6); \end{tikzpicture}
\end{figure}

\subsection{Quantization}

\begin{Lem}\label{lem:dsd-extend-condition}

$\dsd_{s}$ and $\dsd_{s+1}$ satisfies the conditions in Lemma \ref{lem:extend-quantization}.

\end{Lem}

\begin{proof}

Using the convention in Lemma \ref{lem:extend-quantization}, we denote
$\sd=\dsd_{s+1}$ and $\sd'=\dsd_{s}$. We observe that $I_{1}=I_{\ufv}(\dsd(\ubi^{(s)}))$,
$I_{2}=I_{\fv}(\dsd(\ubi^{(s)}))$, $I_{1}\sqcup I_{2}=I_{\ufv}(\dsd(\ubi^{(s+1)}))$,
and $I_{3}=I_{\fv}(\dsd(\ubi^{(s+1)}))$. Moreover, we have $|I_{2}|=|I_{3}|=2|J|$
and $B_{31}=0$. It remains to check that $B_{32}$ is of full rank.

For any $k\in[1,|J|]$. Denote $v_{k}^{-}:=\binom{c_{k}}{-1}$ and
$v_{k}^{+}:=\nu(c_{k})^{\max}=\binom{\nu(c_{k})}{O([1,l(\ubi^{(s+1)})];\nu(c_{k}))-1}$.
Then $I_{3}=\{v_{k}^{-}|k\in[1,|J|]\}\sqcup\{v_{k}^{+}|k\in[1,|J|]\}$,
$I_{2}=\{v_{k}^{-}[1]|k\in[1,|J|]\}\sqcup\{v_{k}^{+}[-1]|k\in[1,|J|]\}$.
We can compute the column vectors explicitly: $\col_{v_{k}^{-}[1]}B_{32}=f_{v_{k}^{-}}+\sum_{j>k}C_{c_{j}c_{k}}f_{v_{j}^{-}}$
and $\col_{v_{k}^{+}[-1]}B_{32}=-f_{v_{k}^{+}}-\sum_{j<k}C_{\nu(c_{j}),\nu(c_{k})}f_{v_{j}^{+}}$.
It is straightforward to check that these column vectors are linearly
independent, i.e., $B_{32}$ is of full rank.

\end{proof}

Note that $\dsd_{0}=\dsd(\ubi)$ is a seed for the coordinate ring
of the double Bruhat cell $G^{w_{0},w_{0}}$ \cite{BerensteinFominZelevinsky05}.
Moreover, by \cite{BerensteinZelevinsky05}, we can associate a quantization
matrix $\Lambda^{(0)}$ to $\dsd_{0}$. Then we can uniquely extend
it to a quantization matrix $\Lambda^{(s)}$ for any $\dsd_{s}$ by
Lemma \ref{lem:extend-quantization} and Lemma \ref{lem:dsd-extend-condition}.
So we obtain the following.

\begin{Prop}\label{prop:quantize-seed-GHL}

We can uniquely extend the quantization matrix $\Lambda^{(0)}$ for
$\dsd_{0}$ to a quantization matrix $\Lambda_{\infty}$ for $\dsd_{\infty}$.
In particular, we obtain an infinite rank quantum cluster algebra
$\bClAlg(\dsd_{\infty})$.

\end{Prop}

Finally, the upper cluster algebras $\bUpClAlg(\dsd_{s})$ has the
common triangular basis whose structure constants are non-negative,
see \cite[Theorem 6.17]{qin2023analogs}. Moreover, Lemma \ref{lem:extend-bUpClAlg}
implies that $\bUpClAlg(\dsd_{\infty})=\cup_{i}\bUpClAlg(\dsd_{s})$.
So we obtain the following result, which is related to the categorification
conjecture for $\bClAlg(\dsd_{\infty})$ \cite[Conjecture 9.16]{geiss2024representations}.

\begin{Thm}\label{thm:basis-seed-GHL}

The quantum cluster algebra $\bUpClAlg(\dsd_{\infty})$ has the common
triangular basis, whose structure constants are non-negative. 

\end{Thm}

\appendix

\section{Double Bott--Samelson cells\label{sec:Cluster-algebras-from-dBS}}

Choose and fix a $J\times J$ generalized Cartan matrix $C$. Take
any words $\uzeta$, $\ueta$ in $J$. Let $\ddBS_{\beta_{\ueta}}^{\beta_{\uzeta}}$
denote the\emph{ decorated double Bott--Samelson cell} (ddBS for
short) in \cite[Section 2]{shen2021cluster}, and let $\dBS_{\beta_{\ueta}}^{\beta_{\uzeta}}$
denote the associated \emph{(half-decorated) double Bott--Samelson
cell} (dBS for short). Choose any shuffle $\ubi$ of $-\uzeta$ and
$\ueta$. Take $\kk=\C$. Then the coordinate ring $\C[\ddBS_{\beta_{\ueta}}^{\beta_{\uzeta}}]$
is isomorphic to $\upClAlg(\dsd(\ubi))$ and $\C[\dBS_{\beta_{\ueta}}^{\beta_{\uzeta}}]$
is isomorphic to $\upClAlg(\rsd(\ubi))$, see \cite[Theorem 1.1]{shen2021cluster}\cite[Section 2.4]{shen2021cluster}.

Let us briefly recall $\ddBS_{\beta_{\ueta}}^{\beta_{\uzeta}}$ for
the reader's convenience. $\dBS_{\beta_{\ueta}}^{\beta_{\uzeta}}$
could be constructed similarly or via the formalism in \cite{casals2022cluster}.

Consider a pair of Kac-Peterson group $G_{\mathrm{sc}}$ and $G_{\mathrm{ad}}$,
where $G_{\mathrm{sc}}$ is a connected, simply connected, complex
semisimple algebraic group when $C$ is of finite type. Take $G=G_{\mathrm{sc}}$
and $B_{\pm}$ its Borel subgroups. For any sign $\varepsilon$, let
$\cB_{\varepsilon}$ denote the flag variety, i.e., the set consisting
of all Borel subgroups conjugate to $B_{\varepsilon}$. Since $G$
acts transitively on $\cB_{\varepsilon}$ by conjugation, such that
the stabilizer of $B_{\varepsilon}$ is $B_{\varepsilon}$, $\cB_{\varepsilon}$
is isomorphic to the cosets:
\begin{align*}
\cB_{\varepsilon} & \simeq G/B_{\varepsilon}\simeq B_{\varepsilon}\backslash G.
\end{align*}
These isomorphisms identify cosets $xB_{\varepsilon}\simeq B_{\varepsilon}x^{-1}$
with the corresponding flags.

Denote the maximal unipotent subgroups $U_{\varepsilon}:=[B_{\varepsilon},B_{\varepsilon}]$
and define the decorated flag varieties 
\begin{align*}
\cA_{+} & :=G/U_{+},\ \cA_{-}:=U_{-}\backslash G,
\end{align*}
whose elements are called decorated flags. Then we have natural projections
$\pi:\cA_{+}\twoheadrightarrow G/B_{+}\simeq\cB_{+}$, $\pi:\cA_{-}\twoheadrightarrow B_{-}\backslash G\simeq\cB_{-}$,
sending decorated flags to flags.

Let $N$ denote the normalizer of the torus $T=B_{+}\cap B_{-}$,
and $W=N/T$ the Weyl group with the identity $e$. Then $G$ has
the Bruhat decomposition $G=\sqcup_{w}B_{+}wB_{+}$ and the Birkhoff
decomposition $G=\sqcup_{w}B_{-}wB_{+}$. For any pair of flags $(xB_{\varepsilon},yB_{\varepsilon})$,
we denote $xB_{\varepsilon}\xar{w}yB_{\varepsilon}$ if $x^{-1}y\in B_{\varepsilon}wB_{\varepsilon}$.
Similarly, for $(xB_{-},yB_{+})$, we denote $xB_{-}\xline{w}yB_{+}$
if $x^{-1}y\in B_{-}wB_{+}$. We will omit the symbol $e$ when $xB_{-}\xline{e}yB_{+}$.

Note that $G$ naturally acts on the configurations of decorated flags
$A^{0}\in\cA_{+}$, $B^{1},\ldots,B^{l(\uzeta)}\in\cB_{+}$, $B_{0},\ldots B_{l(\ueta)-1}\in\cB_{-}$,
$B_{l(\ueta)}\in\cA_{-}$, such that

\begin{align*}
\xymatrix{A^{0}\ar[r]^{s_{j_{1}}}\ar@{-}[d] & B^{1}\ar[r]^{s_{j_{2}}} & B^{2}\ar[r] & \cdots\ar[r] & B^{l(\uzeta)-1}\ar[r]^{s_{k_{l(\uzeta)}}} & B^{l(\uzeta)}\ar@{-}[d]\\
B_{0}\ar[r]^{s_{k_{1}}} & B_{1}\ar[r]^{s_{k_{2}}} & B_{2}\ar[r] & \cdots\ar[r] & B_{l(\ueta)-1}\ar[r]^{s_{k_{l(\ueta)}}} & A_{l(\ueta)}
}
.
\end{align*}
Here, $B_{0}\xline{}A^{0}$ means $B_{0}\xline{}\pi(A^{0})$, and
$A_{l(\ueta)}\xline{}B^{l(\uzeta)}$ means $\pi(A_{l(\ueta)})\xline{}B^{l(\uzeta)}$.
For any $(\uzeta,\ueta)$, the ddBS $\ddBS_{\beta_{\ueta}}^{\beta_{\uzeta}}$
is defined to be the moduli space of $G$-equivalence classes of such
configurations \cite[Definition 2.21]{shen2021cluster}. It does not
depend on the choice of the words $\uzeta,\ueta$ for the positive
braids $\beta_{\uzeta},\beta_{\ueta}$. 

For cluster algebras defined over $\C$, Shen and Weng constructed
explicit algebra isomorphisms $\kappa:\clAlg(\dsd(\ubi))\simeq\C[\ddBS_{\beta_{\ueta}}^{\beta_{\uzeta}}]$
in \cite{shen2021cluster}. Moreover, the isomorphism does not depend
on the choice of $\ubi$ by \cite[Theorem 1.1]{shen2021cluster},
i.e., the following diagram is commutative:
\begin{align}
\begin{array}{ccc}
\clAlg(\dsd(\ubi')) & \overset{\kappa'}{\xrightarrow{\sim}} & \C[\ddBS_{\beta_{\ueta}}^{\beta_{\uzeta}}]\\
\simeqd(\seq_{\ubi',\ubi}^{\sigma})^{*} &  & \parallel\\
\clAlg(\dsd(\ubi)) & \overset{\kappa}{\xrightarrow{\sim}} & \C[\ddBS_{\beta_{\ueta}}^{\beta_{\uzeta}}]
\end{array}.\label{eq:ddBS-unique-cluster}
\end{align}
Replacing $\dsd$ by $\rsd$ and $\ddBS_{\beta_{\ueta}}^{\beta_{\uzeta}}$
by $\dBS_{\beta_{\ueta}}^{\beta_{\uzeta}}$, the above construction
induces algebra isomorphism $\clAlg(\rsd(\ubi))\simeq\C[\dBS_{\beta_{\ueta}}^{\beta_{\uzeta}}]$,
still denoted by $\kappa$. We have have the following commutative
diagram
\begin{align}
\begin{array}{ccc}
\clAlg(\rsd(\ubi')) & \overset{\kappa'}{\xrightarrow{\sim}} & \C[\dBS_{\beta_{\ueta}}^{\beta_{\uzeta}}]\\
\simeqd(\seq_{\ubi',\ubi}^{\sigma})^{*} &  & \parallel\\
\clAlg(\rsd(\ubi)) & \overset{\kappa}{\xrightarrow{\sim}} & \C[\dBS_{\beta_{\ueta}}^{\beta_{\uzeta}}]
\end{array}.\label{eq:dBS-unique-cluster}
\end{align}

\section{Skew-symmetric bilinear forms\label{sec:Skew-symmetric-bilinear-forms}}

We recall the skew-symmetric bilinear forms on $\oplus_{\mybinom ad\in\W}\Z e_{\mybinom ad}$
for the reader's convenience.

We first recall the bilinear form $\cN$ used in \cite[Section 7]{qin2017triangular}.
Consider symmetric generalized Cartan matrix $C$. Choose a Coxeter
word $\uc$. Define the $J\times J$ deformed Cartan matrix $C_{q}$
such that $(C_{q})_{aa}=z+z^{-1}$, $(C_{q})_{c_{j}c_{k}}=C_{c_{j}c_{k}}\cdot z^{-1}$
if $j<k\in J$, and $(C_{q})_{c_{j}c_{k}}=C_{c_{j}c_{k}}\cdot z$
if $j>k\in J$. Let $C_{q}^{-1}$ denote its inverse whose entries
are formal Laurent series in $z$. Denote $C_{q}^{-1}=\sum_{m\in\Z}C_{q}^{-1}(m)z^{m}$
where $C_{q}^{-1}(m)$ are $\Z$-matrices. The bilinear form $\cN$
on $\oplus_{\mybinom ad\in\W}\Z e_{\mybinom ad}$ is given by 
\begin{align*}
\cN(e_{\mybinom ad},e_{\mybinom bh})= & (C_{q}^{-1})_{ab}(-1-2d+2h)-(C_{q}^{-1})_{ba}(-1-2h+2d)\\
 & -(C_{q}^{-1})_{ab}(1-2d+2h)+(C_{q}^{-1})_{ba}(1-2h+2d).
\end{align*}

Further assume $J=J^{+}\sqcup J^{-}$, i.e., $\uc$ produces a bipartite
orientation, we can choose a height function $\xi$, such that $\xi(J^{+})=\xi(J^{-})+1$.
Identify $\mybinom ad\in\W$ with $(a,\xi_{a}-2d)\in J_{\Z}(\xi)$.
Let $D_{q}$ denote the $J\times J$ diagonal matrix such that $(D_{q})_{aa}=z^{\xi_{a}}$.
Define $C(z):=D_{q}C_{q}D_{q}^{-1}$ and $\tC:=C(z)^{-1}=D_{q}C_{q}^{-1}D_{q}^{-1}$.
Then $C(z)_{aa}=z+z^{-1}$ and $C(z)_{ab}=C_{ab}$, $\forall a\neq b\in J$.
Decompose $\tC=\sum_{m\in\Z}\tC(m)z^{m}$. Then we can rewrite
\begin{align}
\cN(e_{(a,p)},e_{(b,s)})= & \tC_{ab}(-1+p-s)-\tC_{ba}(-1-p+s)\label{eq:N-form}\\
 & -\tC_{ab}(1+p-s)+\tC_{ba}(1-p+s).\nonumber 
\end{align}
The bilinear form $\cN$ produces a $J\times J$ matrix $\cN=\sum_{m\in\Z}\cN(m)z^{m}$
such that $\cN_{a,b}(-m)=\cN(e_{(a,0)},e_{(b,m)})$.

Finally, assume $C$ is a symmetrizable Cartan matrix of finite type.
Define the $J\times J$-matrix $\uB$ such that $\uB_{ab}=C_{ab}\cdot\symm_{b}^{-1}$,
which is symmetric. \cite[(4.1)]{jang2023quantization} used the bilinear
form $\ucN$ where $C$ is replaced by $\uB$ in the above construction.
More precisely, let $D$ denote the diagonal matrix whose diagonal
entries are $\symm_{b}$. Define $\uB(z)=C(z)D^{-1}$ and let $\utB$
denote its inverse, which is symmetric. Then we construct the bilinear
form $\ucN$ such that

\begin{align}
\ucN(e_{(a,p)},e_{(b,s)})= & \utB_{ab}(-1+p-s)-\utB_{ab}(-1-p+s)\label{eq:N-form-1}\\
 & -\utB_{ab}(1+p-s)+\utB_{ab}(1-p+s).\nonumber 
\end{align}


\newcommand{\etalchar}[1]{$^{#1}$}
\providecommand{\bysame}{\leavevmode\hbox to3em{\hrulefill}\thinspace}
\providecommand{\MR}{\relax\ifhmode\unskip\space\fi MR }
\providecommand{\MRhref}[2]{%
  \href{http://www.ams.org/mathscinet-getitem?mr=#1}{#2}
}
\providecommand{\href}[2]{#2}


\end{document}